\documentclass[leqno,11pt]{amsart}

\usepackage{amssymb, amsmath}
\usepackage{amssymb,amsfonts}
\usepackage{amsmath,latexsym}
\usepackage{pdfsync}
\usepackage{bbm}
\usepackage{xcolor}

\def\inte#1{
\displaystyle\mathop{#1\kern0pt}^\circ }

\let\pa=\partial
\let\al=\alpha

\let\d=\delta
\let\e=\varepsilon

\let\r=\rho

\let\vf=\varphi
\let\p=\psi

\let\D=\Delta

\let\wt=\widetilde

\def\cC{{\mathcal C}}

\def\cE{{\mathcal E}}

\def\ve{\varepsilon}

\def\cS{{\mathcal S}}
\def\cT{{\mathcal T}}

\def\fg{\frak{ g}}

\def\pa{\partial}
\def\grad{\nabla}
\def\Ta{{\cT}}
\def\wr{w_{\rm R}}
\def\wi{w_{\rm I}}
\def\fw{\frak{w}}

\def\h{{\rm h}}

\def\v{{\rm v}}
\def\virgp{\raise 2pt\hbox{,}}
\def\cdotpv{\raise 2pt\hbox{;}}

\def\eqdefa{\buildrel\hbox{\footnotesize def}\over =}

\def\C{\mathop{\mathbb C\kern 0pt}\nolimits}
\def\DD{\mathop{\mathbb D\kern 0pt}\nolimits}
\def\EE{\mathop{{\mathbb E \kern 0pt}}\nolimits}
\def\K{\mathop{\mathbb K\kern 0pt}\nolimits}
\def\N{\mathop{\mathbb N\kern 0pt}\nolimits}
\def\Q{\mathop{\mathbb Q\kern 0pt}\nolimits}
\def\R{\mathop{\mathbb R\kern 0pt}\nolimits}
\def\SS{\mathop{\mathbb S\kern 0pt}\nolimits}
\def\ZZ{\mathop{\mathbb Z\kern 0pt}\nolimits}
\def\TT{\mathop{\mathbb T\kern 0pt}\nolimits}
\def\P{\mathop{\mathbb P\kern 0pt}\nolimits}

\def\dive{\mathop{\rm div}\nolimits}

\def\no{\noindent}
\def\na{\nabla}
\def\p{\partial}

\newcommand{\beq}{\begin{equation}}
\newcommand{\eeq}{\end{equation}}
\newcommand{\ben}{\begin{eqnarray}}
\newcommand{\een}{\end{eqnarray}}
\newcommand{\beno}{\begin{eqnarray*}}
\newcommand{\eeno}{\end{eqnarray*}}
\newcommand{\andf}{\quad\hbox{and}\quad}
\newcommand{\with}{\quad\hbox{with}\quad}
\newtheorem{defi}{Definition}[section]
\newtheorem{thm}{Theorem}[section]
\newtheorem{lem}{Lemma}[section]
\newtheorem{rmk}{Remark}[section]

\newtheorem{prop}{Proposition}[section]
\renewcommand{\theequation}{\thesection.\arabic{equation}}

\newdimen\eqjot \eqjot = 1\jot
\def\openupeq{\openup \the\eqjot}

\def\pofbox#1 #2$#3${\setbox0=\hbox{$#3$}\ht0=0pt\dp0=0pt\wd0=0pt\hskip-#1pt\raise#2pt\box0\hskip#1pt}

\begin{document}

\title[Semiclassical limit of  GP equation with  Dirichlet boundary condition]
{Semiclassical limit of  Gross-Pitaevskii equation with  Dirichlet boundary condition}

\author[G. Gui]{Guilong Gui}
\address[G. Gui]{School of Mathematics, Northwest University, Xi'an 710069, China.} \email{glgui@amss.ac.cn}
\author[P. Zhang]{Ping Zhang} \address[P. Zhang]{Academy of
Mathematics $\&$ Systems Science and  Hua Loo-Keng Key Laboratory of
Mathematics, The Chinese Academy of Sciences, Beijing 100190, China, and School of Mathematical Sciences,
University of Chinese Academy of Sciences, Beijing 100049, China.} \email{zp@amss.ac.cn}

\begin{abstract}
In this paper, we justify the semiclassical limit of Gross-Pitaevskii equation with Dirichlet boundary condition on the 3-D upper space under the assumption that the leading order terms to both initial amplitude and initial phase function are sufficiently small in some high enough Sobolev norms. We remark that the main difficulty of the proof lies in the fact that the boundary layer appears in the leading order terms of the amplitude functions and the gradient of the phase functions to the WKB expansions of the solutions. In particular, we partially solved the open question proposed in \cite{CR2009, PNB2005} concerning the semiclassical limit of Gross-Pitaevskii equation with Dirichlet boundary condition.
\end{abstract}
\date{\today}

\maketitle


\vskip 0.2cm


%%%%%%%%%%%%%%%%%%%%%%%%%%%%%%%%%%%%%%%%%%%%%%%%%%%%%%%%%%%%
\renewcommand{\theequation}{\thesection.\arabic{equation}}
\setcounter{equation}{0}
%%%%%%%%%%%%%%%%%%%%%%%%%%%%%%%%%%%%%%%%%%%%%%%%%%%%%%%%%%%%

\section{Introduction}

We consider here the semiclassical limit of Gross-Pitaevskii equation (GP equation for short) with the Dirichlet boundary condition in the three-dimensional upper space $\R^3_{+}:$
\begin{equation}\label{NLS-0}
\begin{cases}
& i \, \varepsilon \, \partial_t \Psi^{\varepsilon}+\frac{\varepsilon^2}{2} \Delta \Psi^{\varepsilon}-(|\Psi^{\varepsilon}|^2-1)\Psi^{\varepsilon}=0, \quad (t, x) \in \mathbb{R}_{+} \times \R^3_{+},\\
&\Psi^{\varepsilon}|_{z=0}=1, \quad \Psi^{\varepsilon}|_{t=0}=a^{\varepsilon}_{0}\, \exp\left(i\frac{\varphi_{0}^{\varepsilon}}{\varepsilon}\right),
\end{cases}
\end{equation}
where $x=(y,z)\in\R^2_\h\times\R_+,$ $a^{\varepsilon}_{0} \geq 0$ and $\varphi_{0}^{\varepsilon}$ are real-valued functions. We assume that
\beq \label{S1eq1}
\begin{split}
a_0^\e=&a_{0,0}^{{\rm in}}+\sum_{j=0}^{m+2} \e^ja_{j,0}^{{\rm in}}+\e^{m+2}R_{a,0}^{{\e}} \with \lim_{\e\to 0}\|R_{a,0}^{{\e}}\|_{H^{s_0-2m-5}}=0,\\
\vf_0^\e=&\vf_{0,0}^{{\rm in}}+\sum_{j=0}^{m+2} \e^j\vf_{j,0}^{{\rm in}}+\e^{m+2}R_{\vf,0}^{{\e}} \with \lim_{\e\to 0}\|\na R_{\vf,0}^{{\e}}\|_{H^{s_0-2m-5}}=0,
\end{split}
\eeq for some $s_0$ large enough.
 We also impose the following condition at infinity:
\begin{equation}\label{infinity-c-1}
\Psi^{\varepsilon}(t, x) \to e^{\frac{i}{\ve}\bigl(u^{\infty}\cdot x-\frac{t}2|u^{\infty}|^2\bigr)} \quad\mbox{as}\quad
 |x|\rightarrow +\infty.
\end{equation}
In what follows, we assume that the constant vector $u^{\infty}=0$ for simplicity.

The motivation for us to study the problem \eqref{NLS-0} comes from many interesting issues concerning a superfluid passing an obstacle (see
for example \cite{FPR92, JP01, PNB2005}). Classical
 Madelung transform introduces two real variables:  $a^{\varepsilon} \geq 0$ and $\varphi^{\varepsilon}$  so that
\begin{equation}\label{m-t}
\Psi^{\varepsilon} = a^{\varepsilon}\, \exp\Bigl(i\frac{\varphi^{\varepsilon}}{\varepsilon} \Bigr).
\end{equation}
By substituting \eqref{m-t} into \eqref{NLS-0} and separating the real and imaginary parts, we find
\begin{equation}\label{NLS-1}
\begin{cases}
&\partial_t a^{\varepsilon}+\nabla \varphi^{\varepsilon}\cdot \nabla a^{\varepsilon}+\frac{1}{2}a^{\varepsilon}\Delta \varphi^{\varepsilon}=0,
\quad (t, x) \in \mathbb{R}_{+} \times \R^3_{+},\\
&a^{\varepsilon}\bigl(\partial_t \varphi^{\varepsilon}+\frac{1}{2}|\nabla \varphi^{\varepsilon}|^2+(a^{\varepsilon})^2-1\bigr)=\frac{\varepsilon^2}{2}\Delta a^{\varepsilon},
\end{cases}
\end{equation}
with the initial-boundary conditions
\begin{equation}\label{NLS-1-ibc}
a^{\varepsilon}|_{z=0}=1, \ \  \varphi^{\varepsilon}|_{z=0}=0,\ \   a^\e\to 1\ \ \mbox{as}\ \ |x|\to \infty \andf a^{\varepsilon}|_{t=0}=a^{\varepsilon}_{0}, \ \  \varphi^{\varepsilon}|_{t=0}=\varphi^{\varepsilon}_{0}.
\end{equation}

 We denote $\rho^{\varepsilon}\eqdefa (a^{\varepsilon})^2$  and  $\r^\e u^{\varepsilon}\eqdefa \r^\e\nabla\varphi^{\varepsilon}$, which
  corresponds to the quantum density and momentum respectively in quantum mechanics (see \cite{LL85} for instance). This allows to rewrite (\ref{NLS-1}-\ref{NLS-1-ibc}) as the following hydrodynamical system:
\begin{equation}\label{hydro-system-1}
\begin{cases}
&\partial_t \rho^{\varepsilon}+\nabla \cdot(\rho^{\varepsilon}u^{\varepsilon})=0, \\
&\rho^{\varepsilon}\bigl(\partial_t u^{\varepsilon}+u^{\varepsilon}\cdot\nabla u^{\varepsilon}\bigr)+\nabla \, p(\rho^{\varepsilon})=\frac{\varepsilon^2}{2}\rho^{\varepsilon}\nabla\left(\frac{\Delta \sqrt{\rho^{\varepsilon}}}{\sqrt{\rho^{\varepsilon}}}\right),
\end{cases}
\end{equation}
with the initial-boundary conditions
\begin{equation}\label{hydro-system-1-ibc}
\begin{cases}
&\rho^{\varepsilon}|_{z=0}=1, \quad \varphi^{\varepsilon}|_{z=0}=0 \andf \r^\e\to 1\ \ \mbox{as}\ \ |x|\to \infty \\
&\rho^{\varepsilon}|_{t=0}=(a^{\varepsilon}_{0})^2, \quad u^{\varepsilon}|_{t=0}=\nabla\varphi^{\varepsilon}_{0},
\end{cases}
\end{equation}
and the pressure law  $p(\rho^{\varepsilon})=\frac{1}{2}(\rho^{\varepsilon})^2.$

The system \eqref{hydro-system-1} is called quantum compressible Euler system,
and the right hand-side of the $u^\e$ equation in \eqref{hydro-system-1} is called quantum pressure.
As $\varepsilon$ approaches to $0$, the quantum pressure is formally negligible and the system (\ref{hydro-system-1}-\ref{hydro-system-1-ibc})
 approaches to the classical compressible Euler equation
\begin{equation}\label{euler-system-1}
\begin{cases}
&\partial_t \rho +\nabla \cdot(\rho\,u )=0,\\
&\rho \left(\partial_t u +u\cdot\nabla u \right)+\nabla \, p(\r)=0,
\end{cases}
\end{equation}
with the initial-boundary conditions
\begin{equation}\label{euler-system-1-ibc}
 \int_0^{+\infty}u \cdot {\bf n}\,dz=0,\ \ \r\to 1\ \ \mbox{as}\ \ |x|\to \infty \andf \rho|_{t=0}=(a_{0,0}^{\rm{in}})^2, \quad u|_{t=0}=\nabla\varphi_{0,0}^{\rm{in}}.
\end{equation}

The justification of the above formal limit has  attracted the interests by many authors. In the whole space case,
G\'erard \cite{Gerard92} proved the limit with analytical initial data. Grenier \cite{Grenier98} solved the limit
problem before the formation of singularity in the limit system  with initial data in Sobolev spaces. The main
idea in \cite{Grenier98} is to use the symmetrizer of the limit system \eqref{euler-system-1} to get $H^s$ energy
estimates which are uniform in $\e$ for a singularly perturbed system. Nevertheless  this method does not work for the semiclassical limit
of  Schr\"odinger-Poisson equations, as the resulting limit
system is not a symmetric hyperbolic one.
Motivated by the work of Brenier \cite{Brenier00}, where the author proved the local-in time
convergence of the scaled Vlasov-Poisson equations to the incompressible Euler
equations, the second author \cite{Z1} used the Wigner measure approach (see \cite{LP93, Z08}) to study the semiclassical
limit of Schr\"odinger-Poisson equation (see  \cite{Z2} for more general nonlinearity).

In order to solve the semiclassical limit of GP equation in the exterior domain with Neumann boundary condition (which corresponds to
the non-slip boundary condition $u\cdot {\bf n}=0$ for the limit system \eqref{euler-system-1}), where we can not
use Wigner transform, the authors \cite{lz3} simplified
the modulated energy functional in \cite{Z1,Z2} and proved that
\beno
|\Psi^\e|^2-\r \to 0 \quad\mbox{in}\quad L^\infty(]0,T[; L^2) \andf \e{\rm Im}\left(\bar{\Psi}^\e\na\Psi^\e\right)-\r u \to 0
\quad\mbox{in}\quad L^\infty(]0,T[; L^1_{\mbox{loc}}),
\eeno
before the formation of singularity in the limit system. This idea has been used and extended by the authors in \cite{Alazard09a,Serfaty17}.
Interested readers may  check \cite{Alazard09b, CR2009}
for the so-called supercritical geometric optics where they allow $p'(0)=0$ for the pressure function in \eqref{euler-system-1}.
One may also  check the books \cite{C08,Z08} and references therein for more information in this context.

For the problem \eqref{NLS-0}, by comparing  \eqref{hydro-system-1-ibc} with \eqref{euler-system-1-ibc}, we find that the boundary condition $\rho^{\varepsilon}|_{z=0}=1$ in \eqref{hydro-system-1-ibc} does not match the boundary condition for $\rho$ in \eqref{euler-system-1-ibc} at the boundary $\{z=0\},$
where we do not have any  restriction on $\r.$ This results in a strong boundary layer near the boundary $\{z=0\}$.
 In fact, if we formally seek for WKB expansions $\Psi^{\varepsilon} = a^{\varepsilon}\, \exp\left(i\frac{\varphi^{\varepsilon}}{\varepsilon} \right)$ of the form
\begin{equation}\label{WKB-expansion-1}
\begin{cases}
& a^{\varepsilon}(t, x)=\sum_{k=0}^{\infty}\varepsilon^{k}\left(a_{k}(t, y, z)+A_{k}(t, y, Z)\right), \\
& \varphi^{\varepsilon}(t, x)=\sum_{k=0}^{\infty}\varepsilon^{k}\left(\varphi_{k}(t, y, z)+\Phi_{k}(t, y, Z)\right),
\end{cases}
\end{equation}
we shall find below that $\Phi_{0}\equiv 0$, and $a_0$, $a_{k}$, $A_{k}$, $\varphi_{k}$, $\Phi_{k}$ with $k=1,2,3,...$ are non-trivial. In this case, we have $\nabla\,u^{\varepsilon}=\na^2\vf^\e \thicksim O(\frac{1}{\varepsilon})$ and $\nabla\,a^{\varepsilon} \thicksim O(\frac{1}{\varepsilon}).$ In some sense, this phenomena has some similarity with the strong boundary layer caused by vanishing viscosity of incompressible Navier-Stokes system
to Euler system (see \cite{OS99}). Lately there are a lot of progresses on this topic (see for instance \cite{GMM18,GN17} and the references therein).

 On the other hand, for the case of the semiclassical limit of  GP equation with the Neumann boundary condition in $\R^3_+,$ that is, $\partial_z\Psi^{\varepsilon}|_{\partial \R^3_{+}}=0$, Chiron and Rousset \cite{CR2009} justified the validity of the WKB expansions on
 some finite time interval $[0,T].$ We remark that
 in this case,  the boundary layer profiles $A_0=\Phi_0=\Phi_1\equiv 0$ in \eqref{WKB-expansion-1}, which implies $\nabla\,u^{\varepsilon} \thicksim O(1)$ and $\nabla\,a^{\varepsilon} \thicksim O(1)$. This weak boundary layer plays a key role in the study of the  nonlinear stability to the WKB expansions.
 Nevertheless, the semiclassical limit of nonlinear Schr\"odinger equation with Dirichlet boundary condition was left open in \cite{CR2009, PNB2005}.

In this paper, we are going to answer this question proposed in  \cite{CR2009, PNB2005} under the condition that both $a_{0,0}^{\rm{in}}-1$ and $\na\vf_0$ are sufficiently small
in some regular enough Sobolev space.

%As mentioned above, since  $\nabla\,u^{\varepsilon} \thicksim O(\frac{1}{\varepsilon})$ and $\nabla\,a^{\varepsilon} \thicksim O(\frac{1}{\varepsilon})$, there is no uniform bound of the quantity $\|\nabla\,u^{\varepsilon}\|_{L^\infty}$, which will prevent the propagation of the high regularity of the velocity and density if we see the $\rho$ equation as a transport equation. To overcome this difficulty, we consider separately remainders of the amplitude $a^{\varepsilon}$ and the argument $\varphi^{\varepsilon}$, which results in two similar nonlinear wave equations (\eqref{w-phi-eqns} below, like \eqref{S3eq4} in Section 3). This couple system rescues the $L^2$ norm of the perturbation of the amplitude, which, together with two key lemmas (Lemmas \ref{lem-est-singu-1}, \ref{lem-est-conv-2}), gains a $\varepsilon$ factor to make up the singularity of $\|\nabla\,u^{\varepsilon}\|_{L^\infty}$ and $\|\nabla\,a^{\varepsilon}\|_{L^\infty}$.

\medbreak  Let us end this introduction by some notations that will be used in all that follows.

For operators $A,B,$ we denote $[A;B]=AB-BA$ to be the  commutator of $A$ and $B.$ For~$a\lesssim b$, we mean that there is a uniform constant $C,$ which may be different on different lines, such that $a\leq Cb$.  We denote by $\int_{\R^3}f|g\,dx$ the $L^2(\R^3_+)$ inner product of $f$ and $g,$ and $L^p(\R^3_+)$ by $L^p_+.$ Finally we shall always denote $\na_y$ by $\na_\h$ and  ${\cT}\eqdefa\,(\partial_t,\, \nabla_{\h}).$

%%%%%%%%%%%%%%%%%%%%%%%%%%%%%%%%%%%%%%%%%%%%%%%%%%%%%%%%%%%%
\renewcommand{\theequation}{\thesection.\arabic{equation}}
\setcounter{equation}{0}
%%%%%%%%%%%%%%%%%%%%%%%%%%%%%%%%%%%%%%%%%%%%%%%%%%%%%%%%%%%%

\section{Formal asymptotic analysis}\label{Sect2}

\subsection{Outer expansion}\

Since the boundary layer is  concentrated in the $\e-$ neighborhood of $\{z =0\}$, we call the domain $\frak{O}\eqdefa
\bigl\{\ x=(y, z):\, y \in \mathbb{R}^2, \, z>\e\ \bigr\}$ the outer region, and the associated vertical variable $z$ is called  outer variable. We also denote the inner region by $\frak{I}\eqdefa\left\{\ x=(y, z):\, y \in \mathbb{R}^2, \, 0 \leq z < \varepsilon\ \right\}$, and call $Z=\frac{z}{\varepsilon}$ the inner variable, which makes us to specify the so-called ``inner limit process".

 In the outer region $\frak{O},$ we formally seek  the solution $(a^{\varepsilon}, \varphi^{\varepsilon})$ of \eqref{NLS-1} with the form:
\begin{equation}\label{outer-expansion-1}
\begin{split}
a^{\varepsilon}(t, x)=\sum_{k=0}^{\infty}\varepsilon^{k}a_{k}(t, y, z) \andf
 \varphi^{\varepsilon}(t, x)=\sum_{k=0}^{\infty}\varepsilon^{k}\varphi_{k}(t, y, z).
\end{split}
\end{equation}
By substituting \eqref{outer-expansion-1} into \eqref{NLS-1}, we get
\begin{equation}\label{outer-conti}
\begin{split}
\sum_{k=0}^{\infty}\varepsilon^{k}\partial_t a_{k}+\sum_{k_1, k_2=0}^{\infty}\varepsilon^{k_1+k_2}\bigl(\nabla \varphi_{k_1}\cdot \nabla a_{k_2}+\frac{1}{2}a_{k_1}\Delta \varphi_{k_2}\bigr)=0
\end{split}
\end{equation}
and
\begin{equation}\label{outer-bern}
\begin{split}
\sum_{k_1, k_2=0}^{\infty}\varepsilon^{k_1+k_2}a_{k_1}\partial_t \varphi_{k_2}+\sum_{k_1, k_2,k_3=0}^{\infty}\varepsilon^{k_1+k_2+k_3}&\frac{a_{k_1}}{2}\bigl(\nabla \varphi_{k_2} \cdot \nabla \varphi_{k_3}+2a_{k_2}a_{k_3}\bigr)\\
&\qquad-\sum_{k=0}^{\infty}\varepsilon^{k}a_{k}=\frac{1}{2}\sum_{k=0}^{\infty}\varepsilon^{k+2}\Delta a_{k}.
\end{split}
\end{equation}
Vanishing the coefficients to the zeroth order  of $\e$ in \eqref{outer-conti} and \eqref{outer-bern} gives
\begin{equation}\label{S2eq4}
\begin{cases}
&\partial_t a_{0}+\nabla \varphi_{0}\cdot \nabla a_{0}+\frac{1}{2}a_{0}\Delta \varphi_{0}=0,\quad (t, x) \in \mathbb{R}_{+} \times \R^3_{+},\\
&\partial_t \varphi_{0}+\frac{1}{2} |\nabla \varphi_{0}|^2 +(a_{0}^2-1)=0,
\end{cases}
\end{equation}
where we used the fact that $a_0$ has a positive lower bound which will be justified in Section \ref{sec-hj}.

In view of \eqref{S1eq1} and \eqref{NLS-1-ibc}, we implement the system \eqref{S2eq4} with the initial and  boundary conditions:
\begin{equation}\label{outer-order-0-bc}
a_0|_{t=0}=a_{0,0}^{\rm{in}}, \quad  \vf_0|_{t=0}=\vf_{0,0}^{\rm{in}} \andf \varphi_{0}|_{z=0}=0,\quad a_0\to 1\quad\mbox{as}\quad |x|\to\infty.
\end{equation}
We shall prove the local well-posedness of the above problem with sufficiently smooth initial data in Section \ref{sec-hj}.

Vanishing the coefficients of  $\varepsilon^{1}$ in \eqref{outer-conti} and \eqref{outer-bern} leads to the coupled system
 $(a_1, \, \varphi_1):$
\begin{equation}\label{outer-order-1}
\begin{cases}
&\partial_t a_{1}+\nabla a_{0}\cdot   \nabla\varphi_{1}+\nabla \varphi_{0}\cdot \nabla a_{1}+\frac{1}{2}a_{1}\Delta \varphi_{0}+\frac{1}{2}a_{0}\Delta \varphi_{1}=0, \quad (t, x) \in \mathbb{R}_{+} \times \R^3_{+},\\
& \partial_t \varphi_{1}+ \nabla \varphi_{0} \cdot \nabla \varphi_{1}+2a_0 a_{1}=0,\\
&a_1|_{t=0}=a_{1,0}^{\rm{in}}, \quad  \vf_1|_{t=0}=\vf_{1,0}^{\rm{in}},
\end{cases}
\end{equation}
where $(a_{1,0}^{\rm{in}},\vf_{1,0}^{\rm{in}})$ is given by \eqref{S1eq1}.

In general, by vanishing the coefficients of  $\varepsilon^{k+2}$ in \eqref{outer-conti}  and  \eqref{outer-bern}  for $k=0,\cdots, m,$  and
using  \eqref{S1eq1}, we find
\begin{equation}\label{outer-order-m}
\begin{cases}
&\partial_t a_{k+2}+\nabla a_{0}\cdot\nabla \varphi_{k+2} +\nabla \varphi_{0}\cdot \nabla a_{k+2}+\frac{\Delta \varphi_{0}}{2}a_{k+2}+\frac{a_{0}}{2}\Delta \varphi_{k+2}=f_{k+1}^a\\
& \partial_t \varphi_{k+2}+ \nabla \varphi_{0} \cdot \nabla \varphi_{k+2}+2a_0 a_{k+2}=\frac{1}{2a_0}\left(\Delta a_{k}+g_{k+1}^{\varphi}\right),\\
&a_{k+2}|_{t=0}=a_{k+2,0}^{\rm{in}}, \quad  \vf_{k+2}|_{t=0}=\vf_{k+2,0}^{\rm{in}},
\end{cases}
\end{equation}
where the source terms $(f_{k+1}^a, \, g_{k+1}^{\varphi})$ are determined by
\begin{equation}\label{remainder-inner-1}
\begin{split}
&f_{k+1}^a\eqdefa-\sum_{k_1=1}^{k+1}\bigl(\nabla \varphi_{k_1}\cdot \nabla a_{k+2-k_1}+\frac{1}{2}a_{k_1}\Delta \varphi_{k+2-k_1}\bigr),\\
&g_{k+1}^{\varphi}\eqdefa-\sum_{k_1=1}^{k+1}a_{k_1}\partial_t \varphi_{k+2-k_1}
-\sum_{\substack{
k_1+k_2+k_3=k+2\\
0 \leq k_1,\,k_2,\,k_3\leq k+1}}
 \frac{a_{k_1}}{2}\bigl(\nabla \varphi_{k_2} \cdot \nabla \varphi_{k_3}+2a_{k_2}a_{k_3}\bigr).
\end{split}
\end{equation}
We shall implement
the above systems with boundary conditions  in Subsection \ref{Sub2.2}.

\subsection{Uniformly valid approximation}\label{Sub2.2}

In  all that follows, we shall always denote
\beq \label{S6eq0}
[f]_\e(x)\eqdefa f\bigl(y, \frac{z}\e\bigr).
\eeq

With the  outer solution $(a_j, \, \varphi_j)$ for $j=0, 1, 2, ...$ in hand, we shall use the Successive Complementary Expansion Method (SCEM for short, see \cite{CM07}) to seek a Uniformly Valid Approximate solutions (UVA for short)  to \eqref{NLS-1}.
In order to do so, we take the following ansatz
\begin{equation}\label{uniform-expansion-1}
\begin{split}
& a^{\varepsilon}(t, x)=\sum_{k=0}^{\infty}\varepsilon^{k}\left(a_{k}(t, x)+[A_{k}]_\e(t, x)\right) \andf\\
& \varphi^{\varepsilon}(t, x)=\sum_{k=0}^{\infty}\varepsilon^{k}\left(\varphi_{k}(t,x)+[\Phi_{k}]_\e(t,x)\right).
\end{split}
\end{equation}
We require that both $A_{k}(t, y, Z)$ and $\Phi_{k}(t, y, Z)$ together with  all of their derivatives are rapidly vanishing as $Z \rightarrow +\infty$.

We denote $\overline{g}(t,y)$ the trace of $g(t, y, z)$ on the boundary   $\{z=0\},$ that is,
\begin{equation}\label{bdyvalue-1}
\overline{g}(t, y)\eqdefa g(t, y, z=0).
\end{equation}
Formally let us write that
\beq \label{S2eq8}
g(t,y,z)=\sum_{j=0}^\infty \frac{z^j}{j!}\overline{\p_z^jg}(t,y).
\eeq
By plugging the ansatz \eqref{uniform-expansion-1} into \eqref{NLS-1} and using  \eqref{outer-conti}, \eqref{outer-bern} and
\eqref{S2eq8}, we obtain
\begin{equation}\label{u-conti}
\begin{split}
&\sum_{k=0}^{\infty}\varepsilon^{k}\partial_t A_{k}+\sum_{k_1, k_2=0}^{\infty}\varepsilon^{k_1+k_2}\Bigl(\bigl(\nabla_{\rm h} \Phi_{k_1}\cdot \nabla_{\rm h} A_{k_2}+\frac{1}{2}A_{k_1}\Delta_{\h} \Phi_{k_2}\bigr)\\
&\qquad\qquad\qquad\qquad\qquad\qquad\qquad\qquad+\varepsilon^{-2}\bigl(\partial_Z \Phi_{k_1}\partial_Z A_{k_2}+\frac{1}{2}A_{k_1}\partial_Z^2 \Phi_{k_2}\bigr)\Bigr)\\
&+\sum_{k_1, k_2, j=0}^{\infty}\varepsilon^{k_1+k_2+j}\frac{Z^{j}}{j!}
\Bigl(\overline{\nabla_y\partial_{z}^{j}\varphi_{k_1}}\cdot \nabla_{\h} A_{k_2}+\nabla_{\h} \Phi_{k_1}\cdot \overline{\nabla_{\h} \partial_{z}^{j} a_{k_2}}+\frac{\overline{\partial_{z}^{j}a_{k_1}}}{2} \Delta_{\h} \Phi_{k_2}\\
&\qquad+\frac{A_{k_1}}{2}\overline{\Delta \partial_{z}^{j} \varphi_{k_2}}+
\varepsilon^{-1}\bigl(\overline{\partial_{z}^{j+1}\varphi_{k_1}}\partial_Z A_{k_2}+\partial_{Z}\Phi_{k_1} \overline{\partial_{z}^{j+1} a_{k_2}}\bigr)+\varepsilon^{-2}\frac{\overline{\partial_{z}^{j}a_{k_1}}}{2}\partial_Z^2 \Phi_{k_2}\Bigr)=0,
\end{split}
\end{equation}
and
\begin{equation}\label{u-bern}
\begin{split}
&\sum_{k_1, k_2, j=0}^{\infty}\varepsilon^{k_1+k_2+j}\frac{Z^{j}}{j!}\bigl(\overline{\partial_z^ja_{k_1}}\partial_t \Phi_{k_2}+A_{k_1}\overline{\partial_t\partial_z^j \varphi_{k_2}}\bigr)+\sum_{k_1, k_2=0}^{\infty}\varepsilon^{k_1+k_2}(A_{k_1}\partial_t \Phi_{k_2})\\
&-\sum_{k=0}^{\infty}\varepsilon^{k}A_{k}+\sum_{k_1, k_2,k_3, j_1, j_2=0}^{\infty}\varepsilon^{k_1+k_2+k_3+ j_1+j_2}\frac{Z^{j_1+j_2}}{j_1!j_2!}\Bigl(\overline{\partial_{z}^{j_1}a_{k_1}}  \overline{\nabla_{\h} \partial_{z}^{j_2}\varphi_{k_2}} \cdot \nabla_{\h} \Phi_{k_3}+\\
&\qquad\qquad+\frac{A_{k_1}}{2} \overline{\nabla \partial_{z}^{j_1}\varphi_{k_2}} \cdot \overline{\nabla\partial_{z}^{j_2} \varphi_{k_3}}+3\overline{\partial_{z}^{j_1}a_{k_1}}\overline{\partial_{z}^{j_2}a_{k_2}}A_{k_3}
+\varepsilon^{-1}\,\overline{\partial_{z}^{j_1}a_{k_1}}
\overline{\partial_{z}^{j_2+1}\varphi_{k_2}}\partial_Z\Phi_{k_3}\,\Bigr) \\
&+\sum_{k_1, k_2,k_3, j_1=0}^{\infty}\varepsilon^{k_1+k_2+k_3+j_1}\frac{Z^{j_1}}{j_1!}
\Bigl( A_{k_1}  \overline{\nabla_{\h}\partial_{z}^{j_1}\varphi_{k_2}} \cdot \nabla_{\h}\Phi_{k_3} +\frac{\overline{\partial_{z}^{j_1}a_{k_1}}}{2}\nabla_{\h}\Phi_{k_2} \cdot \nabla_{\h}\Phi_{k_3}+\\
&\qquad
+3\overline{\partial_{z}^{j_1}a_{k_1}}A_{k_2}A_{k_3}+
\varepsilon^{-1} A_{k_1}  \overline{\partial_{z}^{j_1+1}\varphi_{k_2}}\partial_Z\Phi_{k_3}
+\varepsilon^{-2}\frac{\overline{\partial_{z}^{j_1}a_{k_1}}}{2}
\partial_{Z}\Phi_{k_2}\partial_Z\Phi_{k_3}\Bigr)\\
&+\sum_{k_1, k_2,k_3=0}^{\infty}\varepsilon^{k_1+k_2+k_3}\frac{A_{k_1}}{2}\Bigl(\nabla_{\h}\Phi_{k_2} \cdot \nabla_{\h}\Phi_{k_3}+2A_{k_2}A_{k_3}+\varepsilon^{-2}
\partial_{Z}\Phi_{k_2}\partial_Z\Phi_{k_3}\Bigr)\\
&=\frac{1}{2}\sum_{k=0}^{\infty}\varepsilon^{k}\partial_Z^2 A_{k}+\frac{1}{2}\sum_{k=0}^{\infty}\varepsilon^{k+2}\Delta_{\h} A_{k}\with \D_\h=\p_{y_1}^2+\p_{y_2}^2.
\end{split}
\end{equation}
The coefficients  of  $\varepsilon^{-2}$ in \eqref{u-conti} and \eqref{u-bern} yields
\begin{equation}\label{a-negative-2}
\begin{split}
\frac{1}{2}(A_{0}+\overline{a_{0}})\partial_Z^2 \Phi_{0}+\partial_Z \Phi_{0}\partial_Z A_{0}=0,
\end{split}
\end{equation}
and
\begin{equation}\label{phi-negative-2}
\begin{split}
\frac{1}{2}(A_0+\overline{a}_0)|\partial_Z \Phi_{0}|^2=0,
\end{split}
\end{equation}
respectively.

We first assume that $A_0+\overline{a}_0$ has a positive lower bound, which we shall justify in Section \ref{sec-blp}. Then due to $\Phi_{0}|_{Z=\infty}=0$, we deduce from \eqref{phi-negative-2}  that
\begin{equation}\label{phi-negative-2-a}
\begin{split}
\Phi_{0}(t, y, Z)\equiv 0.
\end{split}
\end{equation}

Similarly by virtue of \eqref{S2eq4} and \eqref{phi-negative-2-a},  we find that the coefficients   of $\varepsilon^{-1}$ in \eqref{u-conti}  and  of $\varepsilon^{0}$ in \eqref{u-bern} give respectively
\begin{equation}\label{phi-negative-1}
\begin{split}
\frac{1}{2}(A_0+\overline{a}_0)\partial_Z^2 \Phi_{1}+(\partial_Z \Phi_{1}+\overline{\partial_z \varphi}_{0}) \partial_Z A_0=0,
\end{split}
\end{equation}
and
\begin{equation}\label{a-order-0}
\begin{split}
\frac{1}{2}\partial_Z^2A_0=(A_0+\overline{a}_0)\partial_Z\Phi_{1}\overline{\partial_z \varphi}_{0}+\frac{\overline{a}_0+A_0}{2} (\partial_Z \Phi_1)^2+A_0^3+3 \overline{a}_0 A_0^2+2 \overline{a}_0^2 A_0.
\end{split}
\end{equation}

According to the  boundary condition $
a^{\varepsilon}|_{z=0}=1$ in \eqref{NLS-1-ibc}, we have
the matched condition on the boundary $\{z=0 \}$ that $A_0(Z=0)+a_0(z=0)=1.$ So that we impose the following
boundary condition for $A_0:$
\begin{equation}\label{a-order-0-bc}
A_0|_{Z=0}=1-a_0(t,y,0).
\end{equation}
 Furthermore, we require that both $A_{0}(t, y, Z)$ and $\Phi_{1}(t, y, Z)$ along with  all of their derivatives are rapidly vanishing as $Z \rightarrow +\infty.$   We shall present the unique solvability of
  the required solution to the system (\ref{phi-negative-1}-\ref{a-order-0})   in Section \ref{sec-blp}.

Notice the boundary condition $
 \varphi^{\varepsilon}|_{z=0}=0$ in \eqref{NLS-1-ibc} and $
\varphi_{0}|_{z=0}=0$ in \eqref{outer-order-0-bc},
we have the following matched condition of $\vf_1$ on the boundary $\{z=0 \}$
\begin{equation}\label{varphi-order-0-bc}
\varphi_1|_{z=0}=-\Phi_1(t,y,0).
\end{equation}

We implement the system \eqref{outer-order-1}
with the Dirichlet boundary condition \eqref{varphi-order-0-bc}, and we shall prove its unique solvability  in Section \ref{sec-asy}.

Inductively, assuming that we already obtain $\left(a_0,  \varphi_0\right),$  $\left(a_{j+1},\, \varphi_{j+1}\right)$ and
$\left(A_{j}, \Phi_{j+1}\right)$ with $0 \leq j \leq k \leq m-1,$ we get, by
vanishing of the coefficients of $\varepsilon^{k}$ in \eqref{u-conti} and of  $\varepsilon^{k+1}$ in \eqref{u-bern}, that
\begin{equation}\label{phi-order-m}
\begin{split}
\frac{\overline{a}_0+A_0}{2}\partial_Z^2 \Phi_{k+2}+\partial_Z A_0 \partial_Z \Phi_{k+2} +(\partial_Z \Phi_{1}+\overline{\partial_z \varphi}_{0}) \partial_Z A_{k+1}+\frac{A_{k+1}}{2} \partial_Z^2 \Phi_{1} =F_k,
\end{split}
\end{equation}
and
\begin{equation}\label{a-order-m1}
\begin{split}
\frac{1}{2}\partial_Z^2A_{k+1}=&(A_0+\overline{a}_0)(\partial_Z\Phi_{1}+\overline{\partial_z \varphi}_{0}) \partial_Z \Phi_{k+2}\\
&+(3A_0^2+6 \overline{a}_0 A_0+2 \overline{a}_0^2+\partial_Z\Phi_{1}\overline{\partial_z \varphi}_{0}+\frac{1}{2}|\partial_Z\Phi_{1}|^2) A_{k+1}+G_k,
\end{split}
\end{equation}
where the source terms $(F_k,\, G_k)$ depend only on $\left(\overline{\p_z^\ell a_0}, \overline{\p_z^\ell\varphi_0}\right)$ for $\ell\leq k+2,$ and   $\left(\overline{\p_z^\ell a_{j+1}}, \overline{\p_z^k\varphi_{\ell+1}}\right)$ and $\left(A_{j},  \na\Phi_{j+1}\right)$ for $0 \leq j \leq k$
and $0\leq \ell+j\leq k+1$, the explicit form of  which will be presented in the Appendix \ref{appA-re}.

Thanks to the matched boundary condition on  $\{z=0 \}$ for the outer and inner solutions,  we impose the following condition for $A_{k+1}:$
\begin{equation}\label{a-order-m1-bc}
%\begin{cases}
A_{k+1}|_{Z=0}=-a_{k+1}(t,y,0).
% \quad A_{m+1}|_{Z=+\infty}=0, \\ &\Phi_{m+2}|_{Z=+\infty}=0, \quad \partial_Z\Phi_{m+2}|_{Z=+\infty}=0,\end{cases}
\end{equation}
We also require that both $A_{k+1}(t, y, Z)$ and $\Phi_{k+2}(t, y, Z)$ together with  all of their derivatives are rapidly vanishing as $Z \rightarrow +\infty.$

Finally, with thus obtained $\Phi_{k+2}$, according to  the matched boundary condition on  $\{z=0 \}$ for the outer and inner solutions, we implement the system \eqref{outer-order-m} with the Dirichlet boundary condition
\beq \label{S2eq23}
\varphi_{k+2}|_{z=0}=-\Phi_{k+2}(t,y,0). \eeq

The unique solvability of the systems (\ref{phi-order-m}-\ref{a-order-m1}) and \eqref{outer-order-m} with the boundary condition \eqref{S2eq23} will be outlined in Section \ref{sec-asy}.

%%%%%%%%%%%%%%%%%%%%%%%%%%%%%%%%%%%%%%%%%%%%%%%%%%%%%%%%%%%%
\renewcommand{\theequation}{\thesection.\arabic{equation}}
\setcounter{equation}{0}
%%%%%%%%%%%%%%%%%%%%%%%%%%%%%%%%%%%%%%%%%%%%%%%%%%%%%%%%%%%%

\section{The main result and its sketch of the proof} We first  observe from \eqref{S2eq4} that
\beq \label{vf10}
\begin{split}
\partial_t \varphi_{0}|_{t=0}=&-\left(\frac{1}{2} |\nabla \varphi_{0}|^2 +(\r_{0}-1)\right)|_{t=0}\\
=&1-\left(a_{0,0}^{\rm{in}}\right)^2-\frac12|\na\vf_{0,0}^{\rm{in}}|^2\eqdefa \vf_{0,1}^{\rm{in}}.
\end{split}
\eeq

To guarantee the local existence of smooth solutions to (\ref{S2eq4}-\ref{outer-order-0-bc}), we need
the following compatibility conditions for the initial data:

\no ($\mathcal{A}_0$):  Let $(\vf_{0,0}^{\rm{in}}, \, \vf_{0,1}^{\rm{in}}) \in H^{s_0} \times H^{s_0-1}$ for  $4\leq s_0\in\N.$ We
assume that the data satisfies $\partial_t^j\varphi_0(0,y,0)=0$ for $y\in\R^2$ and $j=0,\cdots, s_0-1.$

\begin{defi}\label{S2def1}
Let $s, T>0,$ we define the functional space $W^{s}_{T}\eqdefa \bigcap_{j=0}^{[s]}C^{j}([0, T]; \, H^{s-j}(\R^3_{+})),$
where $[s]$ denotes the integer part of $s$ and its norm is given by
\begin{equation}\label{outer-order-0-17-b}
\begin{split}
\|\varphi(t)\|_{W^{s}}^2
\eqdefa \sum_{j=0}^{[s]}\|\partial_t^j&\varphi(t)\|_{H^{s-j}}^2 \andf \|\varphi\|_{W^{s}_T}\eqdefa \sup_{t\in [0,T]}\|\varphi(t)\|_{W^{s}}.
\end{split}
\end{equation}
\end{defi}

We shall prove in Section \ref{sec-hj} that

\begin{prop}\label{S2prop1}
{\sl Let $4\leq s_0\in\N $ and $a_{0,0}^{\rm{in}}-1\in H^{s_0-1},$  $\vf_{0,0}^{\rm{in}} \in H^{s_0}$ which satisfies the compatibility condition $(\mathcal{A}_0)$. We assume that
\begin{equation}\label{S3eq27}
\bigl\|\bigl(a_{0,0}^{\rm{in}}-1,\na\vf_{0,0}^{\rm{in}}\bigr)\bigr\|_{H^{s_0-1}}\leq c,
\end{equation} for some sufficiently small positive constant $c$, then
there exists a positive constant ${\cC}$ so that for $T_0\eqdefa {\cC}c^{-1},$ the system (\ref{S2eq4}-\ref{outer-order-0-bc}) has a unique solution $(a_0,\varphi_0)$ on $[0, T_0],$
which satisfies
\begin{equation}\label{S3eq28}
\left\|\left(a_0-1,\p_t\vf_0,\na\vf_0\right)\right\|_{W_{T_0}^{s_0-1}} \leq C \, \bigl\|\bigl(a_{0,0}^{\rm{in}}-1,\na\vf_{0,0}^{\rm{in}}\bigr)\bigr\|_{H^{s_0-1}}.
\end{equation}}
\end{prop}

We remark that the main difficulty in the proof of Proposition \ref{S2prop1} lies in the boundary condition
$\varphi_{0}|_{z=0}=0$ in \eqref{outer-order-0-bc} so that one can not apply the standard theory on symmetric hyperbolic system to prove its local well-posedness. The new idea here is to reformulate (\ref{S2eq4}-\ref{outer-order-0-bc}) to be an initial and boundary value problem of a nonlinear wave equation (\ref{S3eq4}-\ref{outer-order-0-initial}) (see Section \ref{sec-hj}) under the smallness initial condition \eqref{S3eq27}.

In order to solve the boundary layer equation, we recall the boundary layer profile space from \cite{CR2009}:

\begin{defi}\label{S2def2}
{\sl Let $s \in \R^+$ and $\gamma_0>0$ be a  positive constant, we  define $W^{s}_{\gamma_0}(\R^3_{+})$ as
the completion of $\left\{F(y, Z) \in H^s(\mathbb{R}_{\h}^{2};\, H^{\infty}(\mathbb{R}^+_{Z}))\, \right\}$
 with $\|F\|_{W^{s}_{\gamma_0}(\R^3_{+})}$ being finite, and
\begin{equation*}\label{space-def-1}
\mathcal{W}^{s}_{\gamma_0, T}=\mathcal{W}^{s}_{\gamma_0}([0, T]\times\R^3_{+} )\eqdefa \bigcap_{j=0}^{[s]}C^{j}([0, T]; \, W^{s-j}_{\gamma_0}(\R^3_{+})),
\end{equation*}
where the norms are given by
\beq \label{S4eq1}
\begin{split}
&\|F\|_{W^{s}_{\gamma_0}}\eqdefa \max_{0\leq\ell\leq [s]}\sup_{Z\in\R^+}\bigl(e^{\gamma_0 \,Z} \|\partial_Z^{\ell}F(\cdot,Z)\|_{H^s(\mathbb{R}^2_{\h})}\bigr) \andf\\
& \|G\|_{\mathcal{W}^{s}_{\gamma_0,T}}
\eqdefa  \max_{0\leq j\leq [s]}\sup_{t\in [0,T]}\|\p_t^jG(t)\|_{W^{s-j}_{\gamma_0}}.
\end{split}
\eeq}
\end{defi}

We shall prove in Section \ref{sec-blp} that

\begin{prop}\label{S2prop2}
{\sl Under the assumptions of Proposition \ref{S2prop1},  the coupled equations (\ref{phi-negative-1}-\ref{a-order-0}) with the boundary condition \eqref{a-order-0-bc} has a unique solution $(A_0,\, \Phi_1)$ in $\mathcal{W}^{s_0-\frac32}_{1, T_0},$ where $T_0$ is determined by Proposition \ref{S2prop1}. Furthermore, there holds
\begin{equation}\label{phi-negative-1-solution}
\bigl\|(A_0,\, \Phi_1)\bigr\|_{\mathcal{W}^{s_0-\frac32}_{1, T_0}}\leq C\bigl\|\bigl(a_{0,0}^{\rm{in}}-1,\na\vf_{0,0}^{\rm{in}}\bigr)\bigr\|_{H^{s_0-1}}.
\end{equation}}
\end{prop}

To solve the systems \eqref{outer-order-1} and \eqref{outer-order-m} for the inner expansions, we are going
to solve first a linear wave equation in  Section \ref{sec8}. More precisely, let
$\vf_0$ be determined by Proposition \ref{S2prop1} and $f\in W^{s-1}_{T}(\R^3_+)$ for $s\leq s_0-1$ and $T\leq T_0,$ we are going to solve the following linear wave equation
\beq\label{S8eq1a}
\begin{split}
 P(\vf_0, D)\vf\eqdefa& \partial_t^2 \varphi-\dive \left(a^2_0\nabla\varphi\right)+2 \nabla \varphi_{0} \cdot \nabla \partial_t\varphi\\
 &+\dive \bigl((\nabla\varphi_0 \cdot \nabla\varphi)\nabla\varphi_0 \bigr)+\na\partial_t \varphi_0 \cdot\nabla\varphi+\D\vf_0\p_t\vf=f,
\end{split}
\eeq
together with the following initial and boundary conditions:
\beq \label{S8eq2}
\vf(t,y,0)=g(t,y)\in W^{s+\frac12}_T(\R^2) \andf \vf|_{t=0}=\vf_{,0}^{\rm in}, \quad \p_t\vf|_{t=0}=\vf_{,1}^{\rm in}.
\eeq

The result about the unique solvability of the system (\ref{S8eq1a}-\ref{S8eq2}) states as follows:

\begin{thm}\label{S8thm1}
{\sl Let $(\vf_{,0}^{\rm{in}}, \vf_{,1}^{\rm{\rm{in}}})$ satisfy
$\na\vf_{,0}^{\rm{in}}, \vf_{,1}^{\rm{\rm{in}}}\in \,H^{s-1}$ and  the compatibility condition:
   $\partial_t^\ell\left(\varphi-g\right)(0,y,0)=0$ for $y\in\R^2$ and $\ell=0,\cdots, s-1.$
Let $f\in W^{s-1}_T$ for some integer $s\in [4, s_0]$. Then under the assumptions of Proposition \ref{S2prop1}, the system (\ref{S8eq1a}-\ref{S8eq2}) has a unique solution $\vf$ on $[0, T],$ which satisfies
\beno
\left\|\left(\p_t\vf,\na\vf\right)\right\|_{W_{T}^{s-1}}\leq  C\bigl(\|g\|_{ W^{s+\frac12}_{T}(\R^2)}+\|\na \vf_{,0}^{\rm{in}}\|_{H^{s-1}}+\|\vf_{,1}^{\rm{in}}\|_{H^{s-1}}+\|f\|_{W^{s-1}_{T}}\bigr).
\eeno
}
\end{thm}

With $(a_0,\vf_0)$ being determined by Proposition \ref{S2prop1} and $(A_0,\, \Phi_1)$ being determined by Proposition \ref{S2prop2},
 we are going to solve the system \eqref{outer-order-1} with the boundary condition \eqref{varphi-order-0-bc}.

 We first observe from \eqref{S2eq4} and  \eqref{outer-order-1} that
\begin{equation*}\label{outer-conti-1-1}
\begin{split}
\partial_t (a_{1}a_0)+\dive \, (a_0a_1 \nabla \varphi_0)+\frac{1}{2}\dive\,( a_{0}^2 \nabla \varphi_{1})=0.
\end{split}
\end{equation*}
Then we get, by taking $\p_t$ to the $\vf_1$ equation of \eqref{outer-order-1} and inserting the above equation to the resulting one, that
\begin{equation}\label{outer-conti-1-2}
\begin{split}
 \partial_t^2 \varphi_{1}-\Delta \varphi_1&+\dive \bigl((\partial_t \varphi_0+\frac{1}{2}|\nabla \varphi_0|^2)\nabla\varphi_1\bigr)+ \partial_t(\nabla \varphi_{0} \cdot \nabla \varphi_{1})\\
 &\qquad\quad+\dive (\partial_t \varphi_1 \nabla\varphi_0 )+\dive \bigl((\nabla\varphi_0 \cdot \nabla\varphi_1)\nabla\varphi_0 \bigr)=0.
\end{split}
\end{equation}
Noticing from \eqref{outer-order-1} that
\beq \label{S5eq2}
 \partial_t \varphi_{1}\bigl|_{t=0}=-\nabla \varphi_{0} \cdot \nabla \varphi_{1}\bigl|_{t=0}-2a_0 a_{1}\bigl|_{t=0}
 =\na\vf_{0,0}^{\rm in}\cdot\na \vf_{1,0}^{\rm in}-2a_{0,0}^{\rm in}a_{1,0}^{\rm in}\eqdefa \vf_{1,1}^{\rm in},
 \eeq
we complement the equation \eqref{outer-conti-1-2} with the boundary condition \eqref{varphi-order-0-bc} and the initial data \begin{equation}\label{varphi-1-initial-c}
\varphi_1|_{t=0}=\vf_{1,0}^{\rm in}, \quad \partial_t \varphi_1|_{t=0}=\vf_{1,1}^{\rm in}.
\end{equation}

By applying Theorem \ref{S8thm1}, we shall prove in Section \ref{sec-asy} that

\begin{prop}\label{S2prop3}
{\sl Let $s_0\geq 6$ be an integer. Let  $a_{1,0}^{\rm in}, \na\vf_{1,0}^{\rm in}\in H^{s_0-3}$  which satisfy the compatibility condition:
   $\partial_t^\ell\left(\varphi_1+\Phi_1\right)(0,y,0)=0$ for $y\in\R^2$ and $\ell=0,\cdots, s_0-3.$ Then under the assumptions of Proposition \ref{S2prop1},
the system \eqref{outer-order-1} with boundary condition \eqref{varphi-order-0-bc}    has a unique solution
$(a_1,\varphi_1)$  on $[0, T_0]$ such that
\begin{equation}\label{S5eq4}
\begin{split}
\bigl\|\bigl(a_1,\p_t\vf_1, \na\vf_1\bigr)\bigr\|_{W^{s_0-3}_{T_0}}\leq C\bigl(\bigl\|\bigl(a_{1,0}^{\rm in}, \na\vf_{1,0}^{\rm in}\bigr)\bigr\|_{ H^{s_0-3}}
+\bigl\|\bigl(a_{0,0}^{\rm{in}}-1,\na\vf_{0,0}^{\rm{in}}\bigr)\bigr\|_{H^{s_0-1}}\bigr).\end{split} \end{equation}
for  $T_0$ being determined by Proposition \ref{S2prop1}.}
\end{prop}

Let us  turn to  the solvability of
the boundary layer problem (\ref{phi-order-m}-\ref{a-order-m1}) with the boundary condition \eqref{a-order-m1-bc} for $0\leq k\leq m.$
In fact,  we shall prove in Section \ref{sec-asy} that

\begin{prop}\label{S2prop4}
{\sl Let $s_0\geq 2k+5$ be an integer. Let $a_{j,0}^{\rm in}, \na\vf_{j,0}^{\rm in}\in H^{s_0-2j-1}$ with $j=1, ..., k+1,$ which satisfy the
compatibility conditions:

\no ($\mathcal{A}_{k+1}$):   $\partial_t^\ell \left(\varphi_j+\Phi_j\right)(0,y,0)=0$ for $y\in\R^2,$ $\ell=0,\cdots, s_0-2j-1$
and $j=1,\cdots, k+1.$

\no Then under the assumptions of Proposition \ref{S2prop1}, the system (\ref{phi-order-m}-\ref{a-order-m1}) with the boundary condition \eqref{a-order-m1-bc}  has  a unique  solution $(A_{k+1},\, \Phi_{k+2})$ in
$\mathcal{W}^{s_0-2(k+2)+\frac12}_{1, T_0}.$ Moreover, there holds
\begin{equation}\label{phi-negative-1-solution}
\begin{split}
\bigl\|(A_{k+1},\, \Phi_{k+2})\bigr\|_{\mathcal{W}^{s_0-2(k+2)+\frac12}_{1, T_0}}\leq C\Bigl(&\|(a_{0,0}^{\rm{in}}-1,\na\vf_{0,0}^{\rm{in}})\|_{H^{s_0-1}}+\sum_{j=1}^{k+1}\|(a_{j,0}^{\rm{in}},\na\vf_{j,0}^{\rm{in}})\|_{H^{s_0-2j-1}}\Bigr).
\end{split}
\end{equation}}
\end{prop}

Then along the same line to the proof of Proposition \ref{S2prop3}, we have

\begin{prop}\label{S2prop5}
{\sl  Let $s_0\geq 2k+7$ be an integer. Let $a_{k+2,0}^{\rm in}, \na\vf_{k+2,0}^{\rm in}\in H^{s_0-2k-5}$ which satisfy the compatibility condition $(\mathcal{A}_{k+2}).$ Then under the assumptions of Propositions \ref{S2prop1} and \ref{S2prop4},
the system \eqref{outer-order-m} with boundary condition \eqref{S2eq23}    has a unique solution
$(a_{k+2},\varphi_{k+2})$  on $[0, T_0]$ such that
\begin{equation}\label{Saeq1}
\begin{split}
\bigl\|\bigl(a_{k+2},\,\p_t\vf_{k+2}, \,\na\vf_{k+2}\bigr)\bigr\|_{W^{s_0-1-2(k+2)}_{T_0}}
\leq C\Bigl(&\|(a_{0,0}^{\rm{in}}-1,\na\vf_{0,0}^{\rm{in}})\|_{H^{s_0-1}}\\
&+\sum_{j=1}^{k+2}\|(a_{j,0}^{\rm{in}},\na\vf_{j,0}^{\rm{in}})\|_{H^{s_0-2j-1}}\Bigr)
\end{split}
\end{equation}
for  $T_0$ being determined by Proposition \ref{S2prop1}.}
\end{prop}

Let $(a_j,\vf_j)$  for $j=0, \cdots, m+2,$ and $(A_j, \Phi_{j+1})$  for $j=0, \cdots, m+1,$ be constructed in the previous propositions.
We denote
\beq \label{S6eq1}
\begin{split}
 &\Psi^{a, m}\eqdefa a^{\varepsilon, m} e^{\frac{i}{\varepsilon}\varphi^{\varepsilon, m}}\with
 a^{\varepsilon, m}=a^{\text{int}, \varepsilon, m}+[a^{\text{b}, \varepsilon, m}]_\e,\quad \vf^{\varepsilon, m}=\vf^{\text{int}, \varepsilon, m}+[\vf^{\text{b}, \varepsilon, m}]_\e, \\
& a^{\text{int}, \varepsilon, m}= \sum_{j=0}^{m+1}\varepsilon^j a_j,\quad
 a^{\text{b}, \varepsilon, m}=\sum_{j=0}^{m+1}\varepsilon^j A_j,\quad\vf^{\text{int}, \varepsilon, m}
=\sum_{j=0}^{m+2}\varepsilon^j\varphi_j, \quad\vf^{\text{b}, \varepsilon, m}=\sum_{j=1}^{m+2}\varepsilon^j\Phi_j,
\end{split}
\eeq
and
\beq \label{initial-0}
\begin{split}
 &\mathcal{E}_{0}\eqdefa \|(a_{0,0}^{\rm{in}}-1,\na\vf_{0,0}^{\rm{in}})\|_{H^{s_0-1}}^2
+\sum_{j=1}^{m+2}\|(a_{j,0}^{\rm{in}},\na\vf_{j,0}^{\rm{in}})\|_{H^{s_0-2j-1}}^2.
\end{split}
\eeq

Next let $w$ and $\phi$ be  real-valued functions, we are going to seek the true solution of \eqref{NLS-0} with the form:
\begin{equation}\label{S7eq1}
\begin{split}
\Psi^\varepsilon=\bigl(a^{\varepsilon,m}+w\bigr)e^{i\left(\frac{\varphi^{\varepsilon,m}}{\varepsilon}+\phi\right)},
\end{split}
\end{equation}
where $(w,\phi)$ satisfy the boundary conditions:
\begin{equation*}
\begin{split}
&w|_{z=0}=0,\quad \phi|_{z=0}=0.
\end{split}
\end{equation*}
In view of \eqref{S7eq1}, we write
\begin{equation}\label{expr-Phi-1}
\begin{split}
\Psi^\varepsilon=\Psi^{a,m}+\frak{w}e^{\frac{i}{\varepsilon}\varphi^{\varepsilon,m}} \quad\text{with}\quad \frak{w}= w+(a^{\varepsilon,m}+w)(e^{i\phi}-1)\eqdefa \wr+i\wi.
\end{split}
\end{equation}
It turns out that it is more convenient to handle the estimate of $\fw$ than that of $(w,\phi).$  As a matter of fact, we shall derive in
Section \ref{sec-valid-wkb} that $(\wr,\wi)$ verifies
\begin{equation}\label{w-phi-eqns}
\begin{cases}
&\varepsilon\bigl(\partial_{t}\wr+\mathcal{S}_{u^{\varepsilon, m}}(\wr)\bigr)
+\frac{\varepsilon^{2}}{2}\,\Delta \wi\\
&\quad=\Bigl(\frac{\varepsilon^2}{2}\frac{\Delta a^{\varepsilon, m}}{a^{\varepsilon, m}}
+\varepsilon^{m+2} r^m_\vf \Bigr)\wi-\varepsilon^{m+2}r^m_\vf+\text{Im}(Q^\varepsilon({\fw})) \quad \text{\rm{in}} \quad \R_+\times\mathbb{R}^3_+,\\
&\varepsilon \bigl(\partial_{t}\wi+\mathcal{S}_{u^{\varepsilon, m}}(\wi)\bigr)-\frac{\varepsilon^{2}}{2} \Delta \wr  +\Bigl(2(a^{\varepsilon,m})^{2}+\frac{\varepsilon^2}{2}\frac{\Delta a^{\varepsilon, m}}{a^{\varepsilon, m}}
+\varepsilon^{m+2} r^m_\vf \Bigr) \wr \\
&\quad=\varepsilon^{m+1} a^{\e,m}r_a^m-\text{Re}(Q^\varepsilon({\fw})),\\
& \wr |_{z=0}=0,\quad  \wi |_{z=0}=0,
\end{cases}
\end{equation}
where $u^{\e,m}\eqdefa \na\vf^{\e,m},$ $\cS_f(g)\eqdefa f\cdot\na g+\frac12g\na\cdot f,$ and $Q^\varepsilon({\fw}), r^m_a$ and $ r^m_\vf$ are given  respectively by \eqref{def-Q-0} and  \eqref{defrap}.

By crucially using the symmetric property of the operator $\cS_f(g)$ (see Lemma \ref{lem-est-conv-2}) and the special structure of the system \eqref{w-phi-eqns} (especially that we can have the estimate of $\|\wr\|_{L^2_+}$), we shall prove in Section \ref{sec-valid-wkb} the following proposition:

\begin{prop}\label{S2prop6}
{\sl Let $m, N$ and $s_0$ be integers so that $m, N\geq 4$ and $s_0\geq 2m+9+N.$    Let $\Psi^{a, m}=a^{\varepsilon, m} e^{\frac{i}{\varepsilon}\varphi^{\varepsilon, m}}$ be the approximate solutions of \eqref{NLS-0} constructed in \eqref{S6eq1}. Then under the assumptions of Proposition \ref{S2prop5}, there exists a small enough positive constant $\varepsilon_0>0$ such that for any $\e\in (0,\e_0),$ the system \eqref{NLS-0}
 has a unique solution $\Psi^\varepsilon=(a^{\varepsilon,m}+\frak{w})e^{i\frac{\varphi^{\varepsilon,m}}{\varepsilon}}$ on $[0, T_0].$ Moreover, for all $t\in [0, T_0]$ and  ${\cT}\eqdefa\,(\partial_t,\, \nabla_{\h}),$ there holds
\begin{equation}\label{zhang4}
\begin{split}
&\sum_{j=0}^{N-1}\bigl(\|{\Ta}^j\wr\|_{L^2_+}^2 +\|\varepsilon{\Ta}^j\fw\|_{H^1}^2\bigr)\lesssim \mathcal{E}_{0}\varepsilon^{2m+2}.
\end{split}
\end{equation}
}
\end{prop}

The main result of this paper states as follows, the proof of which will be presented in Section \ref{sec-valid-wkb}.

\begin{thm}\label{thmmain}
{\sl Let $m\geq 4$ and $s_0\geq 2m+13$ be integers. Let $a^{\e,m}, \vf^{\e,m}$ and   $\Psi^{a, m}$ be  constructed in \eqref{S6eq1}.  Let $a_{j,0}^{\rm in}, \na\vf_{j,0}^{\rm in}\in H^{s_0-2j-1}$ with $j=1, ..., m+2,$ which satisfy the compatibility conditions, $\mathcal{A}_{m+2}.$ Then there exist sufficiently small positive constants $c$ and   $\varepsilon_0$ such that  under the condition \eqref{S3eq27}, for any $\e\in (0,\e_0),$ \eqref{NLS-0}
 has a unique  solution $\Psi^\varepsilon,$ which satisfies
\begin{equation}\label{ener-infty-1}
\begin{split}
&\left\|e^{-i\frac{\varphi^{\varepsilon,m}}{\varepsilon}}(\Psi^\varepsilon-\Psi^{a, m})\right\|_{L^\infty_{T_0}(W^{1, \infty)}}\leq C\,\mathcal{E}_{0}^{\frac{1}{2}} \varepsilon^{m-1},
\end{split}
\end{equation} for the positive time $T_0$ being determined by Proposition \ref{S2prop1}.}
\end{thm}

%%%%%%%%%%%%%%%%%%%%%%%%%%%%%%%%%%%%%%%%%%%%%%%%%%%%%%%%%%%%
\renewcommand{\theequation}{\thesection.\arabic{equation}}
\setcounter{equation}{0}
%%%%%%%%%%%%%%%%%%%%%%%%%%%%%%%%%%%%%%%%%%%%%%%%%%%%%%%%%%%%

\section{The local well-posedness of the limit system (\ref{S2eq4}-\ref{outer-order-0-bc})}\label{sec-hj}

In this section, we shall prove the local existence of smooth solutions to the
initial-boundary value problem of the limit system (\ref{S2eq4}-\ref{outer-order-0-bc}).
Let us denote $\rho_0 \eqdefa \,a_0^2.$ We rewrite  (\ref{S2eq4}-\ref{outer-order-0-bc}) as
\begin{equation}\label{S3eq1}
\begin{cases}
&\partial_t \rho_0+\nabla \cdot(\rho_0\,\nabla\,\varphi_{0})=0, \quad (t, x) \in \mathbb{R}_{+} \times \R^3_{+},\\
&\partial_t \varphi_{0}+\frac{1}{2} |\nabla \varphi_{0}|^2 +(\r_{0}-1)=0,\\
&\varphi_{0}|_{z=0}=0\andf \r_0\to 1\quad\mbox{as}\quad |x|\to\infty,\\
&\r_0|_{t=0}=\left(a_{0,0}^{\rm{in}}\right)^2, \quad \vf|_{t=0}=\vf_{0,0}^{\rm{in}}.
\end{cases}
\end{equation}
By substituting the equivalent form of the second equation in \eqref{S3eq1}
\beq \label{S3eq2}
\r_0=-\bigl(\partial_t \varphi_{0}+\frac{1}{2} |\nabla \varphi_{0}|^2 -1\bigr)\eeq
into the first equation of \eqref{S3eq1},
we obtain
\begin{equation}\label{outer-order-0-b}
\partial_t^2 \varphi_{0}+ \nabla \varphi_{0}\cdot \nabla \partial_t\varphi_{0}-\nabla \cdot(\rho_0\,\nabla\,\varphi_{0})=0,
\end{equation}
which can also be equivalently written as
\begin{equation}\label{S3eq4}
\begin{split}
\partial_t^2 \varphi_{0}-\Delta \varphi_{0}+ \nabla \varphi_{0} \cdot \nabla \partial_t\varphi_{0} +\dive\,\bigl( (\partial_t \varphi_{0}+\frac{1}{2} |\nabla \varphi_{0}|^2) \nabla \varphi_{0} \bigr)=0,
\end{split}
\end{equation}
or
\begin{equation}\label{S3eq5}
\partial_t^2 \varphi_{0}+ \nabla \varphi_{0}\cdot \nabla \partial_t\varphi_{0}+\partial_t \rho_0=0.
\end{equation}
Let $\vf_{0,1}^{\rm{in}}$ be given by \eqref{vf10}. We implement the wave equation \eqref{S3eq4}
with the initial-boundary conditions:
\begin{equation}\label{outer-order-0-initial}
\varphi_0|_{t=0}=\vf_{0,0}^{\rm{in}}, \quad \partial_t\varphi_0|_{t=0}=\vf_{0,1}^{\rm{in}} \andf \varphi_0|_{z=0}=0.
\end{equation}

Before proceeding, let us first present the following product law in the space $W^s_T,$ the proof of which will
be postponed in the Appendix \ref{appA}.

\begin{lem}\label{S3lem1}
{\sl Let $2\leq s$ and $W^s_T$ be given by Definition \ref{S2def1}. Then for any $f,g\in W^s_T,$ one has
\beq \label{S3eq6}
\|fg\|_{W^s_T}\leq C_s\|f\|_{W^s_T}\|g\|_{W^s_T}.
\eeq}
\end{lem}

The main result of this section states as follows:

\begin{thm}\label{thm-hj-order-0}
{\sl Let $4\leq s_0\in\N $ and $(\vf_{0,0}^{\rm{in}}, \, \vf_{0,1}^{\rm{in}}) \in H^{s_0}(\R^3_{+}) \times H^{s_0-1}(\R^3_{+})$ which satisfies the compatibility condition $(\mathcal{A}_0)$. We assume that
\begin{equation}\label{small-1}
\|\na\vf_{0,0}^{\rm{in}}\|_{H^{s_0-1}}+ \|\vf_{0,1}^{\rm{\rm{in}}}\|_{H^{s_0-1}} \leq c_0,
\end{equation} for some $c_0$ sufficiently small, then
there exists a positive constant ${\cC}$ so that for $T={\cC}c_0^{-1},$ (\ref{S3eq4}-\ref{outer-order-0-initial}) has  a unique solution $\varphi_0$ on $[0, T],$
which satisfies
\begin{equation}\label{solution-order-0}
\begin{split}
\left\|\left(\p_t\vf_0,\na\vf_0\right)\right\|_{W_T^{s_0-1}}
\leq C \,(\|\na\vf_{0,0}^{\rm{in}}\|_{H^{s_0-1}}+ \|\vf_{0,1}^{\rm{\rm{in}}}\|_{H^{s_0-1}}).
\end{split}
\end{equation}
}
\end{thm}

\begin{proof} It is well-known that the existence of solutions to a nonlinear partial differential equation
can be obtained  by first constructing the appropriate approximate solutions, and then performing uniform estimates for such
approximate solutions, and finally applying
a compactness argument.
For simplicity, here we just present the {\it a priori} estimates for sufficiently smooth solutions of (\ref{S3eq4}-\ref{outer-order-0-initial}) on
$[0,T^\ast[$ with $T^\ast$ being the maximal time of existence.

In what follows, we shall separate the proof into the following steps:

\no{\bf Step 1.} $H^1_{\text{tan}}$ estimate

Due to $\p_t\vf|_{z=0}=0,$ by
taking the $L^2$ inner product of \eqref{outer-order-0-b} with $\partial_t \varphi_0$ and using integration by parts, we get
\begin{equation*}\label{outer-order-0-5}
\begin{split}
\frac{d}{dt}\int_{\R^3_{+}}(|\partial_t \varphi_{0}|^2
+\rho_0|\nabla \varphi_{0}|^2)\,dx=\int_{\R^3_{+}} \bigl(\partial_t\rho_0|\nabla \varphi_{0}|^2-2(\nabla \varphi_{0} \cdot \nabla \partial_t\varphi_{0})\partial_t\varphi_0\bigr) \,dx,
\end{split}
\end{equation*}
from which and \eqref{S3eq2},  we infer
\begin{equation}\label{outer-order-0-5}
\begin{split}
\frac{d}{dt}\int_{\R^3_{+}}(|\partial_t \varphi_{0}|^2
+\rho_0|\nabla \varphi_{0}|^2)\,dx\lesssim &\bigl(\|\p_t\varphi_0\|_{L^2_+}^2+\|\nabla\varphi_0\|_{L^2_+}^2\bigr)\\
&\times\bigl(\|\partial_t^2\varphi_0\|_{L^\infty_+}+(1+\|\nabla \varphi_{0}\|_{L^\infty_+} ) \|\nabla\partial_t\varphi_0\|_{L^\infty_+}\bigr).
\end{split}
\end{equation}

\no{\bf Step 2.}  {$H^2_{\text{tan}}$ estimate}\

Recall that ${\cT}\eqdefa\,(\partial_t,\, \nabla_{\h}).$
Applying $\Ta $ to \eqref{outer-order-0-b} gives
\begin{equation}\label{outer-order-0-6}
\begin{split}
\partial_t^2\Ta \varphi_{0}-\nabla\cdot\left( \rho_0\nabla \Ta\varphi_{0} \right)+ \nabla \varphi_{0} \cdot \nabla \partial_t\Ta\varphi_{0} + \nabla \partial_t\varphi_{0}\cdot\nabla \Ta\varphi_{0} -\nabla\cdot\left( \Ta\rho_{0} \nabla \varphi_{0} \right)=0.
\end{split}
\end{equation}
Due to $\p_t\Ta\vf|_{z=0}=0,$ by
taking the $L^2$ inner product of \eqref{outer-order-0-6} with $\partial_t\Ta \varphi_0$ and using integration by parts, one has
\begin{equation}\label{outer-order-0-7}
\begin{split}
\frac{1}{2}&\frac{d}{dt}\int_{\R^3_{+}}\bigl(|\partial_t\Ta \varphi_{0}|^2+ \rho_0|\nabla \Ta\varphi_{0}|^2\bigr)\, dx
 -\frac{1}{2}\int_{\R^3_{+}}  \partial_t\rho_0|\nabla \Ta\varphi_{0}|^2 \,dx\\
& +\int_{\R^3_{+}} (\nabla \varphi_{0} \cdot \nabla \partial_t\Ta\varphi_{0})\partial_t\Ta \varphi_{0} \,dx +\int_{\R^3_{+}}  \left(\nabla \partial_t \varphi_{0}\cdot \nabla \Ta \varphi_{0}\right) \partial_t\Ta \varphi_{0} \,dx\\
&+\int_{\R^3_{+}} \Ta\rho_0 \,\nabla \varphi_{0} \cdot\nabla\partial_t\Ta \varphi_{0} \,dx =0.
\end{split}
\end{equation}
According to \eqref{S3eq2}, it is easy to observe that
\begin{equation*}\label{outer-order-0-8c}
\begin{split}
&\int_{\R^3_{+}} (\nabla \varphi_{0} \cdot \nabla \partial_t\Ta\varphi_{0})\partial_t\Ta \varphi_{0} \,dx+\int_{\R^3_{+}} \Ta\rho_0\, \,\nabla \varphi_{0} \cdot\nabla\partial_t\Ta \varphi_{0} \,dx \\
&=-\int_{\R^3_{+}} (\nabla \varphi_{0}\cdot  \nabla \Ta\varphi_{0})(\nabla \varphi_{0} \cdot \nabla\partial_t\Ta \varphi_{0} ) \,dx\\
&=-\frac{1}{2}\frac{d}{dt}\int_{\R^3_{+}} (\nabla \varphi_{0}\cdot  \nabla \Ta\varphi_{0})^2\,dx+\int_{\R^3_{+}}(\nabla \varphi_{0} \cdot \nabla\Ta \varphi_{0} )  ( \nabla \partial_t\varphi_{0}\cdot\nabla\Ta \varphi_{0}) \,dx.
\end{split}
\end{equation*}
Plugging the above estimate into \eqref{outer-order-0-7} and using the equation \eqref{S3eq5} yields
\begin{equation*}\label{outer-order-0h-12}
\begin{split}
&\frac{d}{dt}\int_{\R^3_{+}}\bigl(|\partial_t\Ta\varphi_{0}|^2+\rho_0|\nabla \Ta \varphi_{0}|^2- |\nabla \varphi_{0}  \cdot\nabla\Ta \varphi_{0} |^2\bigr)\,dx\\
&=\int_{\R^3_{+}} \left( \partial_t\rho_0|\nabla \Ta \varphi_{0}|^2 -2 (\nabla \partial_t\varphi_{0}\cdot\nabla \Ta\varphi_{0} )(\partial_t\Ta \varphi_{0}+\nabla \varphi_{0}\cdot\nabla\Ta\varphi_0)
\right) \,dx,
\end{split}
\end{equation*}
which together with \eqref{S3eq2} ensures that
\begin{equation}\label{outer-order-0-12a}
\begin{split}
\frac{1}{2}&\frac{d}{dt}\int_{\R^3_{+}}\bigl(|\partial_t\Ta \varphi_{0}|^2+\rho_0|\nabla\Ta \varphi_{0}|^2-|\nabla \varphi_{0}\cdot  \nabla \Ta\varphi_{0}|^2\bigr)\, dx \\
&\lesssim \bigl(\|\partial_t^2\varphi_0\|_{L^\infty_+}+(1+\|\nabla \varphi_{0}\|_{L^\infty_+} ) \|\nabla\partial_t\varphi_0\|_{L^\infty_+}\bigr)\bigl(\|\partial_t\Ta\varphi_0\|_{L^2_+}^2+\|\nabla \Ta\varphi_{0}\|_{L^2_+}^2\bigr).
\end{split}
\end{equation}

\no{\bf Step 3.}  {High-order tangential derivatives estimate}

Let $\ell\in\N, $ by applying the operator ${\Ta}^{\ell} $ (with ${\Ta}=(\partial_t, \nabla_{\h})$ and $\Ta^\ell=\p_t^{\al_1}\na_\h^{\al_2}$ for  $\al_1+|\al_2|=\ell \in \mathbb{N}$) to \eqref{outer-order-0-6}, we find
\begin{equation}\label{outer-order-ellh-6}
\begin{split}
&\partial_t^2{\Ta}^{\ell+1} \varphi_{0}-\nabla\cdot( \rho_0\nabla {\Ta}^{\ell+1}\varphi_{0} )+ \nabla \varphi_{0} \cdot \nabla \partial_t {\Ta}^{\ell+1}\varphi_{0}-\nabla\cdot( {\Ta}^{\ell+1}\rho_{0} \nabla \varphi_{0})=g_{\ell}\\
\end{split}
\end{equation}
with
\begin{equation*}
\begin{split}
&g_{\ell}\eqdefa \nabla\cdot\bigl( [{\Ta}^{\ell};\rho_0]\nabla \Ta\varphi_{0} \bigr)- [{\Ta}^{\ell};\nabla \varphi_{0}] \cdot \nabla \partial_t \Ta\varphi_{0} \\
&\qquad\qquad\qquad\qquad -  {\Ta}^{\ell}(\nabla \Ta\varphi_{0} \cdot \nabla \partial_t \varphi_{0})  +\nabla\cdot( [{\Ta}^{\ell}; \nabla \varphi_{0}]\Ta\rho_{0} ).
\end{split}
\end{equation*}
Noticing that $\partial_t{\Ta}^{\ell+1} \varphi_0|_{z=0}=0,$ by taking the $L^2$ inner product of \eqref{outer-order-ellh-6} with $\partial_t{\Ta}^{\ell+1}\varphi_0$ and using integration by parts, one has
\begin{equation}\label{outer-order-ellh-7}
\begin{split}
\frac{1}{2}&\frac{d}{dt}\int_{\R^3_{+}}\bigl(|\partial_t{\Ta}^{\ell+1}\varphi_{0}|^2+ \rho_0|\nabla {\Ta}^{\ell+1} \varphi_{0}|^2\bigr)\,dx-\frac{1}{2}\int_{\R^3_{+}}  \partial_t\rho_0|\nabla {\Ta}^{\ell+1}\varphi_{0}|^2 \,dx \\
&-\int_{\R^3_{+}}\D\vf_0|\partial_t{\Ta}^{\ell+1} \varphi_{0}|^2\,dx -\int_{\R^3_{+}}  {\Ta}^{\ell}(\nabla \varphi_{0}\cdot\nabla\Ta\varphi_0) | (\nabla \varphi_{0}  \cdot\nabla\partial_t{\Ta}^{\ell+1} \varphi_{0}) \,dx\\
 = &\int_{\R^3_{+}} g_{\ell}\, | \partial_t{\Ta}^{\ell+1} \varphi_{0} \,dx.
\end{split}
\end{equation}
By using integration by parts, one has
\begin{equation*}\label{outer-order-ellh-8c}
\begin{split}
-&\int_{\R^3_{+}}  {\Ta}^{\ell}(\nabla \varphi_{0}\cdot\nabla\Ta\varphi_0) | (\nabla \varphi_{0}  \cdot\nabla\partial_t{\Ta}^{\ell+1} \varphi_{0}) \,dx \\
=&-\frac{1}{2}\frac{d}{dt}\int_{\R^3_{+}}  |\nabla \varphi_{0}  \cdot\nabla{\Ta}^{\ell+1}\varphi_{0}|^2  \,dx
+\int_{\R^3_{+}} \Bigl( (\nabla \varphi_{0}\cdot\nabla{\Ta}^{\ell+1}\varphi_0)( \nabla\partial_t \varphi_{0}  \cdot\nabla{\Ta}^{\ell+1} \varphi_{0}) \\
&\,+\left( \nabla([{\Ta}^{\ell}; \nabla \varphi_{0} ] \cdot\nabla\Ta \varphi_{0}) \cdot \nabla \varphi_{0} + [{\Ta}^{\ell}; \nabla \varphi_{0} ] \cdot\nabla\Ta \varphi_{0}  \Delta \varphi_{0}\right) \partial_t{\Ta}^{\ell+1} \varphi_{0}  \Bigr)\,dx.
\end{split}
\end{equation*}
Plugging the above equality into \eqref{outer-order-ellh-7} yields
\begin{equation}\label{outer-order-ellh-12}
\begin{split}
&\frac{1}{2}\frac{d}{dt}\int_{\R^3_{+}}\bigl(|\partial_t{\Ta}^{\ell+1}\varphi_{0}|^2+ \rho_0|\nabla {\Ta}^{\ell+1} \varphi_{0}|^2-|\nabla \varphi_{0}  \cdot\nabla{\Ta}^{\ell+1} \varphi_{0}|^2\bigr)  \,dx=\mathfrak{R}_{\ell},
\end{split}
\end{equation}
with
\begin{equation*}
\begin{split}
\mathfrak{R}_{\ell}\eqdefa &\frac{1}{2} \int_{\R^3_{+}}  \partial_t\rho_0 |\nabla {\Ta}^{\ell+1} \varphi_{0}|^2
\,dx+\int_{\R^3_{+}}\D\vf_0|\partial_t{\Ta}^{\ell+1} \varphi_{0}|^2\,dx\\
&-\int_{\R^3_{+}} \Bigl( (\nabla \varphi_{0}\cdot\nabla{\Ta}^{\ell+1}\varphi_0)( \nabla\partial_t \varphi_{0}  \cdot\nabla{\Ta}^{\ell+1}
 \varphi_{0}) \\
&\,+\left(g_{\ell}- \nabla([{\Ta}^{\ell}; \nabla \varphi_{0} ] \cdot\nabla\Ta \varphi_{0}) \cdot \nabla \varphi_{0} - [{\Ta}^{\ell}; \nabla \varphi_{0} ] \cdot\nabla\Ta\varphi_{0}  \Delta \varphi_{0}\right) \partial_t{\Ta}^{\ell+1} \varphi_{0}  \Bigr)\,dx,
\end{split}
\end{equation*}
from which, we infer
\begin{equation*}\label{outer-order-ellh-13}
\begin{split}
|\mathfrak{R}_{\ell}|\lesssim &\bigl(\|\partial_t\rho_0\|_{L^\infty_+}+ \|\nabla \varphi_{0}\|_{L^\infty_+}\|\nabla\partial_t \varphi_{0}\|_{L^\infty_+}  \bigr) \|\nabla {\Ta}^{\ell+1} \varphi_{0}\|_{L^2_+}^2+
\|\D\vf_0\|_{L^\infty_+}\|\partial_t{\Ta}^{\ell+1} \varphi_{0}\|_{L^2_+}^2\\
&
+\bigl\|\bigl(g_{\ell}- \nabla([{\Ta}^{\ell}; \nabla \varphi_{0} ] \cdot\nabla\Ta \varphi_{0}) \cdot \nabla \varphi_{0} - [{\Ta}^{\ell}; \nabla \varphi_{0} ] \cdot\nabla\Ta \varphi_{0}  \Delta \varphi_{0}\bigr)\bigr\|_{L^2_+}\|\partial_t{\Ta}^{\ell+1} \varphi_{0}\|_{L^2_+}.
\end{split}
\end{equation*}

Recall that for $s\in\N,$ $\|\vf_0(t)\|_{W^s}^2=\sum_{j=0}^s\|\p_t^j\vf_0(t)\|_{H^{s-j}}.$ Then  by virtue of  \eqref{S3eq2},
 and the Sobolev embedding theorem: $H^2(\mathbb{R}^3_+) \hookrightarrow L^\infty(\mathbb{R}^3_+), $  we deduce that
\begin{equation}\label{outer-order-ell-14}
\begin{split}
&\|\partial_t\rho_0\|_{L^\infty_+}+ \|\nabla \varphi_{0}\|_{L^\infty_+}\,\|\nabla \partial_t \varphi_{0} \|_{L^\infty_+}
\lesssim \bigl(1+\|\na\varphi_0\|_{L^\infty_+}\bigr)\|\p_t\varphi_0\|_{W^3}.
\end{split}
\end{equation}

Next for $4\leq s\in\N,$ we claim that
\beq \label{S3eq7}
\|fg\|_{H^1}\lesssim \|f\|_{H^1}\|g\|_{H^2},
\eeq
and
\beq\label{S3eq10}
\sum_{\ell=1}^{s-2}\|\na([\Ta^\ell;f]g)\|_{L^2_+}\lesssim \|f\|_{W^{s-1}}\|g\|_{W^{s-2}}.
\eeq

Indeed, it follows from Sobolev embedding theorem that
\beno
\begin{split}
\|fg\|_{H^1}=&\|fg\|_{L^2_+}+\|g\na f \|_{L^2_+}+\|f\na g\|_{L^2_+}\\
\leq &\bigl(\|f\|_{L^2_+}+\|\na f\|_{L^2_+}\bigr)\|g\|_{L^\infty_+}+\|f\|_{L^6_+}\|\na g\|_{L^3_+}\lesssim \|f(t)\|_{H^1}\|g(t)\|_{H^2},
\end{split}
\eeno
which yields \eqref{S3eq7}.

By applying \eqref{S3eq7}, we find
\begin{equation*}
\begin{split}
\|\nabla\cdot( [{\Ta}^{\ell}; f]g)\|_{L^2_+}\lesssim& \sum_{i=0}^{\ell-1}\|{\Ta}^{\ell-i} f {\Ta}^{i}g\|_{H^1}\\
\lesssim& \sum_{i=1}^{\ell-1}\|{\Ta}^{\ell-i} f\|_{H^2}\| {\Ta}^{i}g\|_{H^1}
+\|{\Ta}^{\ell} f\|_{H^1}\|g\|_{H^2}\\
\lesssim& \sum_{i=1}^{\ell-1}\|f\|_{W^{\ell-i+2}}\|g\|_{W^{i+1}}
+\|f\|_{W^{\ell+1}}\|g\|_{H^2},
\end{split}
\end{equation*}
which leads to \eqref{S3eq10}.

Notice that $$[{\Ta}^{\ell};\nabla \varphi_{0}] \cdot \nabla \partial_t \Ta\varphi_{0}
=\na\cdot\bigl([{\Ta}^{\ell};\nabla \varphi_{0}]  \partial_t \Ta\varphi_{0}\bigr)-[{\Ta}^{\ell};\D \varphi_{0}]  \partial_t \Ta\varphi_{0},
$$
we get, by applying \eqref{S3eq10} and Lemma \ref{S3lem1}, that
\beno
\begin{split}
\sum_{\ell=1}^{s_0-2}\|g_\ell\|_{L^2_+}\lesssim \bigl(\|\p_t\vf_0\|_{W^{s_0-1}}+&\|\na\vf_0\|_{W^{s_0-1}}+\|\na\vf_0\|_{W^{s_0-1}}^2\bigr)\\
&\times\bigl(
\|\p_t\Ta\vf_0\|_{W^{s_0-2}}+\|\na\Ta\vf_0\|_{W^{s_0-2}}\bigr).
\end{split}
\eeno
Along the same line, one has
\begin{equation*}\label{outer-order-ellh-14}
\begin{split}
&\sum_{\ell=1}^{s_0-2}\bigl\|\bigl(\nabla([{\Ta}^{\ell}; \nabla \varphi_{0} ] \cdot\nabla\Ta \varphi_{0}) \cdot \nabla \varphi_{0} + [{\Ta}^{\ell}; \nabla \varphi_{0} ] \cdot\nabla\Ta \varphi_{0}  \Delta \varphi_{0}\bigr)\bigr\|_{L^2_+}\\
&\lesssim \|\na\varphi_0\|_{L^\infty_+} \|\na\varphi_0\|_{W^{s_0-1}}\|\na\Ta\vf_0\|_{W^{s_0-2}}+\bigl\|[{\Ta}^{\ell}; \nabla \varphi_{0} ] \cdot\nabla\Ta \varphi_{0}\|_{L^6_+}\|\Delta \varphi_{0}\|_{L^3_+}\\
&\lesssim \|\na\varphi_0\|_{H^2}\|\na\varphi_0\|_{W^{s_0-1}}\|\na\Ta\vf_0\|_{W^{s_0-2}}.
\end{split}
\end{equation*}
This together with \eqref{outer-order-ell-14} ensures that
\begin{equation}\label{outer-order-ellh-16}
\begin{split}
\sum_{\ell=1}^{s_0-2}|\mathfrak{R}_{\ell}|\
\lesssim & \bigl(
\|\p_t\Ta\vf_0\|_{W^{s_0-2}}^2+\|\na\Ta\vf_0\|_{W^{s_0-2}}^2\bigr) \\
&\times
\bigl(\|\p_t\vf_0\|_{W^{s_0-1}}+\|\na\vf_0\|_{W^{s_0-1}}+\|\na\vf_0\|_{W^{s_0-1}}^2\bigr).
\end{split}
\end{equation}
Inserting the estimate \eqref{outer-order-ellh-16} into \eqref{outer-order-ellh-12} leads to
\begin{equation}\label{outer-order-ellh-13}
\begin{split}
\sum_{\ell=1}^{s_0-2}&\frac{d}{dt}\int_{\R^3_{+}}\bigl(|\partial_t{\Ta}^{\ell+1}\varphi_{0}|^2+ \rho_0|\nabla {\Ta}^{\ell+1} \varphi_{0}|^2-|\nabla \varphi_{0}  \cdot\nabla{\Ta}^{\ell+1} \varphi_{0}|^2\bigr)  \,dx \\
\lesssim & \bigl(\|\p_t\vf_0\|_{W^{s_0-1}}+\|\na\vf_0\|_{W^{s_0-1}}+\|\na\vf_0\|_{W^{s_0-1}}^2\bigr)\bigl(
\|\p_t\Ta\vf_0\|_{W^{s_0-2}}^2+\|\na\Ta\vf_0\|_{W^{s_0-2}}^2\bigr).
\end{split}
\end{equation}

Let us define three energy functionals of $\vf_0$ as
\beq \label{S3eq11}
\begin{split}
E_s(t)\eqdefa& \|\p_t\vf_0(t)\|_{W^{s-1}}^2+\|\na\vf_0(t)\|_{W^{s-1}}^2;\\
{E}_{s, \text{tan}}(t) \eqdefa& \sum_{\ell=0}^{s-1}\bigl(\|\partial_t{\cT}^\ell\varphi_0(t)\|_{L^2(\R^3_{+})}^2
+\|\nabla{\Ta}^\ell\varphi_0(t)\|_{L^2(\R^3_{+})}^2\bigr);\\
\widetilde{{E}}_{s, \text{tan}}(t) \eqdefa& \sum_{\ell=0}^{s-1}\int_{\R^3_{+}}\bigl(|\partial_t{\Ta}^\ell\varphi_0(t)|^2
+\rho_0|\nabla{\Ta}^\ell\varphi_0(t)|^2-|\nabla \varphi_{0}\cdot \nabla {\Ta}^\ell\varphi_{0}|^2\bigr)\,dx.
\end{split}
\eeq
Then by summing up the estimates, \eqref{outer-order-0-5}, \eqref{outer-order-0-12a} and \eqref{outer-order-ellh-13}, we achieve
\begin{equation}\label{outer-order-14}
\begin{split}
&\frac{d}{dt} {{E}}_{s_0, \text{tan}}(t) \leq C\bigl(1+E_{s_0}(t)\bigr)E_{s_0}^{\frac32}(t).
\end{split}
\end{equation}
For $\d>0$ being sufficiently small, which will be determined later on, we define
\begin{equation}\label{assume-1}
 T^\star_1\eqdefa \sup \Bigl\{\ t< T^\ast:\  E_{s_0}(t)\leq \d\ \ \Bigr\}.
\end{equation}
Then for $t\leq T_1^\star,$
we observe from \eqref{S3eq2} that  there exits a positive constant $C_0$ such that
\begin{equation}\label{tan-equiv-norm-1a}
 C_0^{-1} {E}_{s_0, \text{tan}}(t) \leq\widetilde{{E}}_{s_0, \text{tan}}(t)\leq C_0 {E}_{s_0, \text{tan}}(t),
\end{equation}
provided that $\d$ is sufficiently small in \eqref{assume-1}.

\no{\bf Step 4.} {Full energy estimates}

Let  $
E_s(t),$ $
{E}_{s, \text{tan}}(t)$ be given by \eqref{S3eq11}. We  claim that
\beq \label{S3eq12}
E_\ell(t)\leq C_\ell {E}_{\ell, \text{tan}}(t)\quad \mbox{for}\quad t\leq T_1^\star \andf \ell=2,\cdots,s_0,
\eeq
provided that $\d$ is sufficiently small in \eqref{assume-1}.

When $\ell=2,$ we have
\beq \label{S3eq13}
E_2(t)\leq \|\p_t\vf_0\|_{H^1}^2+\|\p_t^2\vf_0\|_{L^2_+}^2+\|\na\na_\h\vf\|_{L^2_+}^2+\|\p_3^2\vf\|_{L^2_+}^2
+\|\na\p_t\vf_0\|_{L^2_+}^2.
\eeq
Yet in view of \eqref{S3eq4}, we have
\beq \label{S3eq14}
\p_3^2\varphi_{0}=\partial_t^2 \varphi_{0}-\D_\h\vf_0+ \nabla \varphi_{0} \cdot \nabla \partial_t\varphi_{0}
+\dive\bigl((\partial_t \varphi_{0}+\frac{1}{2} |\nabla \varphi_{0}|^2) \na \varphi_{0} \bigr),
\eeq
which implies that
\beno
\begin{split}
\|\p_3^2\varphi_{0}\|_{L^2_+}\leq &\|\partial_t^2 \varphi_{0}\|_{L^2_+}+\|\D_\h\vf_0\|_{L^2_+}+ C\bigl(\|\nabla \varphi_{0}\|_{L^\infty_+}^2\|\na^2\vf_0\|_{L^2_+}\\
&+(\|\nabla \varphi_{0}\|_{L^\infty_+}
+\|\D\vf_0\|_{L^3_+})\|\nabla \partial_t\varphi_{0}\|_{L^2_+}\bigr)\\
\leq &C\bigl(\bigl(1+\|\na\vf_0\|_{H^2}^2\bigr){E}_{2, \text{tan}}^{\frac12}(t)+\|\na\vf_0\|_{H^2}^2\|\p_3^2\varphi_{0}\|_{L^2_+}\bigr).
\end{split}
\eeno
So that as long as $\d$ is sufficiently small in  \eqref{assume-1}, we obtain
\beno
\|\p_3^2\varphi_{0}\|_{L^2_+}\leq C {E}_{2, \text{tan}}^{\frac12}(t).
\eeno
Inserting the above estimate into \eqref{S3eq13} gives rise to
\beno
E_2(t) \leq C {E}_{2, \text{tan}}(t).
\eeno
This proves \eqref{S3eq12} for $\ell=2.$

Now we assume that \eqref{S3eq12} holds for $\ell=k,$ we are going to prove that \eqref{S3eq12} holds for
$\ell=k+1\leq s_0.$ We first notice by the definition that
\beno
E_{k+1}(t)\leq \sum_{j=0}^k\bigl(\|\p_t^{j+1}\vf_0(t)\|_{H^{k-j}}^2+\|\na\p_t^j\vf_0(t)\|_{H^{k-j}}^2\bigr).
\eeno

\noindent$\bullet$\underline{
When $j=k.$}

We observe from  \eqref{S3eq11} that
\beno
\|\p_t^{k+1}\vf_0(t)\|_{L^2_+}^2+\|\na \p_t^{k}\vf_0(t)\|_{L^2_+}^2\leq {E}_{k+1, \text{tan}}(t).
\eeno

\noindent$\bullet$\underline{
When $j=k-1.$}

It follows from  \eqref{S3eq11} that
\beno
\|\p_t^{k}\vf_0\|_{H^1}^2=\|\p_t^{k}\vf_0\|_{L^2_+}^2+\|\na \p_t^{k}\vf_0\|_{L^2_+}^2\leq {E}_{k+1, \text{tan}}(t).
\eeno
Whereas notice that
\beno
\begin{split}
\|\na\p_t^{k-1}\vf_0(t)\|_{H^1}^2=&\|\na\p_t^{k-1}\vf_0(t)\|_{L^2_+}^2+\|\na^2\p_t^{k-1}\vf_0(t)\|_{L^2_+}^2\\
\leq &\|\p_t\vf_0(t)\|_{W^{k-1}}^2+\|\na\na_\h\p_t^{k-1}\vf_0(t)\|_{L^2_+}^2+\|\p_3^2\p_t^{k-1}\vf_0(t)\|_{L^2_+}^2.
\end{split}
\eeno
We deduce from \eqref{S3eq14} that
\beno
\begin{split}
\|\p_3^2\p_t^{k-1}\varphi_{0}\|_{L^2_+}\leq \|\p_t^{k+1} \varphi_{0}\|_{L^2_+}&+\|\D_\h\p_t^{k-1}\vf_0\|_{L^2_+}
 +\|\p_t^{k-1}(\na\vf_0\cdot\na\p_t\vf_0)\|_{L^2_+}\\
 &\qquad+\bigl\|\p_t^{k-1}\dive\bigl((\p_t\vf_0+\frac12|\na\vf_0|^2)\na\vf_0\bigr)\bigr\|_{L^2_+}.
\end{split}
\eeno
Yet for $j\in [0,k-2],$ it follows from the law of product, Lemma \ref{S3lem1}, that
\beq\label{S3eq19}
\begin{split}
\|\p_t^{j+1}(\na\vf_0\cdot\na\p_t\vf_0)\|_{H^{k-j-2}}\leq &\|\na\vf_0\cdot\na\p_t\vf_0\|_{W^{k-1}}\\
\leq &C\|\na\vf_0\|_{W^{k-1}}\|\na\p_t\vf_0\|_{W^{k-1}},
\end{split}
\eeq
and
\beq\label{S3eq20}
\begin{split}
\bigl\|\p_t^{j+1}\dive\bigl((\p_t\vf_0+\frac12|\na\vf_0|^2)\na\vf_0\bigr)\bigr\|_{H^{k-j-2}}\leq &\bigl\|(\p_t\vf_0+\frac12|\na\vf_0|^2)\na\vf_0\bigr\|_{W^{k}}\\
\leq &C\bigl(\|\p_t\vf_0\|_{W^k}+\|\na \vf_0\|_{W^k}^2\bigr)\|\na \vf_0\|_{W^k}.
\end{split}
\eeq
Therefore, we obtain
\beno
\|\p_3^2\p_t^{k-1}\varphi_{0}\|_{L^2_+}\leq C\left({E}^{\frac12}_{k+1, \text{tan}}(t)+(1+\|\na\vf_0\|_{W^k})\|\na\vf_0\|_{W^k} {E}^{\frac12}_{k+1}(t)\right).
\eeno
 $\|\na\p_t^{k-1}\vf_0(t)\|_{H^1}$ shares the same estimate.

  As a result, it comes out
\beno
\|\p_t^{k}\vf_0\|_{H^1}^2+\|\na\p_t^{k-1}\vf_0(t)\|_{H^1}^2\leq  C\Bigl({E}_{k+1, \text{tan}}(t)+(1+\|\na\vf_0\|_{W^k}^2)\|\na\vf_0\|_{W^k}^2 {E}_{k+1}(t)\Bigr).
\eeno

\noindent$\bullet$\underline{
When $k-j\geq 2.$}

We have
\beno
\begin{split}
\|\p_t^{j+1}\vf_0\|_{H^{k-j}}^2=&\|\p_t^{j+1}\vf_0\|_{H^{k-j-1}}^2+\|\na^2\p_t^{j+1}\vf_0\|_{H^{k-j-2}}^2\\
\leq &\|\p_t\vf_0\|_{W^{k-1}}^2+\|\na\na_\h\p_t^{j+1}\vf_0\|_{H^{k-j-2}}^2+\|\p_3^2\p_t^{j+1}\vf_0\|_{H^{k-j-2}}^2.
\end{split}
\eeno
By virtue of \eqref{S3eq14}, we find
\beno
\begin{split}
\|\p_3^2&\p_t^{j+1}\vf_0\|_{H^{k-j-2}}\leq \|\p_t^{j+3}\vf_0\|_{H^{k-j-2}}+\|\p_t^{j+1}\D_\h\vf_0\|_{H^{k-j-2}}\\
&+\|\p_t^{j+1}(\na\vf_0\cdot\na\p_t\vf_0)\|_{H^{k-j-2}}+\bigl\|\p_t^{j+1}\dive\bigl((\p_t\vf_0+\frac12|\na\vf_0|^2)\na\vf_0\bigr)\bigr\|_{H^{k-j-2}},
\end{split}
\eeno which together with \eqref{S3eq19} and \eqref{S3eq20} ensures that
\beq \label{S3eq15}
\begin{split}
\|\p_t^{j+1}\vf_0\|_{H^{k-j}}\leq& \|\p_t\vf_0\|_{W^{k-1}}+\|\na\na_\h\p_t^{j+1}\vf_0\|_{H^{k-j-2}}+\|\p_t^{j+3}\vf_0\|_{H^{k-j-2}}\\
&+\|\p_t^{j+1}\D_\h\vf_0\|_{H^{k-j-2}}+C\bigl(1+\|\na\vf_0\|_{W^k}\bigr)\|\na\vf_0\|_{W^k}E_{k+1}^{\frac12}(t).
\end{split}
\eeq
In the case when $k-j\geq 3,$ we have
\beno
\begin{split}
\|\na\na_\h\p_t^{j+1}\vf_0\|_{H^{k-j-2}}=&\|\na\na_\h\p_t^{j+1}\vf_0\|_{H^{k-j-3}}+\|\na^2\na_\h\p_t^{j+1}\vf_0\|_{H^{k-j-3}}\\
\leq & \|\p_t\vf_0\|_{W^{k-1}}+\|\na\na_\h^2\p_t^{j+1}\vf_0\|_{H^{k-j-3}}+\|\p_3^2\na_\h\p_t^{j+1}\vf_0\|_{H^{k-j-3}}.
\end{split}
\eeno
Yet it follows from \eqref{S3eq14} that
\beno
\begin{split}
\|\p_3^2&\na_\h\p_t^{j+1}\vf_0\|_{H^{k-j-3}}\leq   \|\na_\h\p_t^{j+3}\vf_0\|_{H^{k-j-3}}+\|\p_t^{j+1}\na^3_\h\vf_0\|_{H^{k-j-3}}\\
&+\|\na_\h\p_t^{j+1}(\na\vf_0\cdot\na\p_t\vf_0)\|_{H^{k-j-3}}+\bigl\|\na_\h\p_t^{j+1}\dive\bigl((\p_t\vf_0+\frac12|\na\vf_0|^2)\na\vf_0\bigr)\bigr\|_{H^{k-j-3}},
\end{split}
\eeno
from which and \eqref{S3eq19}, \eqref{S3eq20}, we infer
\beno
\begin{split}
\|\p_3^2\na_\h\p_t^{j+1}\vf_0\|_{H^{k-j-3}}\leq &  \|\na_\h\p_t^{j+3}\vf_0\|_{H^{k-j-3}}+\|\p_t^{j+1}\na^3_\h\vf_0\|_{H^{k-j-3}}\\
&+C\bigl(1+\|\na\vf_0\|_{W^k}\bigr)\|\na\vf_0\|_{W^k} E_{k+1}^{\frac12}(t).
\end{split}
\eeno
Inserting the above estimate into \eqref{S3eq15} gives rise to
\beno
\begin{split}
\|\p_t^{j+1}\vf_0\|_{H^{k-j}}\leq & 2\|\p_t\vf_0\|_{W^{k-1}}+\|\p_t^{j+3}\vf_0\|_{H^{k-j-2}}+\|\p_t^{j+1}\D_\h\vf_0\|_{H^{k-j-2}}\\
&+\|\na\na_\h^2\p_t^{j+1}\vf_0\|_{H^{k-j-3}}+\|\na_\h\p_t^{j+3}\vf_0\|_{H^{k-j-3}}\\
&+\|\p_t^{j+1}\na^3_\h\vf_0\|_{H^{k-j-3}}+C\bigl(1+\|\na\vf_0\|_{W^k}\bigr)\|\na\vf_0\|_{W^k} E_{k+1}^{\frac12}(t).
\end{split}
\eeno
By finite steps of iteration and using the inductive assumption for $\ell=k,$ we deduce that
\beq \label{S3eq16}
\|\p_t^{j+1}\vf_0\|_{H^{k-j}}^2\leq C_k{E}_{k+1, \text{tan}}(t)+C\bigl(1+\|\na\vf_0\|_{W^k}^2\bigr)\|\na\vf_0\|_{W^k}^2 E_{k+1}(t).
\eeq
The same estimate holds for $\|\na\p_t^j\vf_0(t)\|_{H^{k-j}}.$

Therefore we conclude that
\beno
E_{k+1}(t)\leq C_k\Bigl({E}_{k+1, \text{tan}}(t)+\bigl(1+\|\na\vf_0\|_{W^k}^2\bigr)\|\na\vf_0\|_{W^k}^2 E_{k+1}(t)\Bigr).
\eeno
Then in view of \eqref{assume-1}, as long as $\d$ is small enough, we deduce \eqref{S3eq12} for $\ell=k+1.$

Now we are in a position to complete the proof of Theorem \ref{thm-hj-order-0}.

Thanks to \eqref{outer-order-14} and \eqref{S3eq12}, we obtain for $t\leq T_1^\star$ that
\beno
\frac{d}{dt}\wt{E}_{s_0, \text{tan}}(t)\leq C_{s_0}\sqrt{\d} {E}_{s_0 \text{tan}}(t)\leq C_{s_0}C_0\sqrt{\d} \wt{E}_{s_0, \text{tan}}(t),
\eeno
where we used \eqref{tan-equiv-norm-1a} in the last step. Applying Gronwall's inequality gives rise to
\beq \label{S3eq17a}
\begin{split}
\wt{E}_{s_0, \text{tan}}(t)\leq &\wt{E}_{s_0, \text{tan}}(0)\exp\left(C_{s_0}C_0\sqrt{\d} t\right)\\
\leq &C_0 {E}_{s_0, \text{tan}}(0)\exp\left(C_{s_0}C_0\sqrt{\d} t\right).
\end{split}
\eeq
On the other hand, we deduce from \eqref{S3eq4} and \eqref{small-1} that
$$
{E}_{s_0, \text{tan}}(0)\leq C_{s_0}(\|\na\vf_{0,0}^{\rm{in}}\|_{H^{s_0-1}}+ \|\vf_{0,1}^{\rm{\rm{in}}}\|_{H^{s_0-1}})^2\leq C_{s_0} c_0^2,
$$
which together with \eqref{S3eq17a} ensures that
\beq \label{S3eq17}
\wt{E}_{s_0, \text{tan}}(t)\leq  C_0C_{s_0}c_0^2 \exp\left(C_{s_0}C_0\sqrt{\d} t\right)\quad\mbox{for}\ \ t\leq T_1^\star.
\eeq

Let us denote $\bar{T}\eqdefa \min\bigl(T_1^\star, (C_0C_{s_0}\sqrt{\d})^{-1}\bigr).$  If we assume by a contradict argument that $T_1^\star <(C_0C_{s_0}\sqrt{\d})^{-1}$, then for $t\leq \bar{T}=T_1^\star,$ we deduce from
\eqref{S3eq17} that
\beno
\wt{E}_{s_0, \text{tan}}(t)\leq C_0C_{s_0}c_0^2 e,
\eeno
from which, and \eqref{S3eq12}, we infer
\beq \label{S3eq18}
\begin{split}
E_{s_0}(t)
\leq C_{s_0}{E}_{s_0, \text{tan}}(t)\leq C_{s_0}C_0\wt{E}_{s_0, \text{tan}}(t)\leq C_0^2C_{s_0}^{2}\,e\,c_0^2.
\end{split}
\eeq
Then as long as we take the positive $c_0$ to be so small that $C_0^2C_{s_0}^{2}\,e\,c_0^2=\frac{\d}{2},$ we find
\beno
E_{s_0}(t)\leq \frac{\d}{2} \quad \mbox{for}\ t\leq \bar{T}=T_1^\star.
\eeno
This contradicts with the definition of $T_1^\star$ given by   \eqref{assume-1}. This in turn  shows that $T_1^\star \geq (C_0C_{s_0}\sqrt{\d})^{-1}=\bigl(\sqrt{2e}\,C_0^{2}C_{s_0}^2\,c_0\bigr)^{-1}.$
This together with \eqref{S3eq18} completes the proof of Theorem \ref{thm-hj-order-0}.
\end{proof}

\begin{rmk}\label{rmk3.1} We remark that it is crucial to apply $\Ta^{\ell}$ to \eqref{outer-order-0-6}, and then perform
the energy estimate for the equation \eqref{outer-order-ellh-6}. Otherwise,
let $\ell\in\N, $ by applying the operator ${\Ta}^{\ell} $  to \eqref{outer-order-0-b}, we find
\begin{equation}\label{outer-order-ell-6a}
\begin{split}
&\partial_t^2 {\Ta}^{\ell}\varphi_{0}-\nabla\cdot\bigl( \rho_0\nabla {\Ta}^{\ell}\varphi_{0} \bigr)+ {\Ta}^{\ell}(\nabla \varphi_{0} \cdot \nabla \partial_t\varphi_{0})=\nabla\cdot\bigl( [{\Ta}^{\ell}; \rho_0]\nabla \varphi_{0}  \bigr).
\end{split}
\end{equation}
Due to $\partial_t{\Ta}^{\ell} \varphi_0|_{z=0}=0,$ by taking the $L^2$ inner product of \eqref{outer-order-ell-6a} with $\partial_t{\Ta}^{\ell} \varphi_0$ and using integration by parts,  we find
\beq\label{S3eq8a}
\begin{split}
&\frac{1}{2}\frac{d}{dt}\int_{\R^3_{+}}\bigl((\partial_t{\Ta}^{\ell} \varphi_{0})^2 +\rho_0|\nabla {\Ta}^{\ell}\varphi_{0}|^2 \bigr) \,dx=\frac{1}{2}\int_{\R^3_{+}}\p_t\rho_0|\nabla {\Ta}^{\ell}\varphi_{0}|^2  \,dx\\
&\qquad-\int_{\R^3_{+}} {\Ta}^{\ell}(\nabla \varphi_{0} \cdot \nabla \partial_t\varphi_{0}) | \partial_t {\Ta}^{\ell}\varphi_{0} \,dx+\int_{\R^3_{+}} \nabla\cdot\bigl( [{\Ta}^{\ell}; \rho_0]\nabla \varphi_{0}  \bigr) | \partial_t {\Ta}^{\ell}\varphi_{0} \,dx.
\end{split}
\eeq
It is easy to observe that
\begin{equation*}\label{outer-order-ell-8c}
\begin{split}
-\int_{\R^3_{+}} {\Ta}^{\ell}(\nabla \varphi_{0} \cdot \nabla \partial_t\varphi_{0}) | \partial_t {\Ta}^{\ell}\varphi_{0} \,dx
=&\frac12\int_{\R^3_{+}} \D\vf_0  (\partial_t {\Ta}^{\ell} \varphi_{0})^2  \,dx\\
&-\int_{\R^3_{+}}    [{\Ta}^{\ell};\nabla \varphi_{0}]  \nabla \partial_t\varphi_{0} | \partial_t {\Ta}^{\ell}\varphi_{0} \,dx.
\end{split}
\end{equation*}
Plugging the above equality into \eqref{S3eq8a} and summing up the resulting inequalities for $\ell$ varying from
$1$ to $s_0-1$ yields
\begin{equation}\label{outer-order-ell-12a}
\begin{split}
&\sum_{\ell=1}^{s_0-1}\frac{d}{dt}\int_{\R^3_{+}}\bigl(|{\Ta}^{\ell} \partial_t\varphi_{0}|^2\, + \rho_0 |\nabla {\Ta}^{\ell}\varphi_{0}|^2\bigr)\,dx\\
&\lesssim \sum_{\ell=1}^{s_0-1}\Bigl(\|\p_t\rho_0\|_{L^\infty_+}\|\nabla {\Ta}^{\ell} \varphi_{0}\|_{L^2_+}^2+\|\D\vf_0\|_{L^\infty_+}\|\p_t {\Ta}^{\ell} \varphi_{0}\|_{L^2_+}^2\\
&\qquad+
\bigl(\| [{\Ta}^{\ell}; \na\vf_0]\nabla \partial_t\varphi_{0})\|_{L^2_+}+
\|\nabla\cdot([{\Ta}^{\ell};\r_0] \na \varphi_{0})\|_{L^2_+}\bigr)\|\partial_t {\Ta}^{\ell}\varphi_{0}\|_{L^2_+}\Bigr).
\end{split}
\end{equation}
Applying \eqref{S3eq10} gives
\beno\begin{split}
\sum_{\ell=1}^{s_0-1}\|\nabla\cdot([{\Ta}^{\ell};\r_0] \na \varphi_{0})\|_{L^2_+}\lesssim & \|\r_0-1\|_{W^{s_0}}\|\na\vf_0\|_{W^{s_0-1}}\\
\lesssim &\bigl(\|\p_t\vf_0\|_{W^{s_0}}+\|\na\vf_0\|_{W^{s_0}}^2\bigr)\|\na\vf_0\|_{W^{s_0-1}},
\end{split}
\eeno
which make us impossible to close the estimate in \eqref{outer-order-ell-12a}.
\end{rmk}

Now let us present  the proof of Proposition \ref{S2prop1}.

\begin{proof}[Proof of Proposition \ref{S2prop1}] We first deduce from \eqref{vf10} and \eqref{S3eq27} that \eqref{small-1} holds as long as $c$ is small enough in \eqref{S3eq27}. Then it follows from Theorem \ref{thm-hj-order-0} that (\ref{S3eq4}-\ref{outer-order-0-initial}) has a unique solution $\varphi_0$ on $[0, T_0]$ with $T_0={\cC}c_0^{-1}$ which satisfies \eqref{solution-order-0} when we take $c_0=c$ in \eqref{small-1}.
Moreover, we deduce from \eqref{S3eq2} and Theorem \ref{thm-hj-order-0} that
\begin{equation*}
1-\r_0=  \partial_t \varphi_{0}+\frac{1}{2} |\nabla \varphi_{0}|^2 \in W^{s_0-1}_{T_0},
\end{equation*}
and
\beq \label{S3eq21}
\begin{split}
&\r_0(t,x)-1\to 0\quad\mbox{as}\quad |x|\to \infty \andf\\
&\|(1-\r_0)\|_{W^{s_0-1}_{T_0}}\leq C_{s_0}\bigl(1+\|\na\vf_0\|_{W^{s_0-1}_{T_0}}\bigr)E^{\frac12}_{s_0}(t)\leq C_{s_0} \|(a_{0,0}^{\rm{\rm{in}}}-1,\,\na\vf_{0,0}^{\rm{in}})\|_{H^{s_0-1}}.
\end{split}
\eeq
Let us define $a_0\eqdefa \sqrt{\r_0}.$ Then we deduce from \eqref{S3eq21} that
\beno
\begin{split}
&a_0(t,x)-1=\frac{\r_0-1}{\sqrt{\r_0}+1}\to 0\quad\mbox{as}\quad |x|\to \infty \andf \\
&\|(1-a_0)\|_{W^{s_0-1}_{T_0}}\leq  C_{s_0} \|(a_{0,0}^{\rm{\rm{in}}}-1,\,\na\vf_{0,0}^{\rm{in}})\|_{H^{s_0-1}}.
\end{split}
\eeno
It is easy to observe that thus obtained $(a_0, \vf_0)$ is indeed the unique solution of (\ref{S2eq4}-\ref{outer-order-0-bc}).
Moreover there holds \eqref{S3eq28}. This completes the proof of Proposition \ref{S2prop1}.
\end{proof}

%%%%%%%%%%%%%%%%%%%%%%%%%%%%%%%%%%%%%%%%%%%%%%%%%%%%%%%%%%%%
\renewcommand{\theequation}{\thesection.\arabic{equation}}
\setcounter{equation}{0}
%%%%%%%%%%%%%%%%%%%%%%%%%%%%%%%%%%%%%%%%%%%%%%%%%%%%%%%%%%%%

\section{Solvability of the boundary layer equations (\ref{phi-negative-1}-\ref{a-order-0}) }\label{sec-blp}

The goal of  this section is to prove the existence of smooth solutions to  the boundary layer equations (\ref{phi-negative-1}-\ref{a-order-0}),
namely the proof of Proposition \ref{S2prop2}.

\begin{proof}[Proof of Proposition \ref{S2prop2}] Once again, we shall only present the {\it a priori} estimates. In view of \eqref{phi-negative-1}, we write
\begin{equation*}\label{phi-negative-1-1}
\begin{split}
\frac{1}{2}(A_0+\overline{a}_0)\partial_Z(\partial_Z \Phi_{1}+\overline{\partial_z \varphi}_{0})+\partial_Z (A_0+\overline{a}_0) (\partial_Z \Phi_{1}+ \overline{\partial_z \varphi}_{0})=0.
\end{split}
\end{equation*}
Multiplying the above equation by $A_0+\overline{a}_0$ yields
\begin{equation*}\label{phi-negative-1-3}
\begin{split}
\partial_Z\left((A_0+\overline{a}_0)^2(\partial_Z \Phi_{1}+\overline{\partial_z \varphi}_{0})\right)=0,
\end{split}
\end{equation*}
which together with the boundary conditions $A_0|_{Z=+\infty}=0=\partial_Z \Phi_1|_{Z=+\infty}$ ensures that
\begin{equation}\label{phi-negative-1-4}
\begin{split}
(A_0+\overline{a}_0)^2(\partial_Z \Phi_{1}+\overline{\partial_z \varphi}_{0})=(\overline{a}_0)^2 \overline{\partial_z \varphi}_{0}.
\end{split}
\end{equation}
On the other hand, we deduce from \eqref{a-order-0} that
\begin{equation*}\label{a-order-0-1}
\begin{split}
\frac{1}{2}\partial_Z^2A_0=&\frac{1}{2}(A_0+\overline{a}_0)(\partial_Z\Phi_{1}+\overline{\partial_z \varphi}_{0})^2\\
&-\frac{1}{2}(A_0+\overline{a}_0) (\overline{\partial_z \varphi}_{0})^2 +A_0(A_0+\overline{a}_0)(A_0+2\overline{a}_0).
\end{split}
\end{equation*}
Inserting \eqref{phi-negative-1-4} into the above equation leads to
\begin{equation}\label{a-order-0-2}
\begin{split}
\frac{1}{2}\partial_Z^2A_0=\frac{1}{2}(\overline{a}_0)^4 (\overline{\partial_z \varphi}_{0})^2(A_0+\overline{a}_0)^{-3}
-\frac{1}{2}(\overline{\partial_z \varphi}_{0})^2(A_0+\overline{a}_0) +A_0(A_0+\overline{a}_0)(A_0+2\overline{a}_0).
\end{split}
\end{equation}
Let us denote
\begin{equation}\label{denote-1}
\widetilde{A_0} \eqdefa A_0+\overline{a}_0 \andf q_0 \eqdefa \frac{d \widetilde{A_0}}{d Z}.
\end{equation}
Then under the assumption that $ A_0+\overline{a}_0>0,$ (which we shall justify below),  one has
\beno  \partial_Z^2 A_0=q_0 \partial_{\widetilde{A_0}} q_0, \eeno
and  it follows from \eqref{a-order-0-2} that
\begin{equation*}\label{a-order-0-3}
\begin{split}
\frac{1}{4}\frac{d\,q_0^2}{d\, \widetilde{A_0}} =\frac{1}{2}(\overline{a}_0)^4 (\overline{\partial_z \varphi}_{0})^2\widetilde{A_0}^{-3}-\frac{1}{2}(\overline{\partial_z \varphi}_{0})^2\widetilde{A_0} +\widetilde{A_0}\bigl(\widetilde{A_0}^2-(\overline{a}_0)^2\bigr),
\end{split}
\end{equation*}
from which, we infer
\begin{equation}\label{a-order-0-4}
\begin{split}
q_0^2 =-(\overline{a}_0)^4 (\overline{\partial_z \varphi}_{0})^2 \widetilde{A_0}^{-2}-(\overline{\partial_z \varphi}_{0})^2\widetilde{A_0}^2 +\widetilde{A_0}^4-2(\overline{a}_0)^2\widetilde{A_0}^2+C_1(t,y).
\end{split}
\end{equation}
Thanks to the conditions $\frac{d A_0}{d Z}|_{Z=+\infty}=A_0|_{Z=+\infty}=0$, one gets
\begin{equation*}\label{a-order-0-6}
\begin{split}
0 =-(\overline{a}_0)^4 (\overline{\partial_z \varphi}_{0})^2 (\overline{a}_0)^{-2}-\bigl(2(\overline{a}_0)^2+(\overline{\partial_z \varphi}_{0})^2\bigr)(\overline{a}_0)^2 +(\overline{a}_0)^4+C_1(t,y),
\end{split}
\end{equation*}
which gives
\begin{equation*}\label{a-order-0-7}
\begin{split}
C_1(t,y) =(\overline{a}_0)^2 \left((\overline{a}_0)^2 + 2(\overline{\partial_z \varphi}_{0})^2\right).
\end{split}
\end{equation*}
By inserting the above equality into \eqref{a-order-0-4} and
multiplying the resulting equality  by $\widetilde{A_0}^2,$  we find
\begin{equation*}\label{a-order-0-8}
\begin{split}
\left(\widetilde{A_0}\frac{d \widetilde{A_0}}{d Z}\right)^2= &\widetilde{A_0}^6-\bigl(2(\overline{a}_0)^2+(\overline{\partial_z \varphi}_{0})^2\bigr)\widetilde{A_0}^4\\
&+(\overline{a}_0)^2 \bigl((\overline{a}_0)^2 + 2(\overline{\partial_z \varphi}_{0})^2\bigr)\widetilde{A_0}^2-(\overline{a}_0)^4 (\overline{\partial_z \varphi}_{0})^2,
\end{split}
\end{equation*}
that is
\begin{equation}\label{a-order-0-9}
\begin{split}
\frac{1}{4}\left(\frac{d \widetilde{A_0}^2}{d Z}\right)^2= \bigl(\widetilde{A_0}^2-(\overline{a}_0)^2\bigr)^2
 \bigl(\widetilde{A_0}^2-(\overline{\partial_z \varphi}_{0}\bigr)^2).
\end{split}
\end{equation}
Notice that according to Proposition \ref{S2prop1}, we have $$(\overline{a}_0)^2-(\overline{\partial_z \varphi}_{0})^2\geq \frac14$$
as long as $c$ is small enough in \eqref{S3eq27}.

Let us denote
\begin{equation}\label{denote-2}
h_0 \eqdefa \left((\overline{a}_0)^2-(\overline{\partial_z \varphi}_{0})^2\right)^{\frac{1}{2}}\geq\frac12 \andf B_0\eqdefa \left(\widetilde{A_0}^2-(\overline{\partial_z \varphi}_{0})^2\right)^{\frac12}
\end{equation}
in case of $\widetilde{A_0}^2-(\overline{\partial_z \varphi}_{0})^2>0,$ which we shall justify later on.

By virtue of \eqref{a-order-0-9} and \eqref{denote-2}, we write
\begin{equation*}\label{a-order-0-10}
\begin{split}
\frac{1}{4}\left(\frac{d B_0^2}{d Z}\right)^2= \bigl(B_0^2-h_0^2\bigr)^2 B_0^2,
\end{split}
\end{equation*}
that is,
\begin{equation}\label{a-order-0-11}
\begin{split}
\frac{d B_0}{d Z}=\pm \bigl(B_0^2-h_0^2\bigr),
\end{split}
\end{equation}
from which, we infer
\begin{equation*}\label{a-order-0-13}
\begin{split}
\frac{B_0-h_0}{B_0+h_0}=C_2(t,y) e^{\pm 2h_0 Z}.
\end{split}
\end{equation*}
Since $B_0-h_0$ is rapidly decaying to zero  as $Z \rightarrow +\infty$, we have
\begin{equation*}\label{a-order-0-14}
\begin{split}
\frac{B_0-h_0}{B_0+h_0}=C_2(t,y) e^{- 2h_0 Z},
\end{split}
\end{equation*}
which gives
\begin{equation}\label{a-order-0-15}
\begin{split}
B_0=h_0 \frac{1+C_2(t,y)e^{- 2h_0 Z}}{1-C_2(t,y)e^{- 2h_0 Z}}.
\end{split}
\end{equation}
While according to the boundary condition \eqref{a-order-0-bc}, one has
\begin{equation*}\label{a-order-0-16}
\begin{split}
B_0^2|_{Z=0}=(A_0+\overline{a}_0)^2|_{Z=0}-  (\overline{\partial_z \varphi}_{0})^2=1-(\overline{\partial_z \varphi}_{0})^2,
\end{split}
\end{equation*}
which together with \eqref{a-order-0-15} ensures that
\begin{equation*}\label{a-order-0-17}
\begin{split}
\sqrt{1-(\overline{\partial_z \varphi}_{0})^2}=h_0 \frac{1+C_2(t,y)}{1-C_2(t,y)}.
\end{split}
\end{equation*}
As a result, it comes out
\begin{equation*}
\begin{split}
C_2(t,y)=&\frac{\sqrt{1-(\overline{\partial_z \varphi}_{0})^2}-h_0}{\sqrt{1-(\overline{\partial_z \varphi}_{0})^2}+h_0}\\
=&\frac{1-(\overline{a}_0)^2}{\bigl(\sqrt{1-(\overline{\partial_z \varphi}_{0})^2}+h_0\bigr)^2},
\end{split}
\end{equation*}
from which,  Proposition \ref{S2prop1} and  trace theorem, we deduce that
\beq \label{a-order-0-18}
\|C_2(\cdot)\|_{W^{s_0-\frac32}_{T_0}(\R^2)}\leq C\|(a_{0,0}^{\rm{\rm{in}}}-1,\,\na\vf_{0,0}^{\rm{in}})\|_{H^{s_0-1}}\leq Cc.
\eeq

Whereas it follows from   \eqref{denote-2} and \eqref{a-order-0-15} that
\begin{equation*}\label{a-order-0-19}
\begin{split}
(A_0+\overline{a}_0)^2-(\overline{\partial_z \varphi}_{0})^2=h_0^2 \Bigl(\frac{1+C_2e^{- 2h_0 Z}}{1-C_2e^{- 2h_0 Z}}\Bigr)^2,
\end{split}
\end{equation*}
which implies that
\begin{equation*}
\begin{split}
A_0&=-\overline{a}_0+\bigg( \frac{(\overline{a}_0)^2(1+C_2^2e^{- 4h_0 Z} )+2C_2(h_0^2-(\overline{\partial_z \varphi}_{0})^2)e^{- 2h_0 Z}}{(1-C_2(t,y)e^{- 2h_0 Z})^2} \bigg)^{\frac{1}{2}}\\
&=4C_2h_0^2e^{-2h_0Z}\bigg(\overline{a}_0+\Bigl( \frac{(\overline{a}_0)^2(1+C_2^2e^{- 4h_0 Z} )+2C_2
(h_0^2-(\overline{\partial_z \varphi}_{0})^2)e^{- 2h_0 Z}}{(1-C_2e^{- 2h_0 Z})^2} \Bigr)^{\frac{1}{2}}\bigg)^{-1}.
\end{split}
\end{equation*}
This together with \eqref{denote-2} and \eqref{a-order-0-18} in particular shows that $A_0\in \mathcal{W}^{s_0-\frac32}_{1, T_0},$ and
there exists some positive constant $c_1$ such that
\begin{equation}\label{a-away-0} \|A_0\|_{\mathcal{W}^{s_0-\frac32}_{1, T_0}}\leq C\|(a_{0,0}^{\rm{\rm{in}}}-1,\,\na\vf_{0,0}^{\rm{in}})\|_{H^{s_0-1}} \andf
A_0+\overline{a}_0 \geq c_1>0.
\end{equation}
Then we rigorously justify that $B_0>0$ for $Z\in\R^+$ as long as $c$ is small enough in \eqref{S3eq27}.

With such $A_0$, in view of \eqref{phi-negative-1-4}, we write
\begin{equation*}\label{phi-negative-1-4-a}
\begin{cases}
&\partial_Z \Phi_{1}=(A_0+\overline{a}_0)^{-2}(\overline{a}_0)^2 \overline{\partial_z \varphi}_{0}-\overline{\partial_z \varphi}_{0},\\
&\Phi_{1}|_{Z=+\infty}=0.
\end{cases}
\end{equation*}
It is easy to observe from the above equation that
\begin{equation*}
\Phi_{1}(Z)=-\overline{\partial_z \varphi}_{0} \,\int_{Z}^{\infty}\frac{A_0(A_0+2\overline{a}_0)}{(A_0+\overline{a}_0)^{2}}\, dZ.
\end{equation*}
which together with \eqref{a-away-0} shows that
\beq \label{phi-negative-1-4-b}
 \|\Phi_1\|_{\mathcal{W}^{s_0-\frac32}_{1, T_0}}\leq C\|(a_{0,0}^{\rm{\rm{in}}}-1,\,\na\vf_{0,0}^{\rm{in}})\|_{H^{s_0-1}}^2.
 \eeq
This ends the proof of Proposition \ref{S2prop2}.
\end{proof}

%%%%%%%%%%%%%%%%%%%%%%%%%%%%%%%%%%%%%%%%%%%%%%%%%%%%%%%%%%%%
\renewcommand{\theequation}{\thesection.\arabic{equation}}
\setcounter{equation}{0}
%%%%%%%%%%%%%%%%%%%%%%%%%%%%%%%%%%%%%%%%%%%%%%%%%%%%%%%%%%%%

\section{The existence of solutions to a linear wave  equation}\label{sec8}

The goal of this section is to present the proof of Theorem \ref{S8thm1}.
Let $\chi(\tau)\in C_c^\infty(\R)$ with $\chi(\tau)=1$ in a neighborhood of $0.$ We denote
\beq \label{S8eq3}
G\eqdefa \chi\left(z(1+|D_\h|^2)^{\frac12}\right)g(t,\cdot) \andf \vf=\phi+G.
\eeq
Then one has $G\in W^{s+1}_T(\R^3),$ and $\phi$ verifies
\begin{equation}\label{S8eq4}
\begin{cases}
& P(\vf_0,D)\phi=-P(\vf_0,D)G+f\eqdefa F\in W^{s-1}_T(\R^3_+),\\
&  \vf|_{z=0}=0,\\
&\phi|_{t=0}=\vf_{,0}^{\rm in}-G|_{t=0}\eqdefa \phi_{0}^{\rm in}, \andf \p_t\vf|_{t=0}=\vf_{,1}^{\rm in}-\p_tG|_{t=0}\eqdefa \phi_{1}^{\rm in}.
\end{cases}
\end{equation}
And the proof of Theorem \ref{S8thm1} is reduced to the following one:

\begin{thm}\label{thm-hlin-order}
{\sl  Let $T\leq T_0$ and $4\leq$  be an integer. Let $F\in W^{s-1}_T$ and $(\phi_0^{\rm{in}},\phi_1^{\rm{\rm{in}}})$ satisfy
$\na\phi_0^{\rm{in}}, \phi_1^{\rm{\rm{in}}}\in \,H^{s-1}$  and the compatibility condition that $\p_t^\ell\phi(0,y,0)=0$ for $\ell=0, \cdots,s-1.$  Then under the assumptions
of Proposition \ref{S2prop1},
\eqref{S8eq4} has  a unique solution $\phi$ on $[0, T],$
which satisfies
\beq \label{S8eq8}
\left\|\left(\p_t\phi,\na\phi\right)\right\|_{W_T^{s-1}}\leq  C\bigl(\|g\|_{ W^{s+\frac12}_T(\R^2)}+\|(\na \phi_0^{\rm{in}},\,\phi_1^{\rm{in}})\|_{H^{s-1}}+\|F\|_{W^{s-1}_T}\bigr).
\eeq
}
\end{thm}

In what follows, we shall always denote $E_s(\phi)$, ${E}_{s, \text{tan}}(\phi)$, and $\widetilde{{E}}_{s, \text{tan}}(\phi)$ to be the energy functionals determined by  \eqref{S3eq11}.

Let us separate the proof of Theorem \ref{thm-hlin-order} into the following lemmas:

\begin{lem}\label{S8lem1}
{\sl Let $\phi$ be a smooth enough solution of \eqref{S8eq4} on $[0,T].$ Then for $t\leq T,$ one has
\beq \label{S8eq5}
\begin{split}
&\frac{d}{dt}\int_{\R^3_+}\left((\p_t\phi)^2+\r_0|\na\phi|^2-(\na\vf_0\cdot\na\phi)^2\right)\,dx\leq C\left(1+ E_4(\vf_0)\right)E_1(\phi(t))
+\|F\|_{L^2_+}^2.
\end{split}
\eeq}
\end{lem}

\begin{proof} By taking $L^2$ inner product of the $\phi$ equation of \eqref{S8eq4} with $\p_t\phi$ and
using integration by parts, one has
\beq \label{S8PE}
\begin{split}
\frac12\frac{d}{dt}&\int_{\R^3_+}\left((\p_t\phi)^2+\r_0|\na\phi|^2-(\na\vf_0\cdot\na\phi)^2\right)\,dx\\
=&\int_{\R^3_+}\left(\frac12\p_t\r_0|\na\phi|^2-(\na\vf_0\cdot\na\phi) (\na\p_t\vf_0\cdot\na\phi)+(\na\p_t\vf_0\cdot\na\phi) \p_t\phi+F \p_t\phi\right)\,dx\\
\leq & C\left(1+\|\p_t^2\vf_0\|_{L^\infty_+}+\bigl(1+\|\na\vf_0\|_{L^\infty_+}\bigr)\|\na\p_t\vf_0\|_{L^\infty_+}\right)
\bigl(\|\p_t\phi\|_{L^2_+}^2+\|\na\phi\|_{L^2_+}^2\bigr)+\|F\|_{L^2_+}^2,
\end{split}
\eeq
from which,  and \eqref{S3eq11},  we deduce \eqref{S8eq5}.
\end{proof}

\begin{lem}[High-order tangential derivatives estimates]\label{S8lem3}
{\sl Let $\phi$ be a smooth enough solution of \eqref{S8eq4} on $[0,T].$ Then for $t\leq T,$ one has
\begin{equation}\label{S8ht-re-106}
\begin{split}
&\frac{d}{dt} \widetilde{E}_{s, \text{tan}}(\phi(t)) \lesssim \bigl(1+E_{s+1}(\vf_0(t))\bigr) E_{s}(\phi(t))+\sum_{\ell=0}^{s-1}\|{\Ta}^{\ell}\,F\|_{L^2_+}^2.
\end{split}
\end{equation}}
\end{lem}
\begin{proof}
Let $\ell\leq s-1$ be an integer. By applying the operator ${\Ta}^{\ell} $ (with ${\Ta}=(\partial_t, \nabla_{\h})$ and $\Ta^\ell=\p_t^{\al_1}\na_\h^{\al_2}$ for  $\al_1+|\al_2|=\ell \in \mathbb{N}$) to the $\phi$ equation in \eqref{S8eq4}, we find
\beq \label{S8eq-ht-1}
\begin{split}
 P(\vf_0, D)\Ta^{\ell}\phi=-[\Ta^\ell; P(\vf_0, D)]\phi+ \Ta^{\ell}\,F,
\end{split}
\eeq
where
\beno \label{S8ht-re}
\begin{split}
[{\Ta}^\ell; P(\vf_0, D)]\phi=&- \dive \bigl([{\Ta}^\ell;\r_0]\nabla\phi\bigr)+2 [{\Ta}^\ell; \nabla \varphi_{0}] \cdot \nabla \partial_t\phi \\
 &+\sum_{k=1}^3\dive \bigl([{\Ta}^\ell;\p_k\varphi_0 \nabla\varphi_0 ]\p_k\phi\bigr)+[{\Ta}^\ell; \na\partial_t\varphi_0] \nabla\phi+[{\Ta}^\ell,\D\vf_0]\p_t\phi.
\end{split}
\eeno
Due to $\p_t{\Ta}^{\ell}\phi|_{z=0}=0,$ by
taking the $L^2$ inner product of \eqref{S8eq-ht-1} with $\partial_t{\Ta}^{\ell} \phi$, we deduce  from \eqref{S8PE} that
\begin{equation}\label{hlin-thh-2}
\begin{split}
&\frac{d}{dt}\int_{\R^3_+}\left((\p_t{\Ta}^{\ell}\phi)^2+\r_0|\na{\Ta}^{\ell}\phi|^2
-(\na\vf_0\cdot\na{\Ta}^{\ell}\phi)^2\right)\,dx\\
&\lesssim \bigl(\|\p_t\r_0\|_{L^\infty_+}+\|\na\vf_0 \|_{L^\infty_+} \|\na\p_t\vf_0 \|_{L^\infty_+}\bigr) \|\na{\Ta}^{\ell}\phi\|_{L^2_+}^2
\\
 &\quad\ +\bigl(\|\na\p_t\vf_0\|_{L^\infty_+}\|\na{\Ta}^{\ell}\phi\|_{L^2_+}+\|[\Ta^\ell; P(\vf_0, D)]\phi\|_{L^2_+}+\|{\Ta}^{\ell}\,F\|_{L^2_+}\bigr)\|\p_t{\Ta}^{\ell}\phi\|_{L^2_+}.
\end{split}
\end{equation}
It follows from \eqref{S3eq10} that
\beno
\begin{split}
&\sum_{\ell=1}^{s-1}\|\dive \bigl([{\Ta}^\ell;\r_0]\nabla \phi\bigr)\|_{L^2_+} \lesssim \|\r_0-1\|_{W^{s}}\|\nabla \phi\|_{W^{s-1}},\\
&\sum_{\ell=1}^{s-1}\|\dive \bigl([{\Ta}^\ell;\p_k\varphi_0 \nabla\varphi_0 ]\p_k\phi\bigr)\|_{L^2_+} \lesssim \|\nabla\varphi_0 \|_{W^{s}}^2\|\nabla\phi\|_{W^{s-1}}.
\end{split}
\eeno
For $s\geq 4,$ we get, by applying Lemma \ref{S3lem1}, that
\beno
\begin{split}
&\sum_{\ell=1}^{s-1}\|[{\Ta}^\ell; \nabla \varphi_{0}] \cdot \nabla \partial_t\phi \|_{L^2_+} \lesssim \|\nabla \varphi_{0}\|_{W^{s-1}}\|\nabla \partial_t\phi\|_{W^{s-2}},\\
&\sum_{\ell=1}^{s-1}\|[{\Ta}^\ell; \na\partial_t\varphi_0] \nabla\phi\|_{L^2_+} \lesssim \|\nabla \partial_t\varphi_{0}\|_{W^{s-1}}\|\nabla \phi\|_{W^{s-2}},\\
&\sum_{\ell=1}^{s-1}\|[{\Ta}^\ell,\D\vf_0]\p_t\phi\|_{L^2_+} \lesssim \|\D\vf_0\|_{W^{s-1}}\|\p_t\phi\|_{W^{s-2}}.
\end{split}
\eeno
This gives rise to
\beno
\begin{split}
 &\sum_{\ell=1}^{s-1}\|[{\Ta}^\ell; P(\vf_0, D)]\phi\|_{L^2_+} \lesssim E_{1+s}^{\frac{1}{2}}(\varphi_0(t))(1+E_{1+s}^{\frac{1}{2}}(\varphi_0(t)))E_{s}^{\frac{1}{2}}(\phi(t)).
\end{split}
\eeno
Inserting the above estimate into \eqref{hlin-thh-2} gives \eqref{S8ht-re-106}. This completes the proof of the lemma.
\end{proof}

\begin{lem}[Full energy estimates]\label{S8lem4}
{\sl Let $\phi$ be a smooth enough solution of \eqref{S8eq4} on $[0,T].$ Then for $t\leq T,$ one has
\begin{equation}\label{full-hell-0}
\begin{split}
&\frac{d}{dt} \widetilde{E}_{s, \text{tan}}(\phi(t)) \lesssim \bigl(1+E_{s+1}(\vf_0(t))\bigr)  {E}_{s, \text{tan}}(\phi(t)) +\|F\|_{W^{s-1}}^2.
\end{split}
\end{equation}}
\end{lem}
\begin{proof}

Let $E_s(\phi)$, ${E}_{s, \text{tan}}(\phi)$, and $\widetilde{{E}}_{s, \text{tan}}(\phi)$ be determined by \eqref{S3eq11}.
Along the same line to the proof of \eqref{S3eq12},  we claim that
\beq \label{full-hell-1}
E_\ell(\phi(t))\leq C_\ell \bigl({E}_{\ell, \text{tan}}(\phi(t))+\|F(t)\|_{W^{\ell-2}}^2\bigr)\quad \mbox{for}\quad t\leq T \andf \ell=2,\cdots,s.
\eeq
In what follows, we just outline the proof.

We first observe from  \eqref{S8eq4} that
 \beq\label{S8eq-vert-1}
\begin{split}
(\r_0-(\partial_3\varphi_0)^2)\partial_3^2\phi=&\partial_t^2 \phi-\r_0\D_{\h}\phi-\na \r_0\cdot\nabla\phi+2 \nabla \varphi_{0} \cdot \nabla \partial_t\phi+(\nabla\varphi_0 \cdot \nabla\phi)\D\varphi_0\\
 &+\D\vf_0\p_t\phi +\sum_{j=1}^3(\nabla\p_j\varphi_0 \cdot \nabla\phi)\p_j\varphi_0+ (\nabla_{\h}\varphi_0 \cdot \partial_3\nabla_{\h}\phi)\partial_3\varphi_0 \\
 &+ (\nabla\varphi_0 \cdot \na_{\h}\nabla\phi)\nabla_{\h}\varphi_0 +\na\partial_t \varphi_0 \cdot\nabla\phi-F,
\end{split}
\eeq
from which, we infer
 \beno
\begin{split}
&E_{2}^{\frac{1}{2}}(\phi(t))\lesssim E_{2, \text{tan}}^{\frac{1}{2}}(\phi(t))+E_{4}^{\frac{1}{2}}(\varphi_0(t))\bigl(1+E_{4}^{\frac{1}{2}}(\varphi_0(t))\bigr)
E_{2}^{\frac{1}{2}}(\phi(t))+\|F\|_{L^2_+}.
\end{split}
\eeno
This together with \eqref{S3eq28} ensures that
 \beq\label{S8eq-vert-2}
\begin{split}
&E_{2}(\phi(t))\lesssim E_{2, \text{tan}}(\phi(t))+\|F\|_{L^2_+}^2.
\end{split}
\eeq
This proves \eqref{full-hell-1} for $\ell=2.$

Now we assume that \eqref{full-hell-1} holds for $\ell=k,$ we are going to prove that \eqref{full-hell-1} holds for
$\ell=k+1\leq s.$ We first notice that
\beno
E_{k+1}(\phi(t))\leq \sum_{j=0}^k\bigl(\|\p_t^{j+1}\phi(t)\|_{H^{k-j}}^2+\|\na\p_t^j\phi(t)\|_{H^{k-j}}^2\bigr).
\eeno

\noindent$\bullet$\underline{
When $j=k.$}

We observe from  \eqref{S3eq11} that
\beno
\|\p_t^{k+1}\phi(t)\|_{L^2_+}^2+\|\na \p_t^{k}\phi(t)\|_{L^2_+}^2\leq {E}_{k+1, \text{tan}}(\phi(t)).
\eeno

\noindent$\bullet$\underline{
When $j=k-1.$}

It follows from  \eqref{S3eq11} that
\beno
\|\p_t^{k}\phi\|_{H^1}^2=\|\p_t^{k}\phi\|_{L^2_+}^2+\|\na \p_t^{k}\phi\|_{L^2_+}^2\leq {E}_{k+1, \text{tan}}(\phi(t)).
\eeno
Whereas notice that
\beno
\begin{split}
\|\na\p_t^{k-1}\phi(t)\|_{H^1}^2=&\|\na\p_t^{k-1}\phi(t)\|_{L^2_+}^2+\|\na^2\p_t^{k-1}\phi(t)\|_{L^2_+}^2\\
\lesssim &\|\p_t\phi\|_{W^{k-1}}^2+\|\na\na_\h\p_t^{k-1}\phi(t)\|_{L^2_+}^2+\|\p_3^2\p_t^{k-1}\phi(t)\|_{L^2_+}^2.
\end{split}
\eeno
Similar to the proofs of \eqref{S3eq19} and \eqref{S3eq20}, we deduce from \eqref{S8eq-vert-1} that
\beno
\|\p_3^2\p_t^{k-1}\phi\|_{L^2_+}\leq C\bigl({E}^{\frac12}_{k+1, \text{tan}}(\phi(t))+\bigl(1+{E}^{\frac12}_{k+1}(\vf_0(t))\bigr){E}^{\frac12}_{k+1}(\vf_0(t)) {E}^{\frac12}_{k+1}(\phi(t))\bigr)+
\|\p_t^{k-1}F\|_{L^2_+}^2.
\eeno
$\|\na\p_t^{k-1}\phi(t)\|_{H^1}$ shares the same estimate.

As a result, it comes out
\beno
\begin{split}
\|\p_t^{k}\phi\|_{H^1}+\|\na\p_t^{k-1}\phi(t)\|_{H^1}\leq C\bigl({E}^{\frac12}_{k+1, \text{tan}}(\phi(t))+{E}^{\frac12}_{k+1}(\vf_0(t)) {E}^{\frac12}_{k+1}(\phi(t))\bigr)+\|\p_t^{k-1}F\|_{L^2_+}^2.
\end{split}
\eeno

\noindent$\bullet$\underline{
When $k-j\geq 2.$}

We have
\beno
\begin{split}
\|\p_t^{j+1}\phi\|_{H^{k-j}}^2=&\|\p_t^{j+1}\phi\|_{H^{k-j-1}}^2+\|\na^2\p_t^{j+1}\phi\|_{H^{k-j-2}}^2\\
\lesssim &\|\p_t\phi\|_{W^{k-1}}^2+\|\na\na_\h\p_t^{j+1}\phi\|_{H^{k-j-2}}^2+\|\p_3^2\p_t^{j+1}\phi\|_{H^{k-j-2}}^2.
\end{split}
\eeno
By virtue of \eqref{S8eq-vert-1}, we find
 \beq\label{full-hell-101}
\begin{split}
&\|\p_3^2\p_t^{j+1}\phi\|_{H^{k-j-2}}\lesssim\|\p_t^{j+3}\phi\|_{H^{k-j-2}}+ \|\p_t\phi\|_{W^{k-1}}+\|\na\na_\h\p_t^{j+1}\phi\|_{H^{k-j-2}}\\
&+\|\p_t^{j+1}\D_\h\phi\|_{H^{k-j-2}}
+\bigl(1+{E}^{\frac12}_{k+1}(\vf_0(t))\bigr){E}^{\frac12}_{k+1}(\vf_0(t)) E_{k+1}^{\frac12}(\phi(t))+\|\p_t^{j+1}F\|_{H^{k-j-2}}.
\end{split}
\eeq
The same estimate holds for $\|\p_t^{j+1}\phi\|_{H^{k-j}}.$

In the case when $k-j\geq 3,$ we have
\beno
\begin{split}
\|\na\na_\h\p_t^{j+1}\phi\|_{H^{k-j-2}}^2=&\|\na\na_\h\p_t^{j+1}\phi\|_{H^{k-j-3}}^2
+\|\na^2\na_\h\p_t^{j+1}\phi\|_{H^{k-j-3}}^2\\
\lesssim & \|\p_t\phi\|_{W^{k-1}}^2+\|\na\na_\h^2\p_t^{j+1}\phi\|_{H^{k-j-3}}^2+\|\p_3^2\na_\h\p_t^{j+1}\phi\|_{H^{k-j-3}}^2.
\end{split}
\eeno
Yet it follows from \eqref{S8eq-vert-1} and Lemma \ref{S3lem1} that
\beno
\begin{split}
\|\p_3^2\na_\h\p_t^{j+1}\phi\|_{H^{k-j-3}}\lesssim&  \|\na_\h\p_t^{j+3}\phi\|_{H^{k-j-3}}+\|\p_t^{j+1}\na^3_\h\phi\|_{H^{k-j-3}}\\
&+\bigl(1+{E}^{\frac12}_{k+1}(\vf_0(t))\bigr){E}^{\frac12}_{k+1}(\vf_0(t))  E_{k+1}^{\frac12}(\phi(t))+\|\na_\h\p_t^{j+1}F\|_{H^{k-j-3}}.
\end{split}
\eeno
Inserting the above estimate into \eqref{full-hell-101} gives rise to
\beno
\begin{split}
\|\p_t^{j+1}\phi\|_{H^{k-j}}\lesssim & \|\p_t\phi\|_{W^{k-1}}+\|\p_t^{j+3}\phi\|_{H^{k-j-2}}+\|\p_t^{j+1}\D_\h\phi\|_{H^{k-j-2}}\\
&+\|\na\na_\h^2\p_t^{j+1}\phi\|_{H^{k-j-3}}+\|\na_\h\p_t^{j+3}\phi\|_{H^{k-j-3}}+\|\p_t^{j+1}\na^3_\h\phi\|_{H^{k-j-3}}\\
&+C\bigl(1+{E}^{\frac12}_{k+1}(\vf_0(t))\bigr){E}^{\frac12}_{k+1}(\vf_0(t))  E_{k+1}^{\frac12}(\phi(t))+\|F\|_{W^{k-1}}.
\end{split}
\eeno
By finite steps of iteration and using the inductive assumption for $\ell=k,$ we deduce that
\beno
\|\p_t^{j+1}\phi\|_{H^{k-j}}\leq C_k\bigl({E}_{k+1, \text{tan}}^{\frac12}(\phi(t))+\|F\|_{W^{k-1}}\bigr)+C{E}^{\frac12}_{k+1}(\vf_0(t))  E_{k+1}^{\frac12}(\phi(t)).
\eeno
The same estimate holds for $\|\na\p_t^j\phi(t)\|_{H^{k-j}}.$

Therefore, by virtue of \eqref{S3eq28}, we conclude that
\beno
E_{k+1}(t)\leq C_k\bigl({E}_{k+1, \text{tan}}(\phi(t))+\|F\|_{W^{k-1}}^2\bigr).
\eeno
This proves \eqref{full-hell-1} for $\ell=k+1.$ By combining \eqref{S8ht-re-106} with \eqref{full-hell-1},
 we achieve \eqref{full-hell-0}. This completes the proof of Lemma \ref{S8lem4}.
\end{proof}

\begin{proof}[Proof of Theorem \ref{thm-hlin-order}]
By applying Gronwall's inequality to \eqref{full-hell-0} and using \eqref{full-hell-1}, we deduce \eqref{S8eq8}.
This completes the proof of Theorem \ref{thm-hlin-order}.
\end{proof}

%%%%%%%%%%%%%%%%%%%%%%%%%%%%%%%%%%%%%%%%%%%%%%%%%%%%%%%%%%%%
\renewcommand{\theequation}{\thesection.\arabic{equation}}
\setcounter{equation}{0}
%%%%%%%%%%%%%%%%%%%%%%%%%%%%%%%%%%%%%%%%%%%%%%%%%%%%%%%%%%%%

\section{The existence of solutions to the other asymptotic equations}\label{sec-asy}

Let us first present the proof of Proposition \ref{S2prop3}.

\begin{proof}[Proof of Proposition \ref{S2prop3}] It follows from \eqref{varphi-order-0-bc} and Proposition \ref{S2prop2} that
\beno
\vf_1|_{z=0}=-\Phi_1|_{Z=0}\in W^{s_0-\frac32}_{T_0}(\R^2).
\eeno
Then we deduce from Theorem \ref{S8thm1} that
the wave equation \eqref{outer-conti-1-2} with boundary condition  \eqref{varphi-order-0-bc} and initial condition \eqref{varphi-1-initial-c}
has a unique solution $\vf_1$ on $[0,T_0].$ Furthermore, we have
\beq \label{S5eq5}
\begin{split}
\bigl\|(\p_t\vf_1,\na\vf_1)\|_{W^{s_0-3}_{T_0}}\leq& C\bigl(\bigl\|\bigl( \na\vf_{1,0}^{\rm in}, \vf_{1,0}^{\rm in},\bigr)\bigr\|_{ H^{s_0-3}}
+\bigl\|\Phi_1|_{z=0}\bigr\|_{ W^{s_0-\frac32}_T}\bigr)\\
\leq & C\bigl(\bigl\|\bigl(a_{1,0}^{\rm in}, \na\vf_{1,0}^{\rm in}\bigr)\bigr\|_{ H^{s_0-3}}
+\bigl\|\bigl(a_{0,0}^{\rm in}, \na\vf_{0,0}^{\rm in}\bigr)\bigr\|_{ H^{s_0-1}}\bigr). \end{split}
\eeq
Note from the $\vf_1$ equation of \eqref{outer-order-1} that \beno
a_1=-\frac1{2a_0}\left(\p_t\vf_1+\na\vf_0\cdot\na\vf_1\right),
\eeno
from which, \eqref{S5eq5} and Lemma \ref{S3lem1}, we deduce \eqref{S5eq4}.
This completes the proof of the proposition.
\end{proof}

Next let us present the proof of Proposition \ref{S2prop4}.

\begin{proof}[Proof of Proposition \ref{S2prop4}] Inductively, we assume that we already have
\beq \label{S5eq11ab}
\begin{split}
&(a_0-1, \na\vf_0)\in W^{s_0-1}_{T_0},\quad
(a_{j+1},\na\vf_{j+1})\in W^{s_0-1-2(j+1)}_{T_0} \andf\\
& (A_j,\Phi_{j+1})\in W^{s_0-2(j+1)+\frac12}_{1,T_0}\quad \mbox{ for}\quad j=0,\cdots, k,
\end{split}
\eeq
 we consider
the boundary layer problem (\ref{phi-order-m}-\ref{a-order-m1}) with boundary condition \eqref{a-order-m1-bc}.
We first get, by inserting \eqref{phi-negative-1-4} into \eqref{a-order-m1}, that
\begin{equation}\label{a-order-m1-1}
\begin{split}
\frac{1}{2}\partial_Z^2A_{k+1}=&\frac{\overline{a}_0^2 \overline{\partial_z \varphi}_{0}}{A_0+\overline{a}_0}  \partial_Z \Phi_{k+2}+G_k\\
&+\bigl(3A_0^2+6 \overline{a}_0 A_0+2 \overline{a}_0^2+\partial_Z\Phi_{1}\overline{\partial_z \varphi}_{0}+\frac{1}{2}|\partial_Z\Phi_{1}|^2\bigr) A_{k+1}.
\end{split}
\end{equation}
Here according to \eqref{S5eq11ab} and \eqref{app-u-bern-2} in the Appendix \ref{appA-re},  $G_k\in W^{s_{0,k}}_{1,T_0},$
where and in what follows, we always denote
$$s_{0, k}\eqdefa s_0-2(k+2)+\frac12.$$

Whereas by multiplying \eqref{phi-order-m} by $(A_0+\overline{a}_0),$ we find
\begin{equation*}\label{phi-order-m-1}
\begin{split}
\frac{1}{2}\partial_Z\bigl((A_0+\overline{a}_0)^2 \partial_Z\Phi_{k+2}\bigr) &+(A_0+\overline{a}_0)(\partial_Z \Phi_{1}+\overline{\partial_z \varphi}_{0}) \partial_Z A_{k+1}\\
&+\frac{1}{2} A_{k+1}(A_0+\overline{a}_0)\partial_Z^2 \Phi_{1} =(A_0+\overline{a}_0)\,F_k,
\end{split}
\end{equation*}
where  $F_k\in W^{s_{0,k}}_{1,T_0}$ according to
\eqref{S5eq11ab} and \eqref{app-u-conti-1} in the Appendix \ref{appA-re},

Then we deduce from  \eqref{phi-negative-1} that
\begin{equation*}\label{phi-order-m-2}
\begin{split}
\frac{1}{2}\partial_Z\bigl((A_0+\overline{a}_0)^2 \partial_Z\Phi_{k+2}\bigr) &+(A_0+\overline{a}_0)(\partial_Z \Phi_{1}+\overline{\partial_z \varphi}_{0}) \partial_Z A_{k+1}\\
&- A_{k+1}(\partial_Z \Phi_{1}+\overline{\partial_z \varphi}_{0}) \partial_Z A_0 =(A_0+\overline{a}_0)\,F_k.
\end{split}
\end{equation*}
By virtue of \eqref{phi-negative-1-4}, we write
\begin{equation*}\label{phi-order-m-3}
\begin{split}
\frac{1}{2}\partial_Z\bigl((A_0+\overline{a}_0)^2 \partial_Z\Phi_{k+2}\bigr) +(\overline{a}_0)^2 \overline{\partial_z \varphi}_{0} \Bigl(\frac{\partial_Z A_{k+1}}{ A_0+\overline{a}_0} -\frac{A_{k+1} \partial_Z A_0}{ (A_0+\overline{a}_0)^{2}}\Bigr) =(A_0+\overline{a}_0)\,F_k,
\end{split}
\end{equation*}
that is,
\begin{equation}\label{phi-order-m-4}
\begin{split}
\partial_Z\Bigl(\frac{1}{2}(A_0+\overline{a}_0)^2 \partial_Z\Phi_{k+2} +(\overline{a}_0)^2 \overline{\partial_z \varphi}_{0}\frac{ A_{k+1}}{ A_0+\overline{a}_0} \Bigr) =(A_0+\overline{a}_0)\,F_k.
\end{split}
\end{equation}
Integrating  \eqref{phi-order-m-4} over $[Z,\infty[$ gives rise to
\begin{equation}\label{phi-order-m-6}
\begin{split}
\partial_Z\Phi_{k+2} =-2(\overline{a}_0)^2 \overline{\partial_z \varphi}_{0}\frac{ A_{k+1}}{ (A_0+\overline{a}_0)^{3}}
+\frac2{ (A_0+\overline{a}_0)^{2}}\int_Z^\infty(A_0+\overline{a}_0)\,F_k\,dZ'.
\end{split}
\end{equation}

Plugging \eqref{phi-order-m-6} into \eqref{a-order-m1-1} leads to
\begin{equation}\label{a-order-m1-2}
\begin{split}
\partial_Z^2A_{k+1}=\fg A_{k+1}+\widetilde{G}_k,
\end{split}
\end{equation}
Here and in all that follows, we always denote
\begin{equation}\label{phi-order-m-ind-7-1}
\begin{split}
\fg \eqdefa&  6A_0^2+12 \overline{a}_0 A_0+4 \overline{a}_0^2+2\partial_Z\Phi_{1}\overline{\partial_z \varphi}_{0}+|\partial_Z\Phi_{1}|^2-4\frac{(\overline{a}_0)^4 (\overline{\partial_z \varphi}_{0})^2}{ (A_0+\overline{a}_0)^{4}},\\
\widetilde{G}_k\eqdefa &2{G}_k+ \frac{4\overline{a}_0^2 \overline{\partial_z \varphi}_{0}}{ (A_0+\overline{a}_0)^{3}}\int_Z^\infty(A_0+\overline{a}_0)\,F_k\,dZ' \in W^{s_{0,k}}_{1,T_0}.
\end{split}
\end{equation}

Recalling the notation from \eqref{bdyvalue-1} that $\bar{a}_{k+1}(y)={a}_{k+1}(y,0),$ we reduce the resolution of  the problem \eqref{phi-order-m}, \eqref{a-order-m1} and \eqref{a-order-m1-bc} to the following  system
\begin{equation}\label{phi-order-m-7}
\begin{cases}
&\partial_Z^2A_{k+1}=\fg A_{k+1}+\widetilde{G}_k,\\
&\partial_Z\Phi_{k+2} =-\frac{2(\overline{a}_0)^2 \overline{\partial_z \varphi}_{0}}{ (A_0+\overline{a}_0)^{3}}  A_{k+1} +\frac2{ (A_0+\overline{a}_0)^{2}}\int_Z^\infty(A_0+\overline{a}_0)\,F_k\,dZ',\\
& A_{k+1}|_{Z=0}=-\bar{a}_{k+1},\quad A_{k+1}|_{Z=+\infty}=0, \quad \Phi_{k+2}|_{Z=+\infty}=0.
\end{cases}
\end{equation}

Let's now handle the system \eqref{phi-order-m-7}. In order to do it, let $
 \widetilde{A}_{k+1}\eqdefa A_{k+1}+e^{-3Z}\bar{a}_{k+1}.
$
Then $\widetilde{A}_{k+1}$ verifies
\begin{equation}\label{phi-order-m-ind-7-1a}
\begin{cases}
&\partial_Z^2\widetilde{A}_{k+1}=\fg\, \widetilde{A}_{k+1}+\widetilde{G}_k-(9+\fg)\,e^{-3Z}\bar{a}_{k+1},\\
& \widetilde{A}_{k+1}|_{Z=0}=0,\quad \widetilde{A}_{k+1}|_{Z=+\infty}=0.
\end{cases}
\end{equation}
By applying  the operator $\cT^\ell$ (recalling that $\cT=(\p_t,\na_\h)$) with  $\ell\in \bigl[0,[s_{0, k}]\bigr]$ to the equation \eqref{phi-order-m-ind-7-1a}, and then taking the $L^2$ inner product of the resulting equation with $-e^{2Z}\cT^\ell\widetilde{A}_{k+1}$, we have
\beno
 \begin{split}
 \bigl\|e^{Z}\cT^\ell\partial_Z\widetilde{A}_{k+1}\bigr\|_{L^2_+}^2
 -2\bigl\|e^{Z}\cT^\ell\widetilde{A}_{k+1}&\bigr\|_{L^2_+}^2
 = -\int_{\mathbb{R}^3_+}\cT^\ell(\fg\, \widetilde{A}_{k+1}) | e^{2Z}\cT^\ell\widetilde{A}_{k+1}\,dy\,dZ\\
 &-\int_{\mathbb{R}^3_+}\cT^\ell\bigl(\widetilde{G}_k-(9+\fg)\,e^{-3Z}\bar{a}_{k+1}\bigr) | e^{2Z}\cT^\ell\widetilde{A}_{k+1}\,dy\,dZ.
 \end{split}
 \eeno
Notice that
 \beno
 \begin{split}
 -\int_{\mathbb{R}^3_+}\cT^\ell(\fg\, \widetilde{A}_{k+1}) | e^{2Z}\cT^\ell\widetilde{A}_{k+1}\,dy\,dZ
 &=-\int_{\mathbb{R}^3_+}\fg\,|e^{Z}\cT^\ell\widetilde{A}_{k+1}|^2\,dy\,dZ\\
 &\quad
 -\int_{\mathbb{R}^3_+}e^{Z}\bigl[\cT^\ell;\fg\bigr]\, \widetilde{A}_{k+1} | e^{Z}\cT^\ell\widetilde{A}_{k+1}\,dy\,dZ,
  \end{split}
 \eeno
we infer
\begin{equation*}
 \begin{split}
 \bigl\|e^{Z}&\cT^\ell\partial_Z\widetilde{A}_{k+1}\bigr\|_{L^2_+}^2
 +\int_{\mathbb{R}^3_+}(\fg-2)\,|e^{Z}\cT^\ell\widetilde{A}_{k+1}|^2\,dy\,dZ\\
 \leq& \bigl\|e^{Z}\bigl[\cT^\ell;\fg\bigr]\, \widetilde{A}_{k+1}\|_{L^2_+}\bigl\|e^{Z}\cT^\ell\widetilde{A}_{k+1}\bigr\|_{L^2_+}+C\Bigl(\|\widetilde{G}_k\|_{W^{s_{0,k}}_{1,T_0}}\\
 &+\bigl\| e^{-2Z}\cT^\ell(\bar{a}_{k+1})\bigr\|_{L^2_+}+\bigl\| e^{-2Z}\cT^\ell(\fg\bar{a}_{k+1})\bigr\|_{L^2+}\Bigr)
 \bigl\|e^{Z}\cT^\ell\widetilde{A}_{k+1}\bigr\|_{L^2_+}.
 \end{split}
\end{equation*}
Whereas it follows from \eqref{S3eq28}, \eqref{phi-negative-1-solution} and \eqref{phi-order-m-ind-7-1} that
 $\fg-2\geq 2-\|\fg-4\|_{L^\infty_+}\geq \frac{3}{2}$ as long as $c$ is small enough in \eqref{S3eq27}. So that we obtain
 \beno \label{phi-order-m-ind-7-7}
 \begin{split}
 \|e^{Z}\cT^\ell\partial_Z\widetilde{A}_{k+1}\|_{L^2_+}^2
 &+\|e^{Z}\cT^\ell\widetilde{A}_{k+1}\|_{L^2_+}^2
 \lesssim \bigl\|e^{Z}[\cT^\ell,\fg]\, \widetilde{A}_{k+1}\bigr\|_{L^2_+}^2\\
 &+\|\widetilde{G}_k\|_{W^{s_{0,k}}_{1,T_0}}^2+\| e^{-2Z}\cT^\ell\bar{a}_{k+1})\|_{L^2_+}^2+\| e^{-2Z}\cT^\ell(\fg\bar{a}_{k+1})\|_{L^2_+}^2.
 \end{split}
 \eeno
We deduce from   trace theorem and the proof of Lemma \ref{S3lem1} that
\beno
\begin{split}
&\bigl\|e^{Z}\bigl[\cT^\ell;\fg\bigr]\, \widetilde{A}_{k+1}\|_{L^2_+}\lesssim \|\cT\fg\|_{L^\infty_{\rm v}(W^{[s_{0,k}]-1}_{T_0})_\h}
\|e^Z\widetilde{A}_{k+1}\|_{L^2_{\rm v}((W^{[s_{0,k}]}_{T_0})_\h},\\
&\bigl\| e^{-2Z}\cT^\ell(\fg\bar{a}_{k+1})\bigr\|_{L^2+}\lesssim \|\fg\|_{L^2_{\rm v}(W^{[s_{0,k}]}_{T_0})_\h}
\|{a}_{k+1}\|_{W^{[s_{0,k}]+\frac12}_{T_0}}.
\end{split}
\eeno
Therefore, we obtain
  \beno
 \begin{split}
 \|e^{Z}&\cT^\ell\partial_Z\widetilde{A}_{k+1}\|_{L^2_+}^2
 +\frac{5}{4}\|e^{Z}\cT^\ell\widetilde{A}_{k+1}\|_{L^2_+}^2\\
 \lesssim & \Bigl(\bigl\|(a_0-1, \na\vf_0)\|_{W^{s_0-1}_{T_0}}^2+\|(A_0,\Phi_{1})\|_{W^{s_0-\frac32}_{1,T_0}}^2\Bigr)\|e^Z\widetilde{A}_{k+1}\|_{L^2_{\rm v}((W^{[s_{0,k}]}_{T_0})_\h}\\
 &\,+\|\widetilde{G}_k\|_{W^{s_{0,k}}_{1,T_0}}^2+\|a_{k+1}\|_{W^{s_{0,k}}_{T_0}}^2\Bigl(1+\bigl\|(a_0-1, \na\vf_0)\bigr\|_{W^{s_0-1}_{T_0}}^2+\|(A_0,\Phi_{1})\|_{W^{s_0-\frac32}_{1,T_0}}^2\Bigr).
 \end{split}
 \eeno
 In view of \eqref{S3eq28} and \eqref{phi-negative-1-solution}, we get, by
 summing the above inequality for $\ell$ from $0$ to  $[s_{0, k}],$ that
   \beno
 \begin{split}
 &\|e^{Z}\partial_Z\widetilde{A}_{k+1}\|_{L^2_{\rm v}((W^{[s_{0,k}]}_{T_0})_\h}^2
 +\|e^{Z}\widetilde{A}_{k+1}\|_{L^2_{\rm v}((W^{[s_{0,k}]}_{T_0})_\h}^2\lesssim \|\widetilde{G}_k\|_{W^{s_{0,k}}_{1,T_0}}^2+\|a_{k+1}\|_{W^{s_{0,k}}_{T_0}}^2,
 \end{split}
 \eeno
 which in particular implies that
  \beq \label{phi-order-m-ind-7-9}
 \|e^{Z}\widetilde{A}_{k+1}\|_{L^\infty_{\rm v}((W^{[s_{0,k}]}_{T_0})_\h}^2\lesssim \|\widetilde{G}_k\|_{W^{s_{0,k}}_{1,T_0}}^2+\|a_{k+1}\|_{W^{s_{0,k}}_{T_0}}^2.
 \eeq

In general by apply $\p_Z^k$ for $k\leq [s_{0,k}]$ and performing the above energy estimate, we achieve
     \beq \label{phi-order-m-ind-7-10}
 \begin{split}
 &\|\widetilde{A}_{k+1}\|_{W^{s_{0,k}}_{1,T_0}}^2 \lesssim \|\widetilde{G}_k\|_{W^{s_{0,k}}_{1,T_0}}^2+\|a_{k+1}\|_{W^{s_{0,k}}_{T_0}}^2.
  \end{split}
 \eeq
 With the above estimate, we deduce from  the second equation of \eqref{phi-order-m-7} that
      \beq \label{phi-order-m-ind-7-11}
 \begin{split}
 &\|\Phi_{k+2}\|_{W^{s_{0,k}}_{1,T_0}}^2 \lesssim \|\widetilde{G}_k\|_{W^{s_{0,k}}_{1,T_0}}^2+\|F_k\|_{W^{s_{0,k}}_{1,T_0}}^2.
  \end{split}
 \eeq
 Thanks to the definitions of$\widetilde{A}_{k+1}$, $\widetilde{G}_k$, and $F_k$, $G_k$ in Appendix \ref{appA-re}, we deduce \eqref{phi-negative-1-solution}.  This ends the proof of Proposition \ref{S2prop4}.
\end{proof}

Finally let us present the proof of Proposition \ref{S2prop5}.

\begin{proof}[Proof of Proposition \ref{S2prop5}]
We first observe from the $\vf_{k+2}$ equation of \eqref{outer-order-m} that
 \beq \label{S5eq3}
 \begin{split}
 \partial_t \varphi_{k+2}\bigl|_{t=0}=&- \nabla \varphi_{0} \cdot \nabla \varphi_{k+2}\bigl|_{t=0}-2a_0 a_{k+2}\bigl|_{t=0}+
 \frac{1}{2a_0}\left(\Delta a_{k}+g_{k+1}^{\varphi}\right)\bigl|_{t=0}\\
 =&- \nabla \varphi_{0,0}^{\rm in} \cdot \nabla \varphi_{k+2,0}^{\rm in}-2a_{0,0}^{\rm in} a_{k+2,0}^{\rm in}+\frac{1}{2a_{0,0}^{\rm in}}\left(\Delta a_{k,0}^{\rm in}+g_{k+1}^{\varphi}\bigl|_{t=0}\right)\\
 \eqdefa& \vf_{k+2,1}^{\rm in}\in {W^{\bar{s}_{0,k}}_{T_0}} \with \bar{s}_{0,k}\eqdefa s_0-1-2(k+2)\geq 2,
 \end{split}
 \eeq
 we are going to inductively solve the linear  equations \eqref{outer-order-m} with the initial-boundary conditions
\begin{equation}\label{bc-order-m}
\varphi_{k+2}|_{z=0}=-\Phi_{k+2}|_{Z=0}, \quad \varphi_{k+2}|_{t=0}=\vf_{k+2,0}^{\rm in} \andf \partial_t \varphi_{k+2}|_{t=0}=\vf_{k+2,1}^{\rm in}.
\end{equation}

In fact, according to the first equation in \eqref{S2eq4} and \eqref{outer-order-m}, we write
\begin{equation*}\label{outer-conti-m-1}
\begin{split}
\partial_t (a_0 a_{k+2})+\dive ( a_{0}a_{k+2}  \nabla \varphi_{0})+\frac{1}{2}\dive( \r_{0}\nabla\varphi_{k+2})=a_0 f_{k+1}^{a}.
\end{split}
\end{equation*}
By taking $\p_t$ to the
 the $\vf_{k+1}$ equation of \eqref{outer-order-m} and inserting the above equation into the resulting one, we obtain
\begin{equation}\label{outer-bern-m-1}
\begin{split}
&\partial_t^2 \varphi_{k+2} -\dive \bigl(\r_0\nabla\varphi_{k+2} \bigr)+ \partial_t(\nabla \varphi_{0} \cdot \nabla \varphi_{k+2} )\\
 &\quad+\dive (\partial_t \varphi_{k+2} \nabla\varphi_0 )+\dive \bigl((\nabla\varphi_0 \cdot \nabla\varphi_{k} )\nabla\varphi_0 \bigr)\\
 &=-2a_0 \, f_{k+1}^a+\partial_t\Bigl(\frac{1}{2a_0}\bigl(\Delta a_{k}+g_{k+1}^{\varphi}\bigr)\Bigr)+\dive\Bigl(\frac{1}{2a_0}\bigl(\Delta a_{k}+g_{k+1}^{\varphi}\bigr)\nabla\varphi_0\Bigr)\eqdefa F_{k+2}.
\end{split}
\end{equation}
Note that $F_{k+2}$ belongs to ${W^{\bar{s}_{0,k}}_{T_0}}$ according to the definitions of  $f_{k+1}^a$ and $g_{k+1}^{\varphi}$ in \eqref{remainder-inner-1}. Then we deduce from Theorem
\ref{S8thm1} that the system \eqref{outer-bern-m-1} with the boundary condition \eqref{S2eq23} has a unique solution $\vf_{k+2}$ so that
\beq \label{S5eq11}
\begin{split}
\bigl\|(\p_t\vf_{k+2},\na\vf_{k+2})\|_{W^{\bar{s}_{0,k}}_{T_0}}
\leq & C\bigl(\bigl\|\bigl( \na\vf_{k+2,0}^{\rm in}, \vf_{k+2,1}^{\rm in}\bigr)\bigr\|_{ H^{\bar{s}_{0,k}}}+\bigl\|\Phi_{k+2}|_{Z=0}\bigr\|_{ W^{s_{0,k}}_T}+\|F_{k+2}\|_{W^{\bar{s}_{0,k}}_{T_0}}\bigr)\\
\leq & C\bigl(\|(a_{0,0}^{\rm{in}}-1,\na\vf_{0,0}^{\rm{in}})\|_{H^{s_0-1}}
+\sum_{j=1}^{k+2}\|(a_{j,0}^{\rm{in}},\na\vf_{j,0}^{\rm{in}})\|_{H^{s_0-2j-1}}\bigr). \end{split}
\eeq
With $\vf_{k+2}$ thus obtained, it follows from the $a_{k+2}$ equation of \eqref{outer-order-m} and Lemma \ref{S3lem1}  that
\beno
a_{k+2} =\frac1{2a_0}
\Bigl(\frac{1}{2a_0}\left(\Delta a_{k}+g_{k+1}^{\varphi}\right)-\partial_t \varphi_{k+2}-\nabla \varphi_{0} \cdot \nabla \varphi_{k+2}\Bigr)
\in W^{\bar{s}_{0,k}}_{T_0},
\eeno in case $\bar{s}_{0,k}\geq 2.$ Moreover, there holds
\beno
\|\|a_{k+2}\|_{W^{\bar{s}_{0,k}}_{T_0}} \leq C\Bigl(\|(a_{0,0}^{\rm{in}}-1,\na\vf_{0,0}^{\rm{in}})\|_{H^{s_0-1}}
+\sum_{j=1}^{k+2}\|(a_{j,0}^{\rm{in}},\na\vf_{j,0}^{\rm{in}})\|_{H^{s_0-2j-1}}\Bigr).
\eeno
This completes the proof of \eqref{Saeq1}.
\end{proof}

%%%%%%%%%%%%%%%%%%%%%%%%%%%%%%%%%%%%%%%%%%%%%%%%%%%%%%%%%%%%
\renewcommand{\theequation}{\thesection.\arabic{equation}}
\setcounter{equation}{0}
%%%%%%%%%%%%%%%%%%%%%%%%%%%%%%%%%%%%%%%%%%%%%%%%%%%%%%%%%%%%

\section{Some technical lemmas}\label{sec-tech}

Let $(a^{\text{int}, \varepsilon, m}, \varphi^{\text{int}, \varepsilon, m}))$ and $(a^{\text{b}, \varepsilon, m},\varphi^{\text{b}, \varepsilon, m}))$
be determined by \eqref{S6eq1}, we denote
\begin{equation}\label{space-norm-E-0}
\begin{split}
& \cE_{m, T}\eqdefa  \bigl\|(a^{\text{int}, \varepsilon, m}-1,\,\p_t\varphi^{\text{int}, \varepsilon, m}, \na\varphi^{\text{int}, \varepsilon, m})\bigr\|_{W^{s_0-2m-5}_{T}}^2+
\bigl\|(a^{\text{b}, \varepsilon, m},\,\varphi^{\text{b}, \varepsilon, m})\bigr\|_{W^{s_0-2m-\frac72}_{1,T}}^2,\\
 & u^{\e,m}\eqdefa \na \vf^{\e,m}, \quad \mathcal{S}_{f}(g)\eqdefa f\cdot\nabla g+\frac{1}{2} g\,\nabla\cdot f, \quad{\Ta} \eqdefa (\partial_t,\, \nabla_{\h}).
 \end{split}
\end{equation}
Then thanks to Propositions \ref{S2prop1}-\ref{S2prop5}, we have
\begin{equation}\label{space-norm-E-0a}
\begin{split}
& \cE_{m, T} \leq C\,\mathcal{E}_{0},
 \end{split}
\end{equation}
for  $\mathcal{E}_{0}$ being given by   \eqref{initial-0}.

\begin{lem}\label{lem-est-singu-1}
{\sl Let  $s_0\geq 2m+9$ be an integer and $a^{\e,m}$ be defined by \eqref{S6eq1}. Let $f$ and $g$ be smooth enough functions satisfying the homogeneous boundary conditions $f|_{z=0}=g|_{z=0}=0$. Then one has
\begin{equation}\label{est-singu-1a}
\begin{split}
&\varepsilon^2 \bigl|\int_{\mathbb{R}^3_{+}} \frac{\Delta a^{\varepsilon, m}}{a^{\varepsilon, m}}\, f\,g\, dx\bigr|\lesssim \mathcal{E}_0^{\frac{1}{2}}\,\|\varepsilon f\|_{H^1}\,\|\varepsilon g\|_{H^1} \andf \varepsilon^2 \Bigl\|\frac{\Delta a^{\varepsilon, m}}{a^{\varepsilon, m}}\, f\Bigr\|_{L^2_{+}} \lesssim\mathcal{E}_0^{\frac{1}{2}}\,\|\varepsilon f\|_{H^1},\end{split}
\end{equation}
if $j\in [0,s_0-2m-9],$ we also have
\begin{equation}\label{est-singu-1b}
\begin{split}
&\varepsilon^2\bigl|\int_{\mathbb{R}^3_{+}}  \Ta^j\Bigl(\frac{\Delta a^{\varepsilon, m}}{a^{\varepsilon, m}}\Bigr)\, f\,g\, dx\bigr|\lesssim \mathcal{E}_0^{\frac{1}{2}}\,\|\varepsilon f\|_{H^1}\,\|\varepsilon g\|_{H^1},\\
& \e\Bigl\|\bigl[{\Ta}^j;\, \frac{\Delta a^{\varepsilon, m}}{a^{\varepsilon, m}}\bigr] f\Bigr\|_{L^2_+}\lesssim\mathcal{E}_0^{\frac{1}{2}}\,\sum_{k=0}^{j-1}
\|{\Ta}^k f\|_{H^1}.
\end{split}
\end{equation}}
\end{lem}
\begin{proof}
In view of \eqref{S6eq1},  $a^{\varepsilon, m}=a^{\text{int},\varepsilon, m}+[a^{\text{b},\varepsilon, m}]_\e,$  we
write
\begin{equation*}\label{diff-eqns-im-lin-2-1}
\begin{split}
\varepsilon^2  \int_{\mathbb{R}^3_{+}}  \frac{\Delta a^{\varepsilon, m}}{a^{\varepsilon, m}}\, f\,g\, dx
=&\varepsilon^2\int_{\mathbb{R}^3_{+}} \frac{\Delta a^{\text{int},\varepsilon, m}+[\Delta_{\h} a^{\text{b},\varepsilon, m}]_\e}{a^{\varepsilon, m}}
\, f\,g\, dx+\int_{\mathbb{R}^3_{+}}\frac{[\partial_Z^2 a^{\text{b},\varepsilon, m}]_\e}{a^{\varepsilon, m}}\,  f\,g\, dx.
\end{split}
\end{equation*}
Note that $f|_{z=0}=0,$ Hardy's inequality ensures that $\bigl\|z^{-1}f\|_{L^2(\R^3_+)}\leq C\|\p_zf\|_{L^2(\R^3_+)},$ so that we infer
\begin{equation*}\label{diff-eqns-im-lin-2-1}
\begin{split}
{\varepsilon^2}\bigl|\int_{\mathbb{R}^3_{+}}\frac{\Delta a^{\text{int},\varepsilon, m}+[\Delta_{\h} a^{\text{b},\varepsilon, m}]_\e}{a^{\varepsilon, m}}\, f\,g\, dx\bigr|
&\lesssim\Bigl\|\frac{\Delta a^{\text{int},\varepsilon, m}+[\Delta_{\h} a^{\text{b},\varepsilon, m}]_\e}{a^{\varepsilon, m}}\Bigr\|_{L^\infty_+}\, \|\varepsilon f\|_{L^2_+}\|\varepsilon g\|_{L^2_+}\\
&\lesssim \mathcal{E}_0^{\frac{1}{2}}\,\|\varepsilon\,f\|_{L^2_+}\,\|\varepsilon\,g\|_{L^2_+},\\
\end{split}
\end{equation*}
and
\begin{equation*}\label{diff-eqns-im-lin-2-1}
\begin{split}
\bigl|\int_{\mathbb{R}^3_{+}}\frac{[\partial_Z^2 a^{\text{b},\varepsilon, m}]_\e}{a^{\varepsilon, m}}\, f\,g\, dx\bigr|
&=\bigl|\int_{\Omega}\frac{[Z^2\partial_Z^2 a^{\text{b},\varepsilon, m}]_\e}{a^{\varepsilon, m}}\, (\varepsilon\,z^{-1} f)\,(\varepsilon\,z^{-1} g)\, dx\bigr|\\
&\lesssim\Bigl\|\frac{Z^2\partial_Z^2 a^{\text{b},\varepsilon, m}}{a^{\varepsilon, m}}\Bigr\|_{L^\infty_+} \|\varepsilon\,z^{-1}f\|_{L^2_+}\,\|\varepsilon\,z^{-1}g\|_{L^2_+}\\
&\lesssim\Bigl\|\frac{Z^2\partial_Z^2 a^{\text{b},\varepsilon, m}}{a^{\varepsilon, m}}\Bigr\|_{L^\infty_+} \|\varepsilon\,\partial_zf\|_{L^2_+}\|\varepsilon\,\partial_zg\|_{L^2_+}\lesssim \mathcal{E}_{0}^{\frac{1}{2}}\|\varepsilon\,\partial_zf\|_{L^2_+}\|\varepsilon\,\partial_zg\|_{L^2_+},\\
\end{split}
\end{equation*}
where we used the fact that
$\|\frac{Z^2\partial_Z^2 a^{\text{b},\varepsilon, m}}{a^{\varepsilon, m}}\|_{L^\infty_+} \lesssim \mathcal{E}_{0}^{\frac{1}{2}}$
due to $\partial_Z^2 a^{\text{b},\varepsilon, m} \in \mathcal{W}^{s_0-2m-\frac{11}2}_{1, T_0}$. This leads to the first inequality of \eqref{est-singu-1a}.

Along the same line, we observe that
\begin{equation}\label{import-est-1-1}
\begin{split}
\varepsilon^2 \Bigl\|\frac{\Delta a^{\varepsilon, m}}{a^{\varepsilon, m}}\, f\Bigr\|_{L^2_+}
\leq&\varepsilon^2\Bigl\|\frac{\Delta a^{\text{int},\varepsilon, m}+[\Delta_{\h} a^{\text{b},\varepsilon, m}]_\e}{a^{\varepsilon, m}}\Bigr\|_{L^\infty_+}
\|f\|_{L^2_+}+\Bigl\|\frac{[\partial_Z^2 a^{\text{b},\varepsilon, m}]_\e}{a^{\varepsilon, m}}\,f\Bigr\|_{L^2_+}\\
 \lesssim&\varepsilon\,\mathcal{E}_0^{\frac{1}{2}}\,\|\varepsilon\,f\|_{L^2_+}+\Bigl\|\frac{[\partial_Z^2 a^{\text{b},\varepsilon, m}]_\e}{a^{\varepsilon, m}}\,f\Bigr\|_{L^2_+}.
\end{split}
\end{equation}
Whereas it follows from  inequality, $\bigl\|z^{-1}f\|_{L^2(\R^3_+)}\leq C\|\p_zf\|_{L^2(\R^3_+)},$  that we
\begin{equation*}\label{import-est-1-2}
\begin{split}
\Bigl\|\frac{[\partial_Z^2 a^{\text{b},\varepsilon, m}]_\e}{a^{\varepsilon, m}}\, f\Bigr\|_{L^2_+}
=&\Bigl\|\frac{[Z\partial_Z^2 a^{\text{b},\varepsilon, m}]_\e}{a^{\varepsilon, m}}\, (\varepsilon\,z^{-1} f)\Bigr\|_{L^2_+}\\
\lesssim&\Bigl\|\frac{Z\partial_Z^2 a^{\text{b},\varepsilon, m}}{a^{\varepsilon, m}}\Bigr\|_{L^\infty_+} \|\varepsilon\,z^{-1}f\|_{L^2_+}\lesssim \mathcal{E}_{0}^{\frac{1}{2}}\|\varepsilon\,\partial_zf\|_{L^2_+},
\end{split}
\end{equation*}
which together with \eqref{import-est-1-1} ensures  the second inequality of \eqref{est-singu-1a}.

The inequalities of \eqref{est-singu-1b} can be proved along the same line, we omit the details here.
This completes the proof of Lemma \ref{lem-est-singu-1}.
\end{proof}

\begin{lem}\label{lem-est-conv-2}
{\sl Let  $f$ and $g$ be smooth enough functions which satisfy the homogenous boundary condition $f|_{z=0}=g|_{z=0}=0.$
 Then one has
 \begin{equation}\label{conv-id-1}
\begin{split}
&\int_{\mathbb{R}^3_+} \mathcal{S}_{u^{\varepsilon, m}}(f) \,g\,dx=-\int_{\mathbb{R}^3_+} \mathcal{S}_{u^{\varepsilon, m}}(g) \,\,f\,dx,\quad\int_{\mathbb{R}^3_+} \mathcal{S}_{u^{\varepsilon, m}}(f) \,\,f\,dx=0,\\
&\int_{\mathbb{R}^3_+} \bigl(\mathcal{S}_{u^{\varepsilon, m}}(g)  \p_tf-\mathcal{S}_{u^{\varepsilon, m}}(f)\p_tg \bigr)\, dx=\frac{d}{dt}\int_{\mathbb{R}^3_+}\mathcal{S}_{u^{\varepsilon, m}}(g)\,f\,dx
- \int_{\mathbb{R}^3_+}\mathcal{S}_{\p_tu^{\varepsilon,m}}(g)  \,f \,dx.
\end{split}
\end{equation}
Moreover, for $j\in [0, s_0-2m-6],$
there holds
\ben\label{est-conv-2-1}
& \|\mathcal{S}_{\Ta^ju^{\varepsilon, m}}(f)\|_{L^2_+}\lesssim \mathcal{E}_0^{\frac{1}{2}}\|f\|_{H^1}
\andf  \|[{\Ta}^j;\,\mathcal{S}_{u^{\varepsilon, m}}](f)\|_{L^2_+}\lesssim  \mathcal{E}_0^{\frac{1}{2}}\sum_{k=0}^{j-1}\|{\Ta}^kf\|_{H^1}.
\een}
\end{lem}
\begin{proof} The first two equalities in \eqref{conv-id-1} can be obtained by using integration by parts.
Whereas observing that
\begin{equation*}\label{conv-id-1-1}
\begin{split}
\int_{\mathbb{R}^3_+} \bigl(\mathcal{S}_{u^{\varepsilon, m}}(g)\,\partial_t f \, -& \mathcal{S}_{u^{\varepsilon, m}}(f)  \,\partial_t g)\, dx\\
=&\frac{d}{dt}\int_{\mathbb{R}^3_+}\mathcal{S}_{u^{\varepsilon, m}}(g)
\,f\,dx-\int_{\R^3_+}\dive\bigl(u^{\varepsilon, m}f\p_tg\bigr)- \int_{\mathbb{R}^3_+}\mathcal{S}_{\p_tu^{\varepsilon,m}}(g)  \,f \,dx.
\end{split}
\end{equation*}
Then the second equation in \eqref{conv-id-1} follows from the homogeneous boundary condition of $f.$

Next we just prove the first inequality of \eqref{est-conv-2-1} for the case $j=0$.
 We observe that
\begin{equation*}\label{valid-soln-3}
\begin{split}
\|f\,\nabla\cdot u^{\varepsilon, m}\|_{L^2_+} =&\bigl\|f\bigl(\dive u^{\text{int},\varepsilon, m}+
[\dive_\h u^{\text{b},\varepsilon, m}]_\e+\varepsilon^{-1}[\p_Z u^{\text{b},\varepsilon, m}]_\e\bigr)\bigr\|_{L^2_+}\\
\lesssim &\|f\|_{L^6_+} \|\dive u^{\text{int},\varepsilon, m} \|_{L^3_+} +\|f\|_{L^2_+}\|\dive_\h u^{\text{b},\varepsilon, m}\|_{L^\infty_+} \\ &+\|z^{-1}\,f\|_{L^2_+}\,\|Z\,\p_{Z} u^{\text{b},\varepsilon, m}\|_{L^\infty_+}\\
\lesssim& \|f\|_{H^1}\bigl( \| u^{\text{int},\varepsilon, m} \|_{H^2} +\|u^{\text{b},\varepsilon, m}\|_{L^\infty_\v(H^{\frac52}_\h)}\bigr)+\|\partial_zf\|_{L^2_+}\,\|Z\,\p_{Z}u^{\text{b},\varepsilon, m}\|_{L^\infty_\v(H^{\frac32}_\h)}.
\end{split}
\end{equation*}
As a result, we achieve
\begin{equation*}
\begin{split}
 \|\mathcal{S}_{u^{\varepsilon, m}}(f)\|_{L^2_+}\lesssim& \|u^{\varepsilon, m}\cdot\nabla\,f\|_{L^2_+}+\|f\,\nabla\cdot u^{\varepsilon, m}\|_{L^2_+}\\
\lesssim &\|f\|_{H^1} \Bigl(\|u^{\varepsilon, m}\|_{L^\infty_+}+\| u^{\text{int},\varepsilon, m} \|_{H^2} +\|u^{\text{b},\varepsilon, m}\|_{L^\infty_\v(H^{\frac52}_\h)}+\|Z\,\p_{Z}u^{\text{b},\varepsilon, m}\|_{L^\infty_\v(H^{\frac32}_\h)}\Bigr)\\
\lesssim &\mathcal{E}_0^{\frac{1}{2}}\|f\|_{H^1},
\end{split}
\end{equation*}
which leads to the first inequality of  \eqref{est-conv-2-1} for $j=0.$

The second inequality of  \eqref{est-conv-2-1} follows from the first one. This ends the proof of  Lemma \ref{lem-est-conv-2}.
\end{proof}

\begin{lem}\label{lem-linear-high-1}
{\sl Let $(\frak{A}_1,\frak{A}_2)$ be smooth enough solution to the following system:
\begin{equation}\label{diff-eqns-gene-lin}
\begin{cases}
&\varepsilon\bigl(\partial_{t}+\mathcal{S}_{u^{\varepsilon, m}}(\cdot)\bigr)\frak{A}_1
+\frac{\varepsilon^{2}}{2}\Delta \frak{A}_2=\Bigl(\frac{\varepsilon^2}{2}\frac{\Delta a^{\varepsilon, m}}{a^{\varepsilon, m}}-\varepsilon^{m+2} \hbar_1\Bigr)\frak{A}_2-\varepsilon^{m+2} \hbar_2+f_1-f_2,\\
&\varepsilon \bigl(\partial_{t}+\mathcal{S}_{u^{\varepsilon, m}}(\cdot)\bigr)\frak{A}_2-\frac{\varepsilon^{2}}{2} \Delta \frak{A}_1 +2(a^{\varepsilon,m})^{2}\frak{A}_1=-\chi_1\frak{A}_1+\varepsilon^{m+1} \chi_2-g_1-g_2,\\
&\frak{A}_1|_{z=0}=0,\quad\,\frak{A}_2|_{z=0}=0.
\end{cases}
\end{equation}
 Then if $s_0-2m-10\geq 0,$ one has
\begin{equation}\label{1-28-10}
\begin{split}
&\frac{d}{dt}\Bigl\{\frac{\varepsilon^2}{4}\bigl\|(\nabla\frak{A}_1,\,\nabla\frak{A}_2)\bigr\|_{L^2_+}^2
+\int_{\mathbb{R}^3_+}\Bigl(\bigl((a^{\varepsilon,m})^{2}
+\frac12\chi_1\bigr)\frak{A}_1^2+\mathcal{S}_{u^{\varepsilon, m}}(\varepsilon \frak{A}_2 )\frak{A}_1\Bigr)\,dx\\
&\,+\frac{1}{2}\int_{\mathbb{R}^3_+} \Bigl(\Bigl(\frac{\varepsilon^2}{2}\frac{\Delta a^{\varepsilon, m}}{a^{\varepsilon, m}}- \varepsilon^{m+2} \hbar_1\Bigr)\frak{A}_2^2- 2\varepsilon^{m+2} \hbar_2\frak{A}_2-2\varepsilon^{m+1} \chi_2 \frak{A}_1 \Bigr)\, dx\Bigr\}\\
&\,+\int_{\mathbb{R}^3_+} \bigl(g_2 \partial_t\frak{A}_1-f_2 \partial_t\frak{A}_2\bigr)\, dx
\lesssim \bigl(1+\|\hbar_1\|_{L^\infty_+}+\|\p_t\hbar_1\|_{L^\infty_+}\bigr)\|\varepsilon\left(\frak{A}_1,\frak{A}_2\right) \|_{H^1}^2\\
&\,+\bigl(1+\|\chi_1\|_{L^\infty_+}+\|\p_t\chi_1\|_{L^\infty_+}\bigr)\|\frak{A}_1\|_{L^2_+}^2
+\|(\varepsilon^{-1}f_1,\,\varepsilon^{-1}g_1,\,g_2,\,f_2)\|_{L^2_+}^2\\
&\,+\|\nabla\,(f_1,\,g_1)\|_{L^2_+}^2+\varepsilon^{2(m+1)}\bigl(\|(\chi_2,\e\hbar_2)\|_{L^2_+}^2+\|(\p_t \chi_2,\partial_t\hbar_2)\|_{L^2_+}^2\bigr).
\end{split}
\end{equation}}
\end{lem}
\begin{proof}
We first get, by taking the $L^2$ inner product of the $\frak{A}_1$ equation of \eqref{diff-eqns-gene-lin} with $-\partial_t \frak{A}_2$, that
\begin{equation}\label{diff-eqns-im-lin-2-1}
\begin{split}
& -\varepsilon\int_{\mathbb{R}^3_+}\bigl(\partial_{t}\frak{A}_1+\mathcal{S}_{u^{\varepsilon, m}}(\frak{A}_1)\bigr)\partial_t \frak{A}_2 \, dx+\frac{\varepsilon^{2}}{4}\frac{d}{dt}\|\nabla\,\frak{A}_2\|_{L^2_+}^2\\
&=-\int_{\mathbb{R}^3_+} \Bigl(\frac{\varepsilon^2}{2}\frac{\Delta a^{\varepsilon, m}}{a^{\varepsilon, m}}- \varepsilon^{m+2} \hbar_1\Bigr)\, \frak{A}_2\, \partial_t \frak{A}_2\, dx-\int_{\mathbb{R}^3_+} \bigl(f_1-\varepsilon^{m+2} \hbar_2-f_2\bigr)\, \partial_t \frak{A}_2\, dx.
\end{split}
\end{equation}
While by substituting the $\frak{A}_2$  equation  of \eqref{diff-eqns-gene-lin} into  $-\int_{\mathbb{R}^3_+} f_1\, \partial_t \frak{A}_2\, dx,$
 one has
\begin{equation*}\label{diff-eqns-im-lin-2-5}
\begin{split}
&-\int_{\mathbb{R}^3_+} f_1\, \partial_t \frak{A}_2\, dx=\frac{\varepsilon}{2}\int_{\mathbb{R}^3_+} \,\nabla\,f_1\cdot  \nabla \frak{A}_1\, dx\\
&\quad+\varepsilon^{-1}\int_{\mathbb{R}^3_+} \,f_1\, \bigl(\varepsilon\mathcal{S}_{u^{\varepsilon, m}}(\frak{A}_2)+\bigl(2(a^{\varepsilon,m})^{2}+\chi_1\bigr) \frak{A}_1 -\varepsilon^{m+1} \chi_2+g_1+g_2\bigr)\, dx.
\end{split}
\end{equation*}
By inserting the above equality into \eqref{diff-eqns-im-lin-2-1} and using integrating by parts, we obtain
\begin{equation}\label{diff-eqns-im-lin-2-6}
\begin{split}
&\frac{d}{dt}\Bigl\{\frac{\e^2}{4}\|\na\frak{A}_2\|_{L^2_+}^2+\frac{1}{2}\int_{\mathbb{R}^3_+} \Bigl(\Bigl(\frac{\varepsilon^2}{2}\frac{\Delta a^{\varepsilon, m}}{a^{\varepsilon, m}}-
\varepsilon^{m+2} \hbar_1\Bigr)\, \frak{A}_2^2- 2\varepsilon^{m+2} \hbar_2\,  \frak{A}_2\Bigr)\, dx\Bigr\}\\
&- \varepsilon\int_{\mathbb{R}^3_+}\bigl(\partial_{t}\frak{A}_1+\mathcal{S}_{u^{\varepsilon, m}}( \frak{A}_1)\bigr) \,\partial_t \frak{A}_2 \, dx-\int_{\mathbb{R}^3_+} f_2\, \partial_t \frak{A}_2\, dx=-\varepsilon^{m+2} \int_{\mathbb{R}^3_+}  \partial_t\hbar_2\,  \frak{A}_2\, dx\\
&+\frac{1}{2}\int_{\mathbb{R}^3_+} \partial_t\Bigl(\frac{\varepsilon^2}{2}\frac{\Delta a^{\varepsilon, m}}{a^{\varepsilon, m}}- \varepsilon^{m+2} \hbar_1\Bigr)\, \frak{A}_2^2\, dx+ \frac{\varepsilon}{2}\int_{\mathbb{R}^3_+} \,\nabla\,f_1\cdot  \nabla \frak{A}_1\, dx\\
&+\varepsilon^{-1}\int_{\mathbb{R}^3_+} \,f_1\, \bigl(\varepsilon\mathcal{S}_{u^{\varepsilon, m}}(\frak{A}_2)+\bigl(2(a^{\varepsilon,m})^{2}+\chi_1\bigr) \frak{A}_1 -\varepsilon^{m+1} \chi_2+g_1+g_2\bigr)\, dx.
\end{split}
\end{equation}

On the other hand,  we get, by taking the $L^2$ inner product of the $\frak{A}_2$ equation of \eqref{diff-eqns-gene-lin} with $\partial_t \frak{A}_1 $, that
\begin{equation}\label{diff-eqns-re-lin-2-1}
\begin{split}
&\varepsilon\int_{\mathbb{R}^3_+} \bigl(\partial_{t}\frak{A}_2 +\mathcal{S}_{u^{\varepsilon, m}}(\frak{A}_2)\bigr) \,\partial_t \frak{A}_1  \, dx+\frac{\varepsilon^2}{4}\frac{d}{dt}\|\nabla\,\frak{A}_1 \|_{L^2_+}^2+\frac{d}{dt}\int_{\mathbb{R}^3_+} \bigl((a^{\varepsilon,m})^{2}
+\frac12 \chi_1\bigr)\frak{A}_1^2 \, dx\\
&=\int_{\mathbb{R}^3_+} \bigl(\varepsilon^{m+1} \chi_2-g_1-g_2\bigr)\partial_t \frak{A}_1\, dx+\int_{\mathbb{R}^3_+} \bigl(2\,a^{\varepsilon,m}\partial_ta^{\varepsilon,m}+\p_t\chi_1\bigr)\frak{A}_1^2 \, dx.
\end{split}
\end{equation}
By using the $\frak{A}_1$ equation of \eqref{diff-eqns-gene-lin}, we write
\begin{equation*}
\begin{split}
&-\int_{\mathbb{R}^3_+} g_1\, \partial_t \frak{A}_1 \, dx
=- \frac{\varepsilon}{2}\int_{\mathbb{R}^3_+} \nabla\,g_1\cdot\,\nabla \frak{A}_2\, dx-\varepsilon^{-1}\int_{\mathbb{R}^3_+} g_1\,(f_1-f_2)\, dx\\
&\qquad+\varepsilon^{m+1}\int_{\mathbb{R}^3_+} g_1\,\hbar_2\, dx+\varepsilon^{-1}\int_{\mathbb{R}^3_+} g_1\, \Bigl(\varepsilon\mathcal{S}_{u^{\varepsilon, m}}(\frak{A}_1)-\bigl(\frac{\varepsilon^2}{2}\frac{\Delta a^{\varepsilon, m}}{a^{\varepsilon, m}}-\varepsilon^{m+2} \hbar_1\bigr)\frak{A}_2\Bigr)\, dx.
\end{split}
\end{equation*}
By inserting the above equality into \eqref{diff-eqns-re-lin-2-1} and using integrating by parts, we find
\begin{equation}\label{diff-eqns-re-lin-2-4}
\begin{split}
&\frac{\varepsilon^2}{4}\frac{d}{dt}\|\nabla\,\frak{A}_1 \|_{L^2_+}^2+\frac{d}{dt}\int_{\mathbb{R}^3_+} \bigl(\bigl((a^{\varepsilon,m})^{2}
+\frac12 \chi_1\bigr)\frak{A}_1^2- \varepsilon^{m+1}\chi_2 \frak{A}_1\bigr) \, dx\\
&\quad+\varepsilon\int_{\mathbb{R}^3_+} \bigl(\partial_{t}\frak{A}_2 +\mathcal{S}_{u^{\varepsilon, m}}(\frak{A}_2)\bigr) \,\partial_t \frak{A}_1  \, dx+\int_{\mathbb{R}^3_+}g_2\,\partial_t \frak{A}_1\, dx=-\varepsilon^{m+1}\int_{\mathbb{R}^3_+}  \partial_t\chi_2\, \frak{A}_1\, dx\\
&\quad+\int_{\mathbb{R}^3_+} \bigl(2\,a^{\varepsilon,m}\partial_ta^{\varepsilon,m}+\p_t\chi_1\bigr)\frak{A}_1^2 \, dx- \frac{\varepsilon}{2}\int_{\mathbb{R}^3_+} \nabla\,g_1\cdot\,\nabla \frak{A}_2\, dx+\varepsilon^{m-1}\int_{\mathbb{R}^3_+} g_1\,\hbar_2\, dx\\
&\quad+\varepsilon^{-1}\int_{\mathbb{R}^3_+} g_1\, \Bigl(f_2-f_1+\varepsilon\mathcal{S}_{u^{\varepsilon, m}}(\frak{A}_1)-\bigl(\frac{\varepsilon^2}{2}\frac{\Delta a^{\varepsilon, m}}{a^{\varepsilon, m}}-\varepsilon^{m+2} \hbar_1\bigr)\Bigr)\frak{A}_2\Bigr)\, dx.
\end{split}
\end{equation}

Thanks to \eqref{conv-id-1}, we get, by summing up \eqref{diff-eqns-im-lin-2-6} and \eqref{diff-eqns-re-lin-2-4}, that
\begin{equation}\label{diff-eqns-imre-lin-2}
\begin{split}
&\frac{d}{dt}\Bigl\{\frac{\e^2}{4}\|(\nabla\,\frak{A}_1,\,\nabla\,\frak{A}_2)\|_{L^2_+}^2
+\frac{1}{2}\int_{\mathbb{R}^3_+} \Bigl(\Bigl(\frac{\varepsilon^2}{2}\frac{\Delta a^{\varepsilon, m}}{a^{\varepsilon, m}}- \varepsilon^{m+2} \hbar_1\Bigr)\, \frak{A}_2^2- 2\varepsilon^{m+2} \hbar_2\,  \frak{A}_2\Bigr)\, dx\\
&\quad+\int_{\mathbb{R}^3_+}\bigl((a^{\varepsilon,m})^{2}
+\frac12 \chi_1\bigr)\frak{A}_1^2-\varepsilon^{m+1} \chi_2 \frak{A}_1+\mathcal{S}_{u^{\varepsilon, m}}(\varepsilon\frak{A}_2 )\,\frak{A}_1\, \bigr)dx\Bigr\}\\
&\quad+\int_{\mathbb{R}^3_+} (g_2\,\partial_t \frak{A}_1-f_2\, \partial_t \frak{A}_2)\, dx=\mathfrak{R},
\end{split}
\end{equation}
where
\begin{equation*}
\begin{split}
\mathfrak{R}\eqdefa & \int_{\mathbb{R}^3_+} \bigl(\mathcal{S}_{\p_tu^{\varepsilon,m}}(\varepsilon\frak{A}_2 )\,\frak{A}_1+\bigl(2\,a^{\varepsilon,m}\partial_ta^{\varepsilon,m}+\p_t\chi_1\bigr)\frak{A}_1^2 \bigr)\, dx\\
&\,+\frac{1}{2}\int_{\mathbb{R}^3_+} \partial_t\Bigl(\frac{\varepsilon^2}{2}\frac{\Delta a^{\varepsilon, m}}{a^{\varepsilon, m}}- \varepsilon^{m+2} \hbar_1\Bigr)\, \frak{A}_2^2\, dx-\varepsilon^{m+1}\int_{\mathbb{R}^3_+}  \bigl(\p_t \chi_2 \frak{A}_1+\varepsilon \partial_t\hbar_2\,  \frak{A}_2\bigr)\, dx\\
&\,+ \frac{\varepsilon}{2}\int_{\mathbb{R}^3_+} \bigl(\nabla\,f_1\cdot  \nabla \frak{A}_1- \nabla\,g_1\cdot\,\nabla \frak{A}_2\bigr)\, dx+\varepsilon^{-1}\int_{\mathbb{R}^3_+} \,\bigl(f_1\,\mathcal{S}_{u^{\varepsilon, m}}(\varepsilon\frak{A}_2)+g_1\, \mathcal{S}_{u^{\varepsilon, m}}(\varepsilon\frak{A}_1)\bigr)\, dx\\
&\,+\varepsilon^{-1}\int_{\mathbb{R}^3_+} \,\Bigl(f_1\, \bigl(2(a^{\varepsilon,m})^{2}+\chi_1\bigr) \frak{A}_1 -g_1\,\bigl(\frac{\varepsilon^2}{2}\frac{\Delta a^{\varepsilon, m}}{a^{\varepsilon, m}}-\varepsilon^{m+2}\hbar_1\Bigr)\frak{A}_2\\
&\qquad\qquad\qquad\qquad\qquad\qquad\qquad+f_1\,g_2+g_1\,f_2-\varepsilon^{m+1}( f_1 \chi_2-\varepsilon g_1\,\hbar_2)\Bigr)\, dx.
\end{split}
\end{equation*}
Notice that $s_0-2m\geq 10,$ we get, by applying Lemmas \ref{lem-est-singu-1} and \ref{lem-est-conv-2}, that
\begin{equation*}\label{diff-eqns-imre-lin-3}
\begin{split}
&\bigl|\int_{\mathbb{R}^3_+}\mathcal{S}_{\pa_tu^{\varepsilon, m}}(\varepsilon\frak{A}_2  )\,\frak{A}_1\,dx\bigr|\lesssim \|\mathcal{S}_{\pa_tu^{\varepsilon, m}}(\varepsilon\frak{A}_2  )\|_{L^2_+}\,\|\frak{A}_1\|_{L^2_+}\lesssim \mathcal{E}_0^{\frac{1}{2}}\|\varepsilon\frak{A}_2 \|_{H^1}\,\|\frak{A}_1\|_{L^2_+},\\
&\frac{\varepsilon^2}{4}\bigl|\int_{\mathbb{R}^3_+} \partial_t\Bigl(\frac{\Delta a^{\varepsilon, m}}{a^{\varepsilon, m}}\Bigr)\frak{A}_1^2\, dx\bigr|\lesssim \mathcal{E}_0^{\frac{1}{2}}\|\varepsilon\frak{A}_1 \|_{H^1}^2,\\
\end{split}
\end{equation*}
and
\begin{equation*}
\begin{split}
&\varepsilon^{-1}\bigl|\int_{\mathbb{R}^3_+} \,\bigl(f_1\,\mathcal{S}_{u^{\varepsilon, m}}(\varepsilon\frak{A}_2)+g_1\, \mathcal{S}_{u^{\varepsilon, m}}(\varepsilon\frak{A}_1)\bigr)\, dx\bigr|\lesssim \mathcal{E}_0^{\frac{1}{2}}\varepsilon^{-1}\|(f_1,\,g_1)\|_{L^2_+}\|\varepsilon(\frak{A}_1,\, \frak{A}_2)\|_{H^1},\\
&\varepsilon^{-1}\bigl|\int_{\mathbb{R}^3_+} \,g_1\,\frac{\varepsilon^2}{2}\frac{\Delta a^{\varepsilon, m}}{a^{\varepsilon, m}}\frak{A}_2\, dx\bigr|\lesssim \varepsilon^{-1}\|g_1\|_{L^2_+}\Bigl\|\frac{\varepsilon^2}{2}\frac{\Delta a^{\varepsilon, m}}{a^{\varepsilon, m}}\frak{A}_2\Bigr\|_{L^2_+}\lesssim \mathcal{E}_0^{\frac{1}{2}}\varepsilon^{-1}\|g_1\|_{L^2_+}\|\varepsilon\frak{A}_2\|_{H^1},
\end{split}
\end{equation*}
and
\begin{equation*}
\begin{split}
&\bigl|\int_{\mathbb{R}^3_+}\bigl(2\,a^{\varepsilon,m}\partial_ta^{\varepsilon,m}+\p_t\chi_1\bigr)\frak{A}_1^2\, dx\bigr|\lesssim \bigl(\mathcal{E}_0+\|\p_t\chi_1\|_{L^\infty_+}\bigr)\|\frak{A}_1^2\|_{L^2_+}^2,\\
&\frac{\varepsilon^{m+2}}{2} \bigl|\int_{\mathbb{R}^3_+} \partial_t\hbar_1\, \frak{A}_2^2\, dx\bigr|\lesssim \varepsilon^{m}\|\p_t\hbar_1\|_{L^\infty_+}\|\varepsilon \frak{A}_2\|_{L^2_+}^2,\\
&\varepsilon^{m+1}\bigl|\int_{\mathbb{R}^3_+}  \bigl(\p_t \chi_2 \frak{A}_1+\varepsilon \partial_t\hbar_2\,  \frak{A}_2\bigr)\, dx\bigr|\lesssim \varepsilon^{m+1}\|(\p_t \chi_2,\partial_t\hbar_2)\|_{L^2_+}\|(\frak{A}_1,\,\varepsilon \frak{A}_2)\|_{L^2_+},
\end{split}
\end{equation*}
and
\begin{equation*}
\begin{split}
&\varepsilon^{-1}\bigl|\int_{\mathbb{R}^3_+} \,f_1\, \bigl(2(a^{\varepsilon,m})^{2}+\chi_1\bigr) \frak{A}_1 \, dx\bigr|\lesssim \varepsilon^{-1}\bigl(\cE_0+\|\chi_1\|_{L^\infty_+}\bigr)\|f_1\|_{L^2_+}\|\frak{A}_1\|_{L^2_+},\\
&\varepsilon^{m+1}\bigl|\int_{\mathbb{R}^3_+} \,g_1 \hbar_1\frak{A}_2\, dx\bigr|\lesssim \varepsilon^m\|\hbar_1\|_{L^\infty_+}\|g_1\|_{L^2_+}\|\varepsilon\frak{A}_2\|_{L^2_+},\\
&\varepsilon^{m}\bigl|\int_{\mathbb{R}^3_+} \bigl(-f_1 \chi_2+\varepsilon g_1\,\hbar_2\bigr)\, dx\bigr|\lesssim\varepsilon^{m}\bigl(\|f_1\|_{L^2_+}\, \|\chi_2\|_{L^2_+}+\varepsilon \|g_1\|_{L^2_+}\,\|\hbar_2\|_{L^2_+}\bigr),
\end{split}
\end{equation*}
and
\begin{equation*}
\begin{split}
&\frac{\varepsilon}{2}\bigl|\int_{\mathbb{R}^3_+} \bigl(\nabla\,f_1\cdot  \nabla \frak{A}_1- \nabla\,g_1\cdot\,\nabla \frak{A}_2\bigr)\, dx\bigr|\lesssim \|(\nabla\,f_1,\nabla\,g_1)\|_{L^2_+}\|\varepsilon(\nabla \frak{A}_1,\,\nabla \frak{A}_2)\|_{L^2_+},\\
&\varepsilon^{-1}\bigl|\int_{\mathbb{R}^3_+} \,(f_1\,g_2+g_1\,f_2)\, dx\bigr|\lesssim \varepsilon^{-1} \bigl(\|f_1\|_{L^2_+}\,\|g_2\|_{L^2_+}+\|g_1\|_{L^2_+}\,\|f_2\|_{L^2_+}\bigr).
\end{split}
\end{equation*}
By substituting the above inequalities into \eqref{diff-eqns-imre-lin-2}, we obtain \eqref{1-28-10}. This completes
the proof of Lemma \ref{lem-linear-high-1}.
\end{proof}

%%%%%%%%%%%%%%%%%%%%%%%%%%%%%%%%%%%%%%%%%%%%%%%%%%%%%%%%%%%%
\renewcommand{\theequation}{\thesection.\arabic{equation}}
\setcounter{equation}{0}
%%%%%%%%%%%%%%%%%%%%%%%%%%%%%%%%%%%%%%%%%%%%%%%%%%%%%%%%%%%%

\section{Validity of the WKB expansion}\label{sec-valid-wkb}

The goal of this section is to present the proof of Theorem \ref{thmmain}.

\begin{prop}\label{lem-appro-solns-1}
{\sl  Let $\Psi^{a, m}$ be given by \eqref{S6eq1}. Then one has
 \begin{equation}\label{Diri-appro-BC-1}
a^{\varepsilon, m}|_{z=0}=1,\quad \varphi^{\varepsilon, m}|_{z=0}=0
 \end{equation}
and
\begin{equation}\label{appr-soln-1}
{\rm GP}(\Psi^{a, m})\eqdefa  i\varepsilon \partial_t\Psi^{a, m}+\frac{\varepsilon^2}{2}\Delta \Psi^{a, m}-\Psi^{a, m}(|\Psi^{a, m}|^2-1)=R^{\varepsilon, m}e^{\frac{i}{\varepsilon}\varphi^{\varepsilon, m}}
\end{equation}
where $R^{\varepsilon, m}$ is of the form:
\begin{equation}\label{S6eq5}
\begin{split}
R^{\varepsilon, m}=&-\varepsilon^{m+1}\,a^{\varepsilon, m}\bigl(\varepsilon\,R_{a}^{\text{int}, m}+[R_{a}^{\text{b}, m}]_\e\bigr)+i\varepsilon^{m+2}\,\bigl(\varepsilon\,R_{\varphi}^{\text{int}, m}+[R_{\varphi}^{\text{b}, m}]_\e\bigr).
\end{split}
\end{equation}
Moreover, there holds
\begin{equation}\label{S6eq5a}
\begin{split}
&\bigl\|(R_{a}^{\text{int}, m},\,\,R_{\varphi}^{\text{int}, m})\bigr\|_{ W^{s_0-2(m+3)}_{T_0}}
+\bigl\|(R_{a}^{\text{b}, m},\,\, R_{\varphi}^{\text{b}, m})\bigr\|_{\mathcal{W}^{s_0-2(m+3)+\frac12}_{1, T_0}}\lesssim \cE_0.
\end{split}
\end{equation}
}
\end{prop}

\begin{proof}
Let  $a^{\varepsilon, m}$ and $\varphi^{\varepsilon, m}$ be given by \eqref{S6eq1}. Then we observe from the computations
presented in Section \ref{Sect2} that
\begin{equation*}\label{appro-nls-1}
\begin{split}
&\partial_t a^{\varepsilon, m}+\nabla (\varphi^{\varepsilon, m}-\varepsilon^{m+2}\varphi_{m+2})\cdot \nabla a^{\varepsilon, m}+\frac{1}{2}a^{\varepsilon, m}\Delta (\varphi^{\varepsilon, m}-\varepsilon^{m+2}\varphi_{m+2})\\
&\qquad\qquad\qquad\qquad\qquad\qquad\qquad\qquad\qquad\qquad=\varepsilon^{m+2}\bigl(R_a^{\text{int}, m}+\varepsilon^{-1}[R_a^{\text{b}, m}]_\e\bigr),\\
&\partial_t (\varphi^{\varepsilon, m}-\varepsilon^{m+2}\varphi_{m+2})+\frac{1}{2}|\nabla (\varphi^{\varepsilon, m}-\varepsilon^{m+2}\varphi_{m+2})|^2+(a^{\varepsilon, m})^2-1-\frac{\varepsilon^2}{2}\frac{\Delta a^{\varepsilon, m}}{a^{\varepsilon, m}}\\
&\qquad\qquad\qquad\qquad\qquad\qquad\qquad\qquad\qquad\qquad=\varepsilon^{m+2}\bigl(R_{\varphi}^{\text{int}, m}+\varepsilon^{-1}[R_{\varphi}^{\text{b}, m}]_\e\bigr),
\end{split}
\end{equation*}
where $R_{a}^{\text{int}, m},\,R_{\varphi}^{\text{int}, m},\,R_{a}^{\text{b}, m}$ and $ R_{\varphi}^{\text{b}, m}$ satisfy
\begin{equation}\label{S6eq5b}
\begin{split}
&\bigl\|(R_{a}^{\text{int}, m},\,\,R_{\varphi}^{\text{int}, m})\bigr\|_{ W^{s_0-2(m+3)}_{T_0}}
+\bigl\|(R_{\varphi}^{\text{b}, m},\,\, R_{\varphi}^{\text{b}, m})\bigr\|_{\mathcal{W}^{s_0-2(m+3)+\frac12}_{1, T_0}}\lesssim \cE_0.
\end{split}
\end{equation}

The above equations can also be written as
\begin{equation*}\label{appro-nls-1a}
\begin{split}
&\partial_t a^{\varepsilon, m}+\nabla\varphi^{\varepsilon, m}\cdot \nabla a^{\varepsilon, m}+\frac{1}{2}a^{\varepsilon, m}\Delta\varphi^{\varepsilon, m}\\
&\quad=  \varepsilon^{m+2}\bigl(\nabla\varphi_{m+2} \cdot \nabla a^{\varepsilon, m}+\frac{1}{2}a^{\varepsilon, m}\Delta\varphi_{m+2}+R_a^{\text{int}, m}+\varepsilon^{-1}[R_a^{\text{b}, m}]_\e\bigr),\\
&\partial_t \varphi^{\varepsilon, m}+\frac{1}{2}|\nabla \varphi^{\varepsilon, m}|^2+(a^{\varepsilon, m})^2-1-\frac{\varepsilon^2}{2}\frac{\Delta a^{\varepsilon, m}}{a^{\varepsilon, m}}\\
&\quad=\varepsilon^{m+2}\Bigl(\partial_t \varphi_{m+2}+\frac{1}{2}(2\nabla \varphi^{\varepsilon, m} \cdot\nabla\varphi_{m+2}-\varepsilon^{m+2}|\nabla\varphi_{m+2}|^2)+R_{\varphi}^{\text{int}, m}+\varepsilon^{-1}
[R_{\varphi}^{\text{b}, m}]_\e\Bigr),
\end{split}
\end{equation*}
from which, \eqref{S6eq5b} and  \eqref{space-norm-E-0a}, we infer
\begin{equation}\label{appro-nls-1b}
\begin{split}
&\partial_t a^{\varepsilon, m}+\nabla\varphi^{\varepsilon, m}\cdot \nabla a^{\varepsilon, m}+\frac{1}{2}a^{\varepsilon, m}\Delta\varphi^{\varepsilon, m}=\varepsilon^{m+2}\bigl(R_a^{\text{int}, m}+\varepsilon^{-1}[R_a^{\text{b}, m}]_\e\bigr),\\
&\partial_t \varphi^{\varepsilon, m}+\frac{1}{2}|\nabla \varphi^{\varepsilon, m}|^2+(a^{\varepsilon, m})^2-1-\frac{\varepsilon^2}{2}\frac{\Delta a^{\varepsilon, m}}{a^{\varepsilon, m}}=\varepsilon^{m+2}\bigl(R_{\varphi}^{\text{int}, m}+\varepsilon^{-1}[R_{\varphi}^{\text{b}, m}]_\e\bigr),
\end{split}
\end{equation}
where $R_{a}^{\text{int}, m},\,R_{\varphi}^{\text{int}, m},\,R_{a}^{\text{b}, m}$ and $ R_{\varphi}^{\text{b}, m}$ satisfy \eqref{S6eq5}.

 On the other hand, it is easy to observe that
\begin{equation*}\label{valid-soln-4}
\begin{split}
&GP(\Psi^{a, m})=R^{\varepsilon, m}e^{\frac{i}{\varepsilon}\varphi^{\varepsilon, m}}\quad \text{with}\quad R^{\varepsilon, m}=\bigl(-a^{\varepsilon, m}R^m_{\varphi}+\frac{\varepsilon^2}{2}\Delta a^{\varepsilon, m}\bigr)+i\varepsilon R^m_{a},\\
\end{split}
\end{equation*}
where
\begin{equation}\label{S6eq7}
\begin{split}
&R^m_{\varphi}=\partial_t \varphi^{\varepsilon, m}+\frac{1}{2}|\grad \varphi^{\varepsilon, m}|^2+(|a^{\varepsilon, m}|^2-1),\\
&R^m_{a}=\partial_t a^{\varepsilon, m}+\grad \varphi^{\varepsilon, m}\cdot \grad a^{\varepsilon, m}+\frac{1}{2} a^{\varepsilon, m} \Delta \varphi^{\varepsilon, m},
\end{split}
\end{equation}
which along with \eqref{appro-nls-1b} implies \eqref{S6eq5}. This ends the proof of Proposition \ref{lem-appro-solns-1}.
\end{proof}

Let us now turn to the proof of Proposition \ref{S2prop6}.

\begin{proof}[Proof of Proposition \ref{S2prop6}]
Once again we shall only present the
{\it a priori} estimates. Let $w$ and $\phi$ be  real-valued functions, we are going to seek the true solution of \eqref{NLS-0}
 with the form \eqref{S7eq1}.
In view of \eqref{S1eq1}, $(w,\phi)$ satisfies the following  initial condition
\begin{equation*}
\begin{split}
&w|_{t=0}=\e^{m+2}\bigl(a^{\rm in}_{m+2,0}+R_{a,0}^\e\bigr),\quad \phi|_{t=0}=\e^{m+2} R_{\vf,0}^\e \with \lim_{\e\to 0}\left\|(R_{a,0}^\e, \na R_{\vf,0}^\e)\right\|_{H^{s_0-2m-5}}=0.
\end{split}
\end{equation*}

We shall divide the proof of  Proposition \ref{S2prop6} into the following steps:

\no{\bf Step 1.} The derivation of the error equation

Substituting \eqref{expr-Phi-1} into \eqref{NLS-0} yields
\begin{equation}\label{eqns-tilde-w-0}
\begin{split}
&i\varepsilon\bigl(\partial_t{\fw}+(u^{\varepsilon,m}\cdot\nabla){\fw}+\frac{1}{2}{\fw}\nabla\cdot u^{\varepsilon,m}\bigr)+\frac{\varepsilon^2}{2}\Delta{\fw}-2\wr(a^{\varepsilon,m})^2\\
&\qquad\qquad\qquad\qquad\qquad\qquad\qquad\qquad\qquad\qquad=R_{\varphi}^m{\fw}-R^{\varepsilon, m}+Q^{\varepsilon}({\fw}),
\end{split}
\end{equation}
where $R_{\varphi}^m$ and $R^{\varepsilon, m}$ are defined in  \eqref{S6eq7} and \eqref{S6eq5}, and
\begin{equation}\label{def-Q-0}
\begin{split}
Q^\varepsilon({\fw})&\eqdefa \bigl(a^{\varepsilon,m}+{\fw}\bigr)\bigl(|a^{\varepsilon,m}+{\fw}|^2
-|a^{\varepsilon, m}|^2\bigr)-2\wr(a^{\varepsilon,m})^2\\
&=a^{\varepsilon,m}\bigl(\wr^2+\wi^2\bigr)
+{\fw}\bigl(\wr^2+\wi^2+2\,a^{\varepsilon,m}\,
\wr\bigr).
\end{split}
\end{equation}

Notice that $\wr=w\,\cos\phi+a^{\varepsilon,m}(\cos\phi-1)$, $\wi=(a^{\varepsilon,m}+w)\,\sin\phi$, we have the following
initial boundary condition for $(\wr,\wi):$
\beq \label{S7eq12a}
\begin{split}
& \wr|_{z=0}=0,\quad \wi|_{z=0}=0,\\
& \wr|_{t=0}=\e^{m+2} \bigl(a^{\rm in}_{m+2,0}+R_{a,0}^\e\bigr)\cos\left(\e^{m+2} R_{\vf,0}^\e\right)\\
&\qquad\qquad+a^{\e,m}(0)\left(\cos\left(\e^{m+2} R_{\vf,0}^\e\right)-1\right)\eqdefa w_{{\rm R},0}^\e,\\
& \wi|_{t=0}=\left(a^{\e,m}(0)+\e^{m+2} \bigl(a^{\rm in}_{m+2,0}+R_{a,0}^\e\bigr)\right)\sin\left(\e^{m+2} R_{\vf,0}^\e\right)\eqdefa w_{{\rm I},0}^\e.
\end{split}
\eeq
Then by taking separating the imaginary and real parts of \eqref{eqns-tilde-w-0}, we derive the system \eqref{w-phi-eqns} for $(\wr,\wi)$
with
\beq \label{defrap}
r^m_a\eqdefa \varepsilon\,R_{a}^{\text{int}, m}+[R_{a}^{\text{b}, m}]_\e \andf r^m_\vf\eqdefa \varepsilon\,R_{\varphi}^{\text{int}, m}+[R_{\varphi}^{\text{b}, m}]_\e.
\eeq

\no{\bf Step 2.}  The estimate of $\|\varepsilon(\wr, \wi)\|_{L^\infty_T(H^1)}$\

In view of \eqref{conv-id-1} and \eqref{w-phi-eqns}, we get, by using $L^2_+$ energy estimate, that
\begin{equation}\label{1-24-zero-1}
\begin{split}
\frac{1}{2}\frac{d}{dt}\int_{\mathbb{R}^3_+} \bigl(|\varepsilon\wi|^2&+| \varepsilon\wr|^2\bigr)\,dx=\int_{\mathbb{R}^3_+} \Bigl(2(a^{\varepsilon,m})^{2}+\frac{\varepsilon^2}{2}\frac{\Delta a^{\varepsilon, m}}{a^{\varepsilon, m}}
+\varepsilon^{m+2} r^m_\vf \Bigr) \wr | \varepsilon\wi \,dx\\
&+\int_{\mathbb{R}^3_+} \bigl( \varepsilon^{m+1} a^{\e,m}r_a^m-\text{Re}Q^\varepsilon({\fw})
\bigr) | \varepsilon\wi\,dx\\
&+\int_{\mathbb{R}^3_+}\Bigl(\bigl(\frac{\varepsilon^2}{2}\frac{\Delta a^{\varepsilon, m}}{a^{\varepsilon, m}}+\varepsilon^{m+2} r^m_\vf\bigr)\,\wi+\text{Im}Q^\varepsilon({\fw})-\varepsilon^{m+2} r^m_\vf\Bigr) \bigl| \varepsilon \wr\,dx.
\end{split}
\end{equation}
If $s_0-2m\geq 9,$ we deduce from \eqref{space-norm-E-0a} and \eqref{S6eq5a} that
\begin{equation*}
\begin{split}
&\bigl|\int_{\mathbb{R}^3_+} \Bigl(2(a^{\varepsilon,m})^{2}+\frac{\varepsilon^2}{2}\frac{\Delta a^{\varepsilon, m}}{a^{\varepsilon, m}}
+\varepsilon^{m+2} r^m_\vf \Bigr) \wr | \varepsilon\wi \,dx\bigr|\\
&\lesssim\Bigl(\|a^{\varepsilon, m}\|_{L^\infty_+}^2+\e^{m+2}\|r^m_\vf\|_{L^\infty_+}+\|\frac{\Delta a^{\varepsilon, m}}{a^{\varepsilon, m}}\Bigr\|_{L^\infty_+}\Bigr)\|\wr\|_{L^2_+}\, \|\varepsilon\wi\|_{L^2_+}\\
&\lesssim \|\wr\|_{L^2_+}^2+ \|\varepsilon\wi\|_{L^2_+}^2,
\end{split}
\end{equation*}
and
\begin{equation*}
\begin{split}
\bigl|&\int_{\mathbb{R}^3_+} \bigl( \varepsilon^{m+1} a^{\e,m}r_a^m-\text{Re}Q^\varepsilon({\fw})
\bigr) | \varepsilon\wi\,dx\bigr|\\
&\lesssim \bigl(\varepsilon^{m+1}\|a^{\e,m}\|_{L^\infty_+} \| r^m_{\varphi}\|_{L^2_+}+\|\text{Re}Q^\varepsilon({\fw})\|_{L^2_+} \bigr) \|\varepsilon\wi\|_{L^2_+}\\
&\lesssim \bigl(\varepsilon^{m+1} \mathcal{E}_0+\|\text{Re}Q^\varepsilon({\fw})\|_{L^2_+}\bigr) \|\varepsilon\wi\|_{L^2_+},
\end{split}
\end{equation*}
and
\begin{equation*}
\begin{split}
&\bigl|\int_{\mathbb{R}^3_+}\Bigl(\bigl(\frac{\varepsilon^2}{2}\frac{\Delta a^{\varepsilon, m}}{a^{\varepsilon, m}}+\varepsilon^{m+2} r^m_\vf\bigr)\,\wi+\text{Im}Q^\varepsilon({\fw})-\varepsilon^{m+2} r^m_\vf\Bigr) \bigl| \varepsilon \wr\,dx\bigr|\\
&\lesssim\Bigl((\frac{\varepsilon^2}{2}\Bigl\|\frac{\Delta a^{\varepsilon, m}}{a^{\varepsilon, m}}\Bigr\|_{L^\infty_+}+\varepsilon^{m+2} \|r^m_\vf\|_{L^\infty_+})\, \|\varepsilon\wi\|_{L^2_+}+\|\varepsilon\,\text{Im}Q^\varepsilon({\fw})\|_{L^2_+}+\varepsilon^{m+3} \|r^m_\vf\|_{L^2_+}\Bigr)\|\wr\|_{L^2_+}\\
&\lesssim \bigl(\mathcal{E}_0^{\frac{1}{2}}\, \|\varepsilon\wi\|_{L^2_+}+\|\varepsilon\,\text{Im}Q^\varepsilon({\fw})\|_{L^2_+}+\varepsilon^{m+2} \mathcal{E}_0^{\frac{1}{2}})\| \wr\|_{L^2_+}.
\end{split}
\end{equation*}
By inserting the above inequalities into \eqref{1-24-zero-1}, we get
\begin{equation}\label{1-24-zero-4}
\begin{split}
\frac{d}{dt}\|\varepsilon(\wr,\,\wi)\|_{L^2_+}^2\lesssim \|(\wr,\e\wi\|_{L^2_+}^2+\|\text{Re}Q^\varepsilon({\fw})\|_{L^2_+}^2
+\|\varepsilon\,\text{Im}Q^\varepsilon({\fw})\|_{L^2_+}^2 +\varepsilon^{2m+2} \mathcal{E}_0.
\end{split}
\end{equation}

On the other hand, by applying Lemma \ref{lem-linear-high-1} with $\hbar_1=\hbar_2=r^m_\vf,$ $f_1=\text{Im}Q^\varepsilon({\fw}),$
$f_2=0,$ $\chi_1=\frac{\varepsilon^2}{2}\frac{\Delta a^{\varepsilon, m}}{a^{\varepsilon, m}}
+\varepsilon^{m+2} r^m_\vf,$ $\chi_2=a^{\e,m}r^m_a,$ $g_1=\text{Re}Q^\varepsilon({\fw})$ and $g_2=0,$ we achieve
 \begin{equation}\label{1-29-7}
\begin{split}
\frac{d}{dt} \widetilde{\mathfrak{E}}_1 \lesssim & \mathfrak{E}_1+\varepsilon^{-2} \bigl(\|Q^\varepsilon({\fw})\|_{L^2_+}^2+
\|\varepsilon\nabla Q^\varepsilon({\fw})\|_{L^2_+}^2\bigr),
\end{split}
\end{equation}
 where
\begin{equation}\label{S8eq1}
\begin{split}
\mathfrak{E}_1\eqdefa &\varepsilon^{2m+2}\mathcal{E}_0+ \|\varepsilon(\,\wi,\,\wr)\|_{H^1}^2
+\|\wr \|_{L^2_+}^2 \andf\\
\widetilde{\mathfrak{E}}_1\eqdefa & C_0\,\varepsilon^{2m+2}\mathcal{E}_0
+ \frac{\varepsilon^2}{4}\bigl\|(\nabla\,\wi,\,\nabla\,\wr)\bigr\|_{L^2_+}^2+\int_{\mathbb{R}^3_+}\mathcal{S}_{u^{\varepsilon, m}}(\varepsilon\wi ) | \wr\,dx\\
&+\frac12\int_{\mathbb{R}^3_+}\Bigl(2(a^{\varepsilon,m})^{2}
+\frac{\varepsilon^2}{2}\frac{\Delta a^{\varepsilon, m}}{a^{\varepsilon, m}}
+\varepsilon^{m+2} r^m_\vf\Bigr)\wr^2\,dx\\
&+\frac{1}{2}\int_{\mathbb{R}^3_+} \Bigl(\Bigl(\frac{\varepsilon^2}{2}\frac{\Delta a^{\varepsilon, m}}{a^{\varepsilon, m}}- \varepsilon^{m+2} r_\vf^m\Bigr)\wi^2- 2\varepsilon^{m+2}r^m_\vf\,  \wi-2\varepsilon^{m+1} a^{\e,m}r_a^m \, \wr\Bigr)\, dx.
\end{split}
\end{equation}

\no{\bf Step 3.} High-order tangential derivatives estimates\

The main result states as follow, the proof of which will be postponed after the proof of Proposition \ref{S2prop6}.
\begin{lem}\label{S8lem5}
{\sl Let $s_0\geq 2m+9+N$ be an integer, we denote \begin{equation}\label{diff-t1-ener-2}
\begin{split}
\mathfrak{E}_{N}\eqdefa&\varepsilon^{2m+2}\mathcal{E}_0+\sum_{j=0}^{N-1}\bigl(\|{\Ta}^j\wr\|_{L^2_+}^2 +\|\varepsilon{\Ta}^j(\wr, \,\wi)\|_{H^1}^2\bigr)\andf\\
 \widetilde{\mathfrak{E}}_{N}\eqdefa & C_N\varepsilon^{2m+2}\mathcal{E}_0+\sum_{j=0}^{N-1}\dot{{E}}_j \with\\
\dot{E}_j\eqdefa & \Bigl\{\frac{1}{4}\bigl\|\varepsilon\left(\nabla\,{\Ta}^j\,\wr,\,
\nabla\,{\Ta}^j\,\wi\right)\bigr\|_{L^2_+}^2+\int_{\mathbb{R}^3_+}(a^{\varepsilon,m})^{2}
|{\Ta}^j\,\wr|^2\,dx\\
&+\frac{1}{2}\int_{\mathbb{R}^3_+} \Bigl(\frac{\varepsilon^2}{2}\frac{\Delta a^{\varepsilon, m}}{a^{\varepsilon, m}}+\varepsilon^{m+2} r^m_\vf\Bigr)\bigl(|{\Ta}^j\,\wi|^2+|{\Ta}^j\,\wr|^2\bigr)\,dx\\
&-\int_{\mathbb{R}^3_+} \bigl(\varepsilon^{m+2} \cT^jr^m_a +\frac12R_{\partial_t, 1, j}\bigr)| {\Ta}^j\wi\,dx\\
&+\int_{\mathbb{R}^3_+}\bigl(\mathcal{S}_{u^{\varepsilon, m}}(\varepsilon {\Ta}^j\,\wi )-\varepsilon^{m+1}\cT^j(a^{\e,m}r^m_a)+\frac12
R_{\partial_t, 2, j}\bigr) | {\Ta}^j\,\wr\,dx,
\end{split}
\end{equation} where
\begin{equation}\label{w-phi-eqns-j-rt}
\begin{split}
&R_{\partial_t, 1, j}\eqdefa [{\Ta}^j;\,\mathcal{S}_{u^{\varepsilon, m}}](\varepsilon\,\wr)
-\Bigl[{\Ta}^j;\,\frac{\varepsilon^2}{2}\frac{\Delta a^{\varepsilon, m}}{a^{\varepsilon, m}}+\varepsilon^{m+2} r^m_\vf\Bigr] \wi,\\
&R_{\partial_t, 2, j}\eqdefa [{\Ta}^j;\,\mathcal{S}_{u^{\varepsilon, m}}]
(\varepsilon\wi)
+\Bigl[{\Ta}^j;\,2(a^{\varepsilon,m})^{2}+\frac{\varepsilon^2}{2}\frac{\Delta a^{\varepsilon, m}}{a^{\varepsilon, m}}
+\varepsilon^{m+2} r^m_\vf \Bigr]\wr.
\end{split}
\end{equation}
Then we have
\begin{equation}\label{diff-t1-ener-3}
\begin{split}
&\frac{d}{dt} \widetilde{\mathfrak{E}}_N
\leq C\, \mathfrak{E}_N
+\varepsilon^{-2}\sum_{j=0}^{N-1}\bigl(\|{\Ta}^j Q^\varepsilon({\fw})\|_{L^2_+}^2+
\|\varepsilon\nabla\,{\Ta}^jQ^\varepsilon({\fw})\|_{L^2_+}^2\bigr).
\end{split}
\end{equation}
}
\end{lem}

\no{\bf Step 4.}
 {Estimates of nonlinear terms}\

 \begin{lem}\label{S8lem6}
 {\sl Let $N\geq 4$ and $s_0\geq 2m+6+N$  be integers. Then one has
 \begin{equation}\label{nonl-est-15}
\begin{split}
\varepsilon^{-2}\sum_{j=0}^{N-1}\bigl(\|{\Ta}^j Q^\varepsilon({\fw})\|_{L^2_+}^2+
\|\varepsilon\nabla\,{\Ta}^jQ^\varepsilon({\fw})\|_{L^2_+}^2\bigr)\lesssim \varepsilon^{-8}\mathfrak{E}_{N}\bigl(1+\varepsilon^{-2}\mathfrak{E}_{N}\bigr)\,\mathfrak{E}_{N}.
\end{split}
\end{equation}}
\end{lem}

The proof of this lemma will be postponed below.

Next, we claim that
\begin{equation}\label{equiv-2}
\widetilde{\mathfrak{E}}_N\thicksim  \mathfrak{E}_N, \quad \text{i.e.} \quad C_1^{-1} \mathfrak{E}_N \leq \widetilde{\mathfrak{E}}_N \leq C_1\mathfrak{E}_N
\end{equation}
for some positive constant $C_1.$

We first get, by a similar the proof of \eqref{S7eq26} and \eqref{diff-t1-ener-j-9}, that
\begin{equation*}\label{diff-t1-rem-1}
\begin{split}
\bigl|\int_{\mathbb{R}^3_+} \bigl(R_{\partial_t, 2, j} | {\Ta}^j\, \,\wr -R_{\partial_t, 1, j} | {\Ta}^j\, \wi\bigr)\, dx\bigr|
\lesssim &  \mathcal{E}_0^{\frac{1}{2}}\|{\Ta}^j\, \varepsilon\,\wi\|_{L^2_+}\sum_{k=0}^{j-1}\|\varepsilon\,
{\Ta}^k\,(\wr,\,\wi)\|_{H^1}\\
&+\mathcal{E}_0^{\frac{1}{2}}\|{\Ta}^j\, \,\wr\|_{L^2_+}\sum_{k=0}^{j}
\bigl(\|\varepsilon\,{\Ta}^k\,\wi\|_{H^1}
+
\|{\Ta}^k\,\wr\|_{L^2_+}\bigr).
\end{split}
\end{equation*}
Whereas it follows from Lemmas \ref{lem-est-singu-1} and \ref{lem-est-conv-2}  that
\begin{equation*}\label{equiv-1}
\begin{split}
&\bigl|\int_{\mathbb{R}^3_+} \Bigl(\frac{\varepsilon^2}{2}\frac{\Delta a^{\varepsilon, m}}{a^{\varepsilon, m}}+\varepsilon^{m+2} r^m_\vf\Bigr)\bigl(|{\Ta}^j\,\wi|^2+|{\Ta}^j\,\wr|^2\bigr)\,dx\bigr|\\
&\lesssim \bigl(\mathcal{E}_0^{\frac{1}{2}}+\e^m\|r^m_\vf\|_{L^\infty_+}\bigr)\bigl\|\varepsilon\,{\Ta}^j(\wr,\wi)\bigr\|_{H^1}^2
\lesssim \mathcal{E}_0^{\frac{1}{2}}\bigl\|\varepsilon\,{\Ta}^j(\wr,\wi)\bigr\|_{H^1}^2,
\end{split}
\end{equation*}
and
\begin{equation*}
\begin{split}
\bigl|\int_{\mathbb{R}^3_+}\mathcal{S}_{u^{\varepsilon, m}}(\varepsilon {\Ta}^j\,\wi ) | {\Ta}^j\,\wr\,dx\bigr|
\lesssim   \mathcal{E}_{0}^{\frac{1}{2}}\|\varepsilon{\Ta}^j\,\wi\|_{H^1}\|{\Ta}^j\,\wr\|_{L^2_+},
\end{split}
\end{equation*}
and
\beno
\e^{m+1}\bigl|\int_{\mathbb{R}^3_+} \bigl(\varepsilon \cT^jr^m_a | {\Ta}^j\wi-\cT^j(a^{\e,m}r^m_a) | {\Ta}^j\,\wr\bigr)\,dx\bigr|
\lesssim \cE_0^{\frac12}\e^{m+1}\bigl(\|\varepsilon{\Ta}^j\,\wi\|_{L^2}+\|{\Ta}^j\,\wr\|_{L^2_+}\bigr).
\eeno
Finally, by virtue of \eqref{space-norm-E-0a}, we have
$\|a^{\varepsilon,m}-1\|_{L^\infty_+}\lesssim \mathcal{E}_{0}^{\frac{1}{2}},$ from which and \eqref{diff-t1-ener-2}, we deduce that
\beno
\begin{split}
\dot{E}_j\geq & \frac{1}{4}\left(\bigl\|\varepsilon\left(\nabla\,{\Ta}^j\,\wr,\,
\nabla\,{\Ta}^j\,\wi\right)\bigr\|_{L^2_+}^2+
\|{\Ta}^j\,\wr\|_{L^2_+}^2\right)\\
&- C\cE_0^{\frac12}\bigl(\|\varepsilon\left({\Ta}^j\,\wr,\,
{\Ta}^j\,\wi\right)\bigr\|_{H^1}^2+\|{\Ta}^j\,\wr\|_{L^2_+}^2\bigr)-C\cE_0\e^{2m+2}.
\end{split}
\eeno
This ensures \eqref{equiv-2} as long as $c$ in \eqref{S3eq27} and $\e$ are small enough and $C_N$ in \eqref{diff-t1-ener-2} satisfies
$C_N\geq C+1.$

Now we are in a position to complete the proof of Proposition \ref{S2prop6}. Indeed
by inserting \eqref{nonl-est-15} into \eqref{diff-t1-ener-3}, we find
\begin{equation}\label{nonl-est-16}
\begin{split}
&\frac{d}{dt} \widetilde{\mathfrak{E}}_N
\leq C\bigl( \mathfrak{E}_N +\varepsilon^{-8}\mathfrak{E}_{N}\bigl(1+\varepsilon^{-2}\mathfrak{E}_{N}\bigr)\,\mathfrak{E}_{N}\bigr).
\end{split}
\end{equation}

 Let $T_0$ be determined by  Proposition \ref{S2prop1} and  $N \geq 4$,
we define
 \begin{equation}\label{def-T-max}
\begin{split}
T^{\star}_2\eqdefa\sup\bigl\{T' \in (0,\, T],\quad \mathfrak{E}_{N}(t)\leq \frak{C}\mathcal{E}_0 \varepsilon^{2m+2}\quad \forall\quad t\in [0, T']\,\bigr\}
\end{split}
\end{equation}
for some positive constant $\frak{C}$ to be determined later on.
 We are going to prove that $T^\star_2=T_0$ provided that $c$ in \eqref{S3eq27} and $\e$ are sufficiently small.

Indeed  for $m\geq 4$ and $\e\leq \Bigl(\frac1{4\frak{C}\cE_0}\Bigr)^{\frac1{2(m-3)}},$  one has
\begin{equation*}\label{def-T-max-1}
\begin{split}
& \varepsilon^{-8}\mathfrak{E}_{N}(t)\leq \frak{C}\mathcal{E}_0 \varepsilon^{2m-6}\leq \frac{1}{4}\quad \forall\ t \in [0, T^{\star}_2],
\end{split}
\end{equation*}
 from which, \eqref{equiv-2} and \eqref{nonl-est-16},   we infer
\begin{equation}\label{est-max-t-2}
\begin{split}
&\frac{d}{dt} \widetilde{\mathfrak{E}}_N \leq  2C\, \mathfrak{E}_N\leq  2CC_1 \widetilde{\mathfrak{E}}_N.
\end{split}
\end{equation}
Thanks to \eqref{S7eq12a}, we get, by
applying Gronwall's inequality to \eqref{est-max-t-2}, that
\begin{equation}\label{est-max-t-4}
\begin{split}
&\widetilde{\mathfrak{E}}_N(t) \leq  Ce^{2CC_1T_0}\,\mathcal{E}_0\varepsilon^{2m+2}\quad\forall t\leq T^\star_2.
\end{split}
\end{equation}
Then by taking $\frak{C}=2CC_1e^{2CC_1T_0}$ in \eqref{def-T-max}, we deduce from
\eqref{equiv-2} that
\begin{equation*}\label{est-max-t-5}
\begin{split}
&\mathfrak{E}_N(t) \leq C_1\widetilde{\mathfrak{E}}_N(t) \leq  \frac{1}{2}\frak{C}\mathcal{E}_0\varepsilon^{2m+2} \quad\forall t\leq T^\star_2.
\end{split}
\end{equation*}
This contradicts with \eqref{def-T-max}, and this in turn shows that $T^\star_2=T_0,$ moreover, there holds \eqref{zhang4}.
This completes the proof of Proposition \ref{S2prop6}. \end{proof}

Proposition \ref{S2prop6} has been proved provided that we present  the proof of Lemmas \ref{S8lem5} and \ref{S8lem6}.

\begin{proof}[Proof of Lemma \ref{S8lem5}]
By applying ${\Ta}^j$ with $j \in \{ 1, 2, ..., N-1\ \}$ to \eqref{w-phi-eqns}, we find
\begin{equation}\label{w-phi-eqns-j}
\begin{cases}
&\varepsilon\bigl(\partial_t
+\cS_{u^{\varepsilon, m}}(\cdot)\bigr)({\Ta}^j\,\wr)+\frac{\varepsilon^{2}}{2}\Delta \left({\Ta}^j\,\wi\right)\\
&\quad=\Bigl(\frac{\varepsilon^2}{2}\frac{\Delta a^{\varepsilon, m}}{a^{\varepsilon, m}}+\varepsilon^{m+2} r^m_\vf
\Bigr) {\Ta}^j\wi-\varepsilon^{m+2} {\Ta}^jr^m_\vf+{\Ta}^j\text{Im}Q^\varepsilon({\fw})
-R_{\partial_t, 1,j},\\
&\varepsilon\bigl( \partial_{t}+\mathcal{S}_{u^{\varepsilon, m}}(\cdot)\bigr)({\Ta}^j\wi)-\frac{\varepsilon^{2}}{2} \Delta ({\Ta}^j\wr )+\Bigl(2(a^{\varepsilon,m})^{2}+\frac{\varepsilon^2}{2}\frac{\Delta a^{\varepsilon, m}}{a^{\varepsilon, m}}
+\varepsilon^{m+2} r^m_\vf \Bigr){\Ta}^j\,\wr\\
&\quad =\varepsilon^{m+1} {\Ta}^j( a^{\e,m}r_a^m)-{\Ta}^j\text{Re}Q^\varepsilon({\fw})
-R_{\partial_t, 2, j},\\
&{\Ta}^j\,\wr|_{z=0}=0,\quad {\Ta}^j\wi|_{z=0}=0,
\end{cases}
\end{equation}
with $R_{\partial_t, 1, j}$ and $R_{\partial_t, 2, j}$ being given by \eqref{w-phi-eqns-j-rt}.

Notice that
\begin{equation*}\label{diff-t1-ener-j-1}
\begin{split}
\int_{\mathbb{R}^3_+} \bigl(R_{\partial_t, 2, j} | {\Ta}^j\partial_t \,\wr -R_{\partial_t, 1, j} | {\Ta}^j\partial_t \wi\bigr)\, dx
=&\frac{d}{dt}\int_{\mathbb{R}^3_+} \bigl(R_{\partial_t, 2, j} | {\Ta}^j\, \,\wr -R_{\partial_t, 1, j} | {\Ta}^j\, \wi\bigr)\, dx\\
&-\int_{\mathbb{R}^3_+} \bigl(\partial_t\,R_{\partial_t, 2, j} | {\Ta}^j \,\wr -\partial_t\,R_{\partial_t, 1, j} | {\Ta}^j\, \wi\bigr)\, dx.
\end{split}
\end{equation*}
Then for $\dot E_j$ given by \eqref{diff-t1-ener-2},
we get, by applying Lemma \ref{lem-linear-high-1}, that
\begin{equation}\label{diff-t1-ener-j-1}
\begin{split}
\frac{d}{dt}&\dot{E}_j(t)-\int_{\mathbb{R}^3_+} \bigl(\partial_t\,R_{\partial_t, 2, j} \, {\Ta}^j \,\wr -\partial_t\,R_{\partial_t, 1, j} \,{\Ta}^j\, \wi\bigr)\, dx\\
\lesssim &\mathcal{E}_0\,\varepsilon^{2m+2}+\|\varepsilon{\Ta}^j(\wr,\wi)\|_{H^1}^2+\|{\Ta}^j\,\wr \|_{L^2_+}^2\\
&+\|(R_{\partial_t, 1, j},R_{\partial_t, 2, j})\|_{L^2_+}^2+\varepsilon^{-2}\bigl(\|{\Ta}^j Q^\varepsilon({\fw})\|_{L^2_+}^2+
\|\varepsilon\nabla\,{\Ta}^jQ^\varepsilon({\fw})\|_{L^2_+}^2\bigr).
\end{split}
\end{equation}
 Observing that
\beno
\begin{split}
\p_t\bigl[{\Ta}^j; \cS_u\bigr] f=&\p_t{\Ta}^j\cS_u(f)-\cS_u({\Ta}^j\p_tf)-\cS_{\p_tu}({\Ta}^jf)\\
=&\bigl[\p_t{\Ta}^j; \cS_u\bigr] f-\cS_{\p_tu}({\Ta}^jf),
\end{split}
\eeno
and
\beno
\begin{split}
\p_t\bigl[{\Ta}^j; g\bigr] f=&\p_t{\Ta}^j(gf)- g{\Ta}^j\p_tf-\p_tg {\Ta}^jf=\bigl[\p_t{\Ta}^j; g\bigr] f- \p_tg{\Ta}^jf.
\end{split}
\eeno
In view of \eqref{w-phi-eqns-j-rt}, we write
\begin{equation*}
\begin{split}
\p_tR_{\partial_t, 2, j}=&\bigl[\p_t{\Ta}^j;\,\mathcal{S}_{u^{\varepsilon, m}}\bigr]
(\varepsilon\wi)
+\Bigl[\p_t{\Ta}^j;\,2(a^{\varepsilon,m})^{2}+\frac{\varepsilon^2}{2}\frac{\Delta a^{\varepsilon, m}}{a^{\varepsilon, m}}
+\varepsilon^{m+2} r^m_\vf\Bigr]\wr\\
&-\mathcal{S}_{\p_tu^{\varepsilon, m}}\left(\e{\Ta}^{j}\wi\right)-\p_t\Bigl(2(a^{\varepsilon,m})^{2}+\frac{\varepsilon^2}{2}\frac{\Delta a^{\varepsilon, m}}{a^{\varepsilon, m}}
+\varepsilon^{m+2} r^m_\vf\Bigr){\Ta}^{j}\wr.
\end{split}
\end{equation*}
It follows from  Lemmas \ref{lem-est-singu-1} and \ref{lem-est-conv-2} that for $s_0-2m-9\geq j+1,$
\begin{equation} \label{S7eq26}
\begin{split}
&\bigl|\int_{\mathbb{R}^3_+}\partial_t\,R_{\partial_t, 2, j} | {\Ta}^j \,\wr\, dx\bigr|\lesssim
\mathcal{E}_0^{\frac{1}{2}}\sum_{k=0}^{j}\bigl(\|{\Ta}^k\varepsilon\wi\|_{H^1}+\|{\Ta}^k \wr\|_{L^2_+}\bigr)\,\|{\Ta}^j \,\wr\|_{L^2_+}.
\end{split}
\end{equation}
Along the same line, we write
\begin{equation}\label{S7eq27}
\begin{split}
\p_tR_{\partial_t, 1, j}\eqdefa& \bigl[\p_t{\Ta}^{j};\,\mathcal{S}_{u^{\varepsilon, m}}\bigr](\varepsilon\,\wr)-\mathcal{S}_{\p_tu^{\varepsilon, m}}\left(\e{\Ta}^{j}\wr\right)\\
&
-\frac{\varepsilon^2}{2}\Bigl[\p_t{\Ta}^{j};\,\frac{\Delta a^{\varepsilon, m}}{a^{\varepsilon, m}}\Bigr] \wi+\frac{\varepsilon^2}{2}\p_t\Bigl(\frac{\Delta a^{\varepsilon, m}}{a^{\varepsilon, m}}\Bigr){\Ta}^{j}\wi+\varepsilon^{m-1} \p_t\bigl[{\Ta}^{j}; r^m_\vf\bigr]\wi.
\end{split}
\end{equation}
Notice that
$$
\bigl[{\Ta}^{j+1};\,\mathcal{S}_{u^{\varepsilon, m}}\bigr](\varepsilon\,\wr)
=\varepsilon\sum_{k=1}^{j+1}C_{j}^k\mathcal{S}_{{\Ta}^{k}u^{\varepsilon, m}}({\Ta}^{j+1-k}\wr).
$$
In view of \eqref{conv-id-1}, we write
\begin{equation*}
\begin{split}
\int_{\mathbb{R}^3_+}  \bigl[{\Ta}^{j+1};\,\mathcal{S}_{u^{\varepsilon, m}}\bigr](\varepsilon\,\wr) | {\Ta}^j\, \wi\, dx=&
\varepsilon\sum_{k=1}^{j+1}C_{j}^k\int_{\mathbb{R}^3_+}  \mathcal{S}_{{\Ta}^{k}u^{\varepsilon, m}}({\Ta}^{j+1-k}\wr) | {\Ta}^j\, \wi\, dx\\
=&-\sum_{k=1}^{j+1}C_{j}^k\int_{\mathbb{R}^3_+}  \mathcal{S}_{{\Ta}^{k}u^{\varepsilon, m}}(\varepsilon{\Ta}^j\, \wi) | {\Ta}^{j+1-k}\wr\, dx,
\end{split}
\end{equation*}
from which and  \eqref{est-conv-2-1}, we infer
\begin{equation*}
\begin{split}
\bigl|\int_{\mathbb{R}^3_+}  \bigl[{\Ta}^{j+1};\,\mathcal{S}_{u^{\varepsilon, m}}\bigr](\varepsilon\,\wr) | {\Ta}^j\, \wi\, dx\bigr|
\lesssim   \mathcal{E}_0^{\frac{1}{2}}\|\varepsilon{\Ta}^j\, \wi\|_{H^1}\sum_{k=0}^{j}\,\|{\Ta}^{k}\wr\|_{L^2_+}.
\end{split}
\end{equation*}
The same estimate holds for $\int_{\mathbb{R}^3_+} \mathcal{S}_{\p_tu^{\varepsilon, m}}\left(\e{\Ta}^{j}\wr\right) | {\Ta}^j\, \wi\, dx.$

While applying Lemma \ref{lem-est-singu-1} yields
\begin{equation*}
\begin{split}
\bigl|\int_{\mathbb{R}^3_+}  \Bigl(-\frac{\varepsilon^2}{2}\bigl[\p_t{\Ta}^{j};\,\frac{\Delta a^{\varepsilon, m}}{a^{\varepsilon, m}}\bigr] \wi+\frac{\varepsilon^2}{2}\p_t\Bigl(\frac{\Delta a^{\varepsilon, m}}{a^{\varepsilon, m}}\Bigr){\Ta}^{j}\wi\Bigr) \bigl|& {\Ta}^j\, \wi\, dx|\\
\lesssim & \mathcal{E}_0^{\frac{1}{2}}\|\varepsilon{\Ta}^j\, \wi\|_{H^1}\sum_{k=0}^{j}\,\|\varepsilon{\Ta}^{k}\wi\|_{H^1}.
\end{split}
\end{equation*}
And it follows from \eqref{S6eq5a} and \eqref{defrap} that
\begin{equation*}
\begin{split}
&\varepsilon^{m+2}\bigl|\int_{\mathbb{R}^3_+}  \p_t\bigl[{\Ta}^{j}; r^m_\vf\bigr]\wi |{\Ta}^j\, \wi\, dx\bigr|\\
&\lesssim \varepsilon^{m}\sum_{k=1}^{j}\bigl( \|{\Ta}^{k}\partial_t r^m_\vf\e{\Ta}^{j-k}\wi\|_{L^2_+}+ \|{\Ta}^{k} r^m_\vf\,\varepsilon{\Ta}^{j-k}\partial_t\wi\|_{L^2_+}\bigr)
\|\varepsilon{\Ta}^j\, \wi\|_{L^2_+}\\
&\lesssim \varepsilon^{m}\sum_{k=1}^{j}\mathcal{E}_0^{\frac{1}{2}}\bigl(
\|\varepsilon\,{\Ta}^{j-k}\wi\|_{L^2_+}+
\|\varepsilon{\Ta}^{j-k}\partial_t\wi\|_{L^2_+}\bigr)
\|\varepsilon{\Ta}^j\, \wi\|_{L^2_+}.
\end{split}
\end{equation*}
As a result, thanks to \eqref{S7eq27}, we conclude
\begin{equation}\label{diff-t1-ener-j-9}
\begin{split}
\bigl|\int_{\mathbb{R}^3_+} \partial_t\,R_{\partial_t, 1, j} | {\Ta}^j\, \wi\, dx\bigr|\lesssim \mathcal{E}_0^{\frac{1}{2}}\|\varepsilon\,{\Ta}^j\, \wi\|_{H^1}
\sum_{k=0}^{j}\bigl(\|{\Ta}^{k}\wr\|_{L^2_+}+
\|{\Ta}^{k}\varepsilon\wi\|_{H^1}\bigr).
\end{split}
\end{equation}
Finally, it follows from  Lemmas \ref{lem-est-singu-1} and \ref{lem-est-conv-2}  that
\beq\label{zhang1}
\|R_{\partial_t, 1, j}\|_{L^2_+}+\|R_{\partial_t, 2, j}\|_{L^2_+}\lesssim \sum_{k=0}^{j-1}\bigl(\|{\Ta}^{k}\wr\|_{L^2_+}+
\|{\Ta}^{k}\varepsilon(\wr,\wi)\|_{H^1}\bigr).
\eeq

By inserting the estimates \eqref{S7eq26}, \eqref{diff-t1-ener-j-9} and \eqref{zhang1} into \eqref{diff-t1-ener-j-1}
gives rise to
\begin{equation}\label{zhang2}
\begin{split}
\frac{d}{dt}\dot{E}_j(t)
\lesssim &\mathcal{E}_0\,\varepsilon^{2m+2}+\sum_{k=0}^j\bigl(\|\varepsilon{\Ta}^k(\wr,\wi)\|_{H^1}^2+\|{\Ta}^k\,\wr \|_{L^2_+}^2\bigr)\\
&+\varepsilon^{-2}\bigl(\|{\Ta}^j Q^\varepsilon({\fw})\|_{L^2_+}^2+
\|\varepsilon\nabla\,{\Ta}^jQ^\varepsilon({\fw})\|_{L^2_+}^2\bigr).
\end{split}
\end{equation}
Summing up \eqref{1-29-7} and \eqref{zhang2} for $j$ from $1$ to $N-1$ leads to \eqref{diff-t1-ener-3}.
This completes the proof of Lemma \ref{S8lem5}.
\end{proof}

\begin{proof}[Proof of Lemma \ref{S8lem6}]
Let's estimate the nonlinear terms in \eqref{diff-t1-ener-3}. Indeed in view of
 \eqref{def-Q-0}, we have
\begin{equation}\label{def-Q-0a}
\begin{split}
&\text{Re}Q^\varepsilon({\fw})
=3a^{\varepsilon,m}\wr^2+a^{\varepsilon,m}\wi^2
+\wr^3+\wr\,\wi^2,\\
&\text{Im}Q^\varepsilon({\fw})=\wi\wr^2+\wi^3+2\,a^{\varepsilon,m}\,
\wi\,\wr.
\end{split}
\end{equation}

Recalling the Sobolev embedding and the classical interpolation inequality that
\beno
\|f\|_{L^\infty_\v(L^2_\h)}^2\lesssim \varepsilon^{-1}\|f\|_{L^2_+}\|\varepsilon\p_zf\|_{L^2},\
\quad\|f\|_{L^2_\v(L^\infty_\h)}^2\lesssim  \|f\|_{L^2_+}\| (1+\na_\h^2)f\|_{L^2_+},\eeno
and
\begin{equation}\label{aniso-sobolev-infty}
\begin{split}
\|f\|_{L^\infty_+}^2\lesssim &\|f\|_{L^\infty_\v(L^2_\h)}\|f\|_{L^\infty_\v(H^2_\h)}\\
\lesssim & \|f\|_{L^2_+}^{\frac{1}{2}}\|f\|_{H^1}^{\frac{1}{2}}
\|(1+{\na}_\h^2)f\|_{L^2_+}^{\frac{1}{2}} \|(1+{\na}_\h^2)\,f\|_{H^1}^{\frac{1}{2}}.
\end{split}
\end{equation}
As a result, it comes out
\begin{equation}\label{nonl-est-6a}
\begin{split}
&\|\wr\|_{L^\infty_+}^2+\|{\Ta}\wr\|_{L^\infty_+}^2
\lesssim \varepsilon^{-1}\,\mathfrak{E}_4,\quad \|{\Ta}^{j-\ell}\,\wr\|_{L^\infty_\v(L^2_\h)}^2\lesssim \varepsilon^{-1}\mathfrak{E}_{j-\ell+1},\\
&\|{\Ta}^{\ell}\wr\|_{L^2_\v(L^\infty_\h)}^2\lesssim \mathfrak{E}_{\ell+2},\quad \|{\Ta}^{\ell}\nabla\wr\|_{L^2_\v(L^\infty_\h)}^2\lesssim \varepsilon^{-2}\mathfrak{E}_{\ell+3},
\end{split}
\end{equation}
and
\begin{equation}\label{nonl-est-6b}
\begin{split}
&\|\wi\|_{L^\infty_+}^2+\|{\Ta}\wi\|_{L^\infty_+}^2
\lesssim \varepsilon^{-2}\,\mathfrak{E}_4,\quad \|{\Ta}^{j-\ell}\,\wi\|_{L^\infty_\v(L^2_\h)}^2\lesssim \varepsilon^{-2}\mathfrak{E}_{j-\ell+1},\\
&\|{\Ta}^{\ell}\wi\|_{L^2_\v(L^\infty_\h)}^2\lesssim \varepsilon^{-2}\mathfrak{E}_{\ell+2},\quad \|{\Ta}^{\ell}\nabla\wi\|_{L^2_\v(L^\infty_\h)}^2\lesssim \varepsilon^{-2}\mathfrak{E}_{\ell+3}.
\end{split}
\end{equation}
It follows from \eqref{nonl-est-6a}  and \eqref{nonl-est-6b} that for $j\leq N-1,$
 \begin{equation}\label{nonl-est-19}
\begin{split}
\|{\Ta}^{j}\,(\wr^2)\|_{L^2_+}^2
\lesssim &\sum_{k=0}^{(j-2)_+}\|{\Ta}^{j-\ell}\,\wr\|_{L^\infty_\v(L^2_\h)}^2
\|{\Ta}^{\ell}\wr\|_{L^2_\v(L^\infty_\h)}^2\\
&\,+
\bigl(\|{\Ta}^{j-1}\,\wr\|_{L^2_+}^2
+\|{\Ta}^{j}\,\wr\|_{L^2_+}^2\bigr)
\bigl(\|\wr\|_{L^\infty_+}^2+\|{\Ta}\wr\|_{L^\infty_+}^2\bigr)\\
\lesssim &\varepsilon^{-1}\Bigl(  \sum_{\ell=0}^{(j-2)_+}\mathfrak{E}_{j-\ell+1}
\mathfrak{E}_{\ell+2}+\mathfrak{E}_4\mathfrak{E}_j\Bigr)\lesssim  \varepsilon^{-1}\mathfrak{E}_{N}^2\quad\mbox{if}\quad N\geq 4.
\end{split}
\end{equation}
Along the same line, we have
\begin{equation*}\label{nonl-est-q-1}
\begin{split}
&\|{\Ta}^{j}\,(\wi^2)\|_{L^2_+}^2+\|{\Ta}^{j}\,(\wi\wr)\|_{L^2_+}^2  \lesssim\varepsilon^{-4}\mathfrak{E}_{N}^2,\\
&\|{\Ta}^{j}\,(\wi^3,\,\wr^3,\,\wi^2\wr,\,\wi\wr^2)\|_{L^2_+}^2  \lesssim\varepsilon^{-6}\mathfrak{E}_{N}^3.
\end{split}
\end{equation*}
Then by virtue of \eqref{space-norm-E-0a}, for $s_0-2m\geq j+7,$ we have
\begin{equation*}\label{nonl-est-1a}
\begin{split}
\|{\Ta}^{j}\,(a^{\varepsilon,m}\wr^2)\|_{L^2_+}^2\lesssim &\sum_{\ell=0}^{j}\|{\Ta}^{j-\ell}\,a^{\varepsilon,m}\|_{L^\infty_+}^2
\|{\Ta}^{\ell} \wr^2\|_{L^2_+}^2
\lesssim \e^{-1}\sum_{k=0}^{N-1}\mathfrak{E}_{\ell+1}^2\lesssim \e^{-1} \mathfrak{E}_{N}^2.
\end{split}
\end{equation*}
Similarly, we have
\begin{equation*}\label{nonl-est-1b}
\begin{split}
\|{\Ta}^{j}(a^{\varepsilon,m}\wi\wr)\|_{L^2_+}^2+\|{\Ta}^{j}(a^{\varepsilon,m}\wi^2)\|_{L^2_+}^2\lesssim \varepsilon^{-4}\mathfrak{E}_{N}^2.
\end{split}
\end{equation*}
Therefore, we conclude that
\beq\label{nonl-est-1c}
\varepsilon^{-2}\sum_{j=0}^{N-1}
\bigl(\|{\Ta}^j\text{Re}(Q^\varepsilon({\fw})\|_{L^2_+}^2
+\|{\Ta}^j\text{Im}(Q^\varepsilon({\fw})\|_{L^2_+}^2\bigr)\lesssim \varepsilon^{-6}\mathfrak{E}_{N}^2+\varepsilon^{-8}\mathfrak{E}_{N}^3.
\eeq

On the other hand, we observe that
\begin{equation}\label{nonl-est-17}
\begin{split}
\|{\Ta}^{j}\,\nabla(a^{\varepsilon,m}\wr^2)\|_{L^2_+}^2
\lesssim &\sum_{\ell=0}^{j}\bigl(\|{\Ta}^{j-\ell}\,\nabla\,a^{\varepsilon,m}\|_{L^\infty_+}^2
\|{\Ta}^{\ell} (\wr)^2\|_{L^2_+}^2\\
&\qquad+\|{\Ta}^{j-\ell}\,a^{\varepsilon,m}\|_{L^\infty_+}^2
\|{\Ta}^{\ell} (\wr\,\nabla\wr)\|_{L^2_+}^2\bigr)\\
\lesssim& \varepsilon^{-2}\sum_{\ell=0}^{j}
\|{\Ta}^{\ell} \wr^2\|_{L^2_+}^2+\sum_{\ell=0}^{j}
\|{\Ta}^{\ell} (\wr\,\nabla\wr)\|_{L^2_+}^2.
\end{split}
\end{equation}
Yet notice that
\begin{equation*}\label{nonl-est-18}
\begin{split}
\|{\Ta}^{\ell}&\,(\wr\,\nabla\,\wr)\|_{L^2_+}^2
\lesssim  \sum_{k=0}^{(\ell-2)_+}\|{\Ta}^{\ell-k}\,\wr\|_{L^\infty_\v(L^2_\h)}^2
\|{\Ta}^{k}\nabla\wr\|_{L^2_\v(L^\infty_\h)}^2\\
&\,+
\|{\Ta}^{\ell-1}\,\wr\|_{L^\infty_\v(L^2_\h)}^2
\|{\Ta}\nabla\wr\|_{L^2_\v(L^\infty_\h)}^2
+\|{\Ta}^{\ell}\,\wr\|_{L^\infty_\v(L^2_\h)}^2
\|\nabla\wr\|_{L^2_\v(L^\infty_\h)}^2,
\end{split}
\end{equation*}
 which together with \eqref{nonl-est-6a} and \eqref{nonl-est-6b} ensures that
\begin{equation}\label{nonl-est-6c}
\begin{split}
\|{\Ta}^{\ell}&\,(\wr\,\nabla\,\wr)\|_{L^2_+}^2
\lesssim  \varepsilon^{-3}\sum_{k=0}^{(\ell-2)_+}\mathfrak{E}_{\ell-k+1}
\mathfrak{E}_{k+3}+
\varepsilon^{-3}\mathfrak{E}_{\ell+1}
\mathfrak{E}_{4}\lesssim \varepsilon^{-3}\mathfrak{E}_{N}^2.
\end{split}
\end{equation}
By inserting \eqref{nonl-est-19} and \eqref{nonl-est-6c} into \eqref{nonl-est-17} gives rise to
\begin{equation*}\label{nonl-est-20}
\begin{split}
\|{\Ta}^{j}\,\nabla(a^{\varepsilon,m}\wr^2)\|_{L^2_+}^2
\lesssim \varepsilon^{-3}\mathfrak{E}_{N}^2\quad\mbox{for}\ \ j\leq N-1.
\end{split}
\end{equation*}
Exactly along the same line, we achieve
\begin{equation*}\label{nonl-est-12}
\begin{split}
&\|{\Ta}^{N-1}\,\nabla(a^{\varepsilon,m}\wr\wi)\|_{L^2_+}^2+\|{\Ta}^{N-1}\,\nabla(a^{\varepsilon,m}\wi^2)\|_{L^2_+}^2\lesssim \varepsilon^{-6}\,\mathfrak{E}_N^2,\\
&\|{\Ta}^{N-1}\,\nabla (\wr^3, \,\wi^3,\,\wr^2\wr,\,\wr\wi^2)\|_{L^2_+}^2\lesssim\,\,(\varepsilon^{-4}\mathfrak{E}_{N})^2\,\mathfrak{E}_{N}.
\end{split}
\end{equation*}

As a result, it comes out
\begin{equation*}
\begin{split}
&\varepsilon^{-2}\sum_{j=0}^{N-1}
\bigl(\|{\Ta}^j\na\text{Re}Q^\varepsilon({\fw})\|_{L^2_+}^2
+\|{\Ta}^j\na\text{Im}Q^\varepsilon({\fw})\|_{L^2_+}^2\bigr)\lesssim \varepsilon^{-8}\mathfrak{E}_{N}^2+\varepsilon^{-10}\mathfrak{E}_{N}^3.\\
\end{split}
\end{equation*}
Along with \eqref{nonl-est-1c}, we obtain \eqref{nonl-est-15}. This completes the proof of Lemma \ref{S8lem6}. \end{proof}

Now we are in a position to complete the proof of Theorem \ref{thmmain}.

\begin{proof}[Proof of Theorem \ref{thmmain}]

In order to get the second order full  derivatives of  $(\wr, \wi)$, we may make use of the system \eqref{w-phi-eqns-j} for $j =0, 1, 2$.
In fact, according to the $\wr$ equation of \eqref{w-phi-eqns-j}, we get
\begin{equation*}\label{diff-im-eqns-lin-t1}
\begin{split}
\varepsilon^{2}\bigl\| \Delta \left({\Ta}^j\,\wi\right)\bigr\|_{L^2_+}\lesssim & \varepsilon\|\partial_t{\Ta}^j\,\wr\|_{L^2_+}
+\varepsilon\|\mathcal{S}_{u^{\varepsilon, m}}({\Ta}^j\,\wr)\|_{L^2_+}+\|{\Ta}^j\text{Im}(Q^\varepsilon({\fw})\|_{L^2_+}\\
&+ \Bigl\|\Bigl(\frac{\varepsilon^2}{2}\frac{\Delta a^{\varepsilon, m}}{a^{\varepsilon, m}}+\varepsilon^{m+2} r^m_\vf\Bigr) {\Ta}^j\wi\Bigr\|_{L^2_+}+\varepsilon^{m+2} \|{\Ta}^jr^m_\vf\|_{L^2_+}+\|R_{\partial_t, 1,j}\|_{L^2_+}^2.
\end{split}
\end{equation*}
Thanks to \eqref{space-norm-E-0a}, \eqref{zhang1} and \eqref{nonl-est-1c}, we get, by
applying Lemmas \ref{lem-est-singu-1} and \ref{lem-est-conv-2}, that
\beno\begin{split}
\varepsilon^{2}\bigl\| \Delta \left({\Ta}^j\,\wi\right)\bigr\|_{L^2_+}\lesssim & \varepsilon\|\partial_t{\Ta}^j\,\wr\|_{L^2_+}+\e^{-4}\cE_4^2\bigl(1+\e^{-2}\cE_4\bigr)\\
&+\cE_0^{\frac12}\e^{{m}+1}+\sum_{\ell=0}^j\bigl(\|\cT^\ell\wr\|_{L^2_+}
+\e\bigl\|\cT^\ell(\wr,\wi)\bigr\|_{H^1}\bigr),
\end{split}
\eeno
from which and Proposition \ref{S2prop6}, we infer
\begin{equation*}\label{est-sec-order-1a}
\begin{split}
&\sum_{j=0}^2\bigl\|\Delta \left({\Ta}^j\,\wi\right)\bigr\|_{L^\infty_{T_0}(L^2_+)}\lesssim  \mathcal{E}_0^{\frac12} \varepsilon^{m-1}.
\end{split}
\end{equation*}
The same estimate holds for $\Delta \left({\Ta}^j\,\wr\right)$, and then
\begin{equation}\label{est-sec-order-1}
\begin{split}
&\sum_{j=0}^2\bigl\|\Delta {\Ta}^j\,(\wr,\,\wi)\bigr\|_{L^\infty_{T_0}(L^2_+)}\lesssim  \mathcal{E}_0^{\frac12} \varepsilon^{m-1}.
\end{split}
\end{equation}

While it follows from Proposition \ref{S2prop6} that
\beq \label{zhang5}
\sum_{j=0}^{3}\|\cT^j(\wr,\wi)\|_{L^\infty_{T_0}(H^1)}\lesssim \cE_0^{\frac12}\e^m.
\eeq

Let us now turn to estimate $\|\frak{w}\|_{W^{1,\infty}}$ for $\frak{w}$ given by \eqref{expr-Phi-1}. We first deduce from
\eqref{nonl-est-6a} and \eqref{nonl-est-6b} that
\beq\label{zhang6}
\|\frak{w}\|_{L^\infty_+}\leq \|\wr\|_{L^\infty_+}+\|\wi\|_{L^\infty_+}\leq C\e^{-2}\frak{E}_4^{\frac12}\leq C\cE_0^{\frac12}\e^{m-1}.
\eeq

On the other hand, we deduce from \eqref{aniso-sobolev-infty} that
 \begin{equation*}
\begin{split}
\|\nabla\,f\|_{L^\infty_+}^2&=\|\nabla_\h\,f\|_{L^\infty_+}^2+\|\partial_z\,f\|_{L^\infty_+}^2\\
\lesssim &\|\nabla_\h\,f\|_{H^1}
\| (1+\na_\h^2)\,\nabla_\h\,f\|_{H^1}+\|\partial_zf\|_{H^1}
 \|(1+\na_\h^2)\,\partial_zf\|_{H^1},
\end{split}
\end{equation*}
which along with the fact that
 \begin{equation*}
\begin{split}
\|\partial_zf\|_{H^1}&\lesssim \|\nabla\,f\|_{L^2_+}+\|\nabla_\h\,\partial_zf\|_{L^2_+}+\|\partial_z^2f\|_{L^2_+}\\
&\lesssim\|f\|_{H^1}+\|\na_\h\,f\|_{H^1}+\|\D\,f\|_{L^2_+}+\|\na_\h^2\,f\|_{L^2_+}.
\end{split}
\end{equation*}
ensures that
 \begin{equation*}
\begin{split}
&\|\nabla\,f\|_{L^\infty_+}^2\lesssim \|\nabla_\h\,f\|_{H^1}
\|(1+\na_\h^2)\,\nabla_\h\,f\|_{H^1}+\bigl(\|f\|_{H^1}+\|\na_\h\,f\|_{H^1}
+\|\D\,f\|_{L^2_+}+\|\na_\h^2\,f\|_{L^2_+}\bigr)\\
&\qquad\times\bigl(\|(1+\na_\h^2)f\|_{H^1}+\|\na_\h(1+\na_\h^2)\,f\|_{H^1}
+\|\D(1+\na_\h^2)\,f\|_{L^2_+}+\|\na_\h^2(1+\na_\h^2)\,f\|_{L^2_+}\bigr)\\
&\lesssim \sum_{j=0}^3\|\Ta\,f\|_{H^1}^2 +\sum_{j=0}^2\|\D\Ta\,f\|_{L^2_+}^2.
\end{split}
\end{equation*}
Therefore, we obtain from \eqref{est-sec-order-1} and \eqref{zhang5} that for any $t\in[0, T_0]$
\begin{equation*}
\begin{split}
\|\na\frak{w}(t)\|_{L^\infty}&\leq \|\na{\wr}(t)\|_{L^{\infty}}+\|\na{\wi}(t)\|_{L^{\infty}}\\
& \leq C (\sum_{j=0}^3\|\Ta\,(\wr,\,\wi)\|_{H^1} +\sum_{j=0}^2\|\D\Ta\,(\wr,\,\wi)\|_{L^2_+})\\
&
\leq C \mathcal{E}^{\frac{1}{2}}_0\varepsilon^{m-1}.
\end{split}
\end{equation*}
This together with \eqref{zhang6} ensures \eqref{ener-infty-1}. This ends the proof of Theorem \ref{thmmain}.
\end{proof}

\smallskip

%%%%%%%%%%%%%%%%%%%%%%%%%%%%%%%%%%%%%%%%%%%%%%%%%%%%%%%%%%%%
\renewcommand{\theequation}{\thesection.\arabic{equation}}
\setcounter{equation}{0}
%%%%%%%%%%%%%%%%%%%%%%%%%%%%%%%%%%%%%%%%%%%%%%%%%%%%%%%%%%%%

\appendix

%%%%%%%%%%%%%%%%%%%%%%%%%%%%%%%%%%%%%%%%%%%%%%%%%%%%%%%%%%%%
\renewcommand{\theequation}{\thesection.\arabic{equation}}
\setcounter{equation}{0}
%%%%%%%%%%%%%%%%%%%%%%%%%%%%%%%%%%%%%%%%%%%%%%%%%%%%%%%%%%%%

\section{The source  terms $F_k$ in \eqref{phi-order-m} and $G_k$ in \eqref{a-order-m1} }\label{appA-re}

Indeed we observe from \eqref{u-conti} that
\begin{equation}\label{app-u-conti-1}
\begin{split}
F_k\eqdefa&-\partial_tA_{k}-\sum_{\ell=0}^k F_{1,\ell}-\sum_{\ell=2}^{k+1}F_{2,\ell}-\sum_{\ell=1}^{k}F_{3,\ell}
-\sum_{\ell=1}^{k+1}F_{4,\ell}+\sum_{\ell_1+ \ell_2=k-1} F_{5,\ell_1, \ell_2}\\
&-\sum_{\substack{\ell_1+\ell_2+j=k\\
 2\leq j\leq k}}F_{6,\ell_1, \ell_2, j}-\sum_{\substack{\ell_1+\ell_2+j=k+1\\
  2\leq j\leq k+1}}F_{7,\ell_1, \ell_2, j}-\sum_{\substack{\ell_1+\ell_2+j=k+2\\ 2\leq j\leq k+2}} F_{8,\ell_1, \ell_2, j}
\end{split}
\end{equation}
where
\begin{equation*}\label{app-u-conti-1a}
\begin{split}
F_{1,\ell}\eqdefa &\nabla_{\rm h} \Phi_{\ell}\cdot \nabla_{\rm h} A_{k-\ell}+\frac{A_{\ell}}{2}(\Delta_{\h} \Phi_{k-\ell}+\overline{\Delta  \varphi_{k-\ell}})+\overline{\nabla_y\varphi_{\ell}}\cdot \nabla_{\h} A_{k-\ell}+\nabla_{\h} \Phi_{\ell}\cdot \overline{\nabla_{\h} a_{k-\ell}}\\
&+\frac{\overline{a_{\ell}}}{2} \Delta_{\h} \Phi_{k-\ell}+Z
\bigl(\overline{\partial_{z}^{2}\varphi_{\ell}}\partial_Z A_{k-\ell}+\partial_{Z}\Phi_{\ell} \overline{\partial_{z}^{2} a_{k-\ell}}\bigr)+
\frac{Z^{2}}{2}
\frac{\overline{\partial_{z}^2a_{\ell}}}{2}\partial_Z^2 \Phi_{k-\ell},\\
F_{2,\ell}\eqdefa& \,\partial_Z \Phi_{\ell}\partial_Z A_{k+2-\ell},\quad F_{3,\ell}\eqdefa\frac{1}{2}(A_{\ell}+\overline{a_{\ell}})\partial_Z^2 \Phi_{k+2-\ell},
\end{split}
\end{equation*}
and
\begin{equation*}\label{app-u-conti-1b}
\begin{split}
F_{4,\ell}\eqdefa&\overline{\partial_{z}\varphi_{\ell}}\partial_Z A_{k+1-\ell}+\partial_{Z}\Phi_{\ell} \overline{\partial_{z} a_{k+1-\ell}}+Z
\frac{\overline{\partial_{z} a_{k+1-\ell}}}{2}\partial_Z^2 \Phi_{\ell},\\
F_{5,\ell_1, \ell_2}\eqdefa& Z
\bigl(\overline{\nabla_y\partial_{z}\varphi_{\ell_1}}\cdot \nabla_{\h} A_{\ell_2}+\nabla_{\h} \Phi_{\ell_1}\cdot \overline{\nabla_{\h} \partial_{z} a_{\ell_2}}+\frac{\overline{\partial_{z}a_{\ell_1}}}{2} \Delta_{\h} \Phi_{\ell_2}+\frac{A_{\ell_1}}{2}\overline{\Delta \partial_{z} \varphi_{\ell_2}}\bigr)\\
&\qquad\qquad\qquad\qquad\qquad\qquad\qquad\qquad+\frac{Z^{2}}{2}
\bigl(\overline{\partial_{z}^{3}\varphi_{\ell_1}}\partial_Z A_{\ell_2}+\partial_{Z}\Phi_{\ell_1} \overline{\partial_{z}^{3} a_{\ell_2}}\bigr),
\end{split}
\end{equation*}
and
\begin{equation*}\label{app-u-conti-1c}
\begin{split}
F_{6,\ell_1, \ell_2, j}\eqdefa&\frac{Z^{j}}{j!}
\Bigl(\overline{\nabla_y\partial_{z}^{j}\varphi_{\ell_1}}\cdot \nabla_{\h} A_{\ell_2}+\nabla_{\h} \Phi_{\ell_1}\cdot \overline{\nabla_{\h} \partial_{z}^{j} a_{\ell_2}} +\frac{\overline{\partial_{z}^{j}a_{\ell_1}}}{2} \Delta_{\h} \Phi_{\ell_2}+\frac{A_{\ell_1}}{2}\overline{\Delta \partial_{z}^{j} \varphi_{\ell_2}}\Bigr),\\
F_{7,\ell_1, \ell_2, j}\eqdefa&\frac{Z^{j}}{j!}
\bigl(\overline{\partial_{z}^{j+1}\varphi_{\ell_1}}\partial_Z A_{\ell_2}+\partial_{Z}\Phi_{\ell_1} \overline{\partial_{z}^{j+1} a_{\ell_2}}\bigr),\quad F_{8,\ell_1, \ell_2, j}\eqdefa\frac{Z^{j}}{j!}
\frac{\overline{\partial_{z}^{j}a_{\ell_1}}}{2}\partial_Z^2 \Phi_{\ell_2}.
\end{split}
\end{equation*}

Whereas we observe from \eqref{u-bern} that
\begin{equation}\label{app-u-bern-2}
\begin{split}
&G_k\eqdefa \sum_{\substack{\ell_1+\ell_2+j=k+1\\
 0\leq \ell_1\leq k}} \frac{Z^{j}}{j!}\bigl(\overline{\partial_z^ja_{\ell_1}}\partial_t \Phi_{\ell_2}+A_{\ell_1}\overline{\partial_t\partial_z^j \varphi_{\ell_2}}\bigr)
 +\sum_{\ell=0}^{k}A_{\ell}\partial_t \Phi_{k+1-\ell}\\
 &+\sum_{\ell_1+\ell_2+\ell_3+j_1+j_2=k+1} \frac{Z^{j_1+j_2}}{j_1!j_2!}\Bigl(\overline{\partial_{z}^{j_1}a_{\ell_1}}  \overline{\nabla_{\h} \partial_{z}^{j_2}\varphi_{\ell_2}} \cdot \nabla_{\h} \Phi_{\ell_3}\Bigr)\\
&+\sum_{\substack{\ell_1+\ell_2+\ell_3+j_1+j_2=k+1\\
0\leq \ell_1\leq k}}\frac{Z^{j_1+j_2}}{j_1!j_2!}A_{\ell_1}\Bigl(\frac{1}{2} \overline{\nabla \partial_{z}^{j_1}\varphi_{\ell_2}} \cdot \overline{\nabla\partial_{z}^{j_2} \varphi_{\ell_3}}+3\overline{\partial_{z}^{j_1}a_{\ell_2}}\overline{\partial_{z}^{j_2}a_{\ell_3}}\Bigr)
\\
&+\sum_{\substack{\ell_1+\ell_2+\ell_3+j_1+j_2=k+2\\
 1\leq \ell_3\leq k+1}}\frac{Z^{j_1+j_2}}{j_1!j_2!}\bigl(\overline{\partial_{z}^{j_1}a_{\ell_1}}
\overline{\partial_{z}^{j_2+1}\varphi_{\ell_2}}\partial_Z\Phi_{\ell_3}\bigr)\\
&+\sum_{\substack{\ell_1+\ell_2+\ell_3+j=k+1\\
0\leq \ell_1\leq k}} \frac{Z^{j}}{j!}
A_{\ell_1}(\overline{\nabla_{\h}\partial_{z}^{j}\varphi_{\ell_2}} \cdot \nabla_{\h}\Phi_{\ell_3})
+\sum_{\ell_1+\ell_2+\ell_3+j=k+1} \frac{Z^{j}}{j!}
\frac{\overline{\partial_{z}^{j}a_{\ell_1}}}{2}\nabla_{\h}\Phi_{\ell_2} \cdot \nabla_{h}\Phi_{\ell_3}\\
&+\sum_{\substack{\ell_1+\ell_2+\ell_3+j=k+1\\
0\leq \ell_1, \ell_2\leq k}} \frac{Z^{j}}{j!}
\,3\overline{\partial_{z}^{j}a_{\ell_1}}A_{\ell_2}A_{\ell_3}
+\sum_{\substack{\ell_1+\ell_2+\ell_3+j=k+2\\
0\leq \ell_1\leq k, \, 1\leq \ell_3\leq k+1}} \frac{Z^{j}}{j!}
A_{\ell_1}\overline{\partial_{z}^{j+1}\varphi_{\ell_2}}\partial_Z\Phi_{\ell_3}\\
&+\sum_{\substack{\ell_1+\ell_2+\ell_3+j=k+3 \\ 1\leq \ell_2+\ell_3\leq k+1}} \frac{Z^{j}}{j!}
\frac{\overline{\partial_{z}^{j}a_{\ell_1}}}{2}
\partial_{Z}\Phi_{\ell_2}\partial_Z\Phi_{\ell_3}
-\frac{1}{2} \Delta_{\h} A_{k-1}.
\end{split}
\end{equation}

\setcounter{equation}{0}
\section{The proof of Lemma \ref{S3lem1} }\label{appA}

\begin{proof}[Proof of Lemma \ref{S3lem1}] We split the proof of \eqref{S3eq6} into the following cases:

\noindent$\bullet$\underline{
When $k<s-\frac32.$} In this case, $s-k >\frac32,$ so that $H^{s-k}(\R^3_+)$  is an algebra. As a result,
it comes out
\beno
\begin{split}
\|\p_t^k(fg)(t)\|_{H^{s-k}}\lesssim &\sum_{\ell=0}^k\|\p_t^\ell f(t)\p_t^{k-\ell}g(t)\|_{H^{s-k}}\\
\lesssim &\sum_{\ell=0}^k\|\p_t^\ell f(t)\|_{H^{s-k}}\|\p_t^{k-\ell}g(t)\|_{H^{s-k}}\\
\lesssim &\sum_{\ell=0}^k\|f(t)\|_{W^{s+\ell-k}}\|g(t)\|_{W^{s-\ell}}\lesssim \|f(t)\|_{W^s}\|g(t)\|_{W^s}.
\end{split}
\eeno

\noindent$\bullet$\underline{
When $k=[s]-1\geq s-\frac32.$} We first observe that
\beq \label{S3eq7a}
\|fg\|_{H^\tau}\lesssim \|f\|_{H^\tau}\|g\|_{H^2} \quad\forall \tau\in [0,2).
\eeq
The proof of the above inequality can be obtained by first extend the function to the whole
space and then using the law of product in the classical Sobolev space (see \cite{bcdbookk}). We skip the details here.

When $[s]\geq 3,$ we get, by applying \eqref{S3eq7a}, that
\beno
\begin{split}
\|\p_t^{[s]-1}(fg)(t)\|_{H^{s+1-[s]}}\lesssim & \|\p_t^{[s]-1}f(t)g(t)\|_{H^{s+1-[s]}}+\|f(t)\p_t^{[s]-1}g(t)\|_{H^{s+1-[s]}}\\
&+\sum_{\ell=1}^{[s]-2}\|\p_t^\ell f(t)\p_t^{[s]-1-\ell}g(t)\|_{H^{s+1-[s]}}\\
\lesssim & \|\p_t^{[s]-1}f(t)\|_{H^{s+1-[s]}}\|g(t)\|_{H^2}+\|f(t)\|_{H^2}\|\p_t^{[s]-1}g(t)\|_{H^{s+1-[s]}}\\
&+\sum_{\ell=1}^{[s]-2}\|\p_t^\ell f(t)\|_{H^2}\|\p_t^{[s]-1-\ell}g(t)\|_{H^{s+1-[s]}}\\
\lesssim & \|f(t)\|_{W^s}\|g(t)\|_{H^2}+\|f(t)\|_{H^2}\|g(t)\|_{W^s}+\sum_{\ell=1}^{[s]-2}\|f(t)\|_{W^{\ell+2}}\|g(t)\|_{W^{s-\ell}}\\
\lesssim & \|f(t)\|_{W^s}\|g(t)\|_{W^s}.
\end{split}
\eeno
The case for $[s]=2$ can be proved along the same line.

\noindent$\bullet$\underline{
When $k=[s].$} We first recall the following law of product in Sobolev space from \cite{bcdbookk}
\beq \label{S3eq7b}
\|fg\|_{s_1+s_2-\frac32}\leq C\|f\|_{H^{s_1}}\|g\|_{H^{s_2}}\quad\mbox{if}\quad s_1, s_2\in ]0,3/2[.
\eeq

In the case when $s-[s]\in \bigl]0,\frac12\bigr[,$
we get, by applying \eqref{S3eq7b}, that
\beno
\|fg\|_{H^{s-[s]}}\lesssim \|fg\|_{H^{\frac12}}\lesssim \|f\|_{H^1}\|g\|_{H^1}.
\eeno
Then applying the above inequality and \eqref{S3eq7a} yields
\beno
\begin{split}
\|\p_t^{[s]}(fg)(t)\|_{H^{s-[s]}}\lesssim & \|\p_t^{[s]}f(t)g(t)\|_{H^{s-[s]}}+\|f(t)\p_t^{[s]}g(t)\|_{H^{s-[s]}}+\sum_{\ell=1}^{[s]-1}\|\p_t^\ell f(t)\p_t^{[s]-\ell}g(t)\|_{H^{s-[s]}}\\
\lesssim & \|\p_t^{[s]}f(t)\|_{H^{s-[s]}}\|g(t)\|_{H^2}+\|f(t)\|_{H^2}\|\p_t^{[s]}g(t)\|_{H^{s-[s]}}\\
&+\sum_{\ell=1}^{[s]-1}\|\p_t^\ell f(t)\|_{H^{1}}\|\p_t^{[s]-\ell}g(t)\|_{H^1}\\
\lesssim & \|f(t)\|_{W^s}\|g(t)\|_{H^2}+\|f(t)\|_{H^2}\|g(t)\|_{W^s}+\sum_{\ell=1}^{[s]-1}\|f(t)\|_{W^{\ell+1}}\|g(t)\|_{W^{s-\ell+1}}\\
\lesssim & \|f(t)\|_{W^s}\|g(t)\|_{W^s}.
\end{split}
\eeno
In the case when $s-[s]\in \bigl]\frac12,1\bigr[,$
we get, by applying \eqref{S3eq7b}, that
\beno
\begin{split}
\|fg\|_{H^{s-[s]}}\lesssim&\|fg\|_{H^{2(s-[s])-\frac12}}
 \lesssim \|f\|_{H^{s-[s]+\frac12}}\|g\|_{H^{s-[s]+\frac12}}.
 \end{split}
\eeno
Applying the above inequality gives rise to
\beno
\begin{split}
\sum_{\ell=1}^{[s]-1}\|\p_t^\ell f(t)\p_t^{[s]-\ell}g(t)\|_{H^{s-[s]}}
\lesssim & \sum_{\ell=1}^{[s]-1}\|\p_t^\ell f(t)\|_{H^{s-[s]+\frac12}}\|\p_t^{[s]-\ell}g(t)\|_{H^{s-[s]+\frac12}}\\
\lesssim & \sum_{\ell=1}^{[s]-1}\|f(t)\|_{W^{s-[s]+\ell+\frac12}}\|g(t)\|_{W^{s-\ell+\frac12}}\\
\lesssim & \|f(t)\|_{W^s}\|g(t)\|_{W^s}.
\end{split}
\eeno

By summing up the above estimates, we obtain \eqref{S3eq6}. This completes the proof of Lemma \ref{S3lem1}.
\end{proof}

\bigbreak \noindent {\bf Acknowledgments.}
G. Gui is supported in part by the National Natural Science Foundation of China under the Grant 11571279. P. Zhang is partially supported by the National Natural Science Foundation of China under Grants 11688101 and 11371347, Morningside Center of Mathematics of The Chinese Academy of Sciences and innovation grant from National Center for Mathematics and Interdisciplinary Sciences.

\end{document}